\documentclass[12pt]{article}
\usepackage[active]{srcltx}
\usepackage{amscd,amsfonts,amssymb,amscd,amsmath,latexsym}
\usepackage[hypertex]{hyperref}
\usepackage[all]{xy}
\makeatletter
\renewcommand{\subsubsection}{\@startsection
{subsubsection}
{1}
{0mm}
{0mm}
{0mm}
{\normalfont\normalsize\itshape}}
\makeatother

\usepackage{mathrsfs}
\usepackage{color}

\title{Differential orbifold $K$-Theory}
\author{Ulrich Bunke\thanks{NWF I - Mathematik,
Universit{\"a}t Regensburg,
93040 Regensburg,
GERMANY, ulrich.bunke@mathematik.uni-regensburg.de} \and Thomas Schick \thanks{Mathematisches Institut, Universit{\"a}t G{\"o}ttingen,
Bunsenstr. 3-5, 37073 G{\"o}ttingen, GERMANY, schick@uni-math.gwdg.de;
supported by Courant Research Center ``Higher order structures in
   mathematics'' via the German Initiative of Excellence} }
 
\newtheorem{theorem}{Theorem}[section] 
\newtheorem{prop}[theorem]{Proposition}
\newtheorem{lem}[theorem]{Lemma}
\newtheorem{ddd}[theorem]{Definition}
\newtheorem{kor}[theorem]{Corollary}

\newtheorem{remark}[theorem]{Remark}

\newcommand{\Fred}{{\tt Fred}}

\newcommand{\hocolim}{{\tt hocolim\:}}
\newcommand{\Fr}{{\tt Fr}}

\newcommand{\bM}{{\bf M}}

\renewcommand{\lim}{{\tt lim\:}}

\renewcommand{\P}{{\mathbb{P}}}

\newcommand{\Z}{\mathbb{Z}}

\newcommand{\diag}{{\tt diag}}
\newcommand{\proof}{{\it Proof.$\:\:\:\:$}}

\newcommand{\R}{\mathbb{R}}
\newcommand{\ori}{{\tt or}}

\newcommand{\Cliff}{{\tt Cliff}}

\renewcommand{\det}{{\tt det}}

\newcommand{\K}{\mathbb{K}}

\newcommand{\T}{{\mathbb{T}}}

\newcommand{\C}{\mathbb{C}}

\newcommand{\Aut}{{\tt Aut}}

\newcommand{\Tr}{{\tt Tr}}

\newcommand{\cE}{\mathcal{E}}

\newcommand{\cV}{\mathcal{V}}
\newcommand{\cY}{\mathcal{Y}}
\newcommand{\cW}{\mathcal{W}}

\newcommand{\cK}{\mathcal{K}}
\newcommand{\cA}{\mathcal{A}}
\newcommand{\cU}{\mathcal{U}}
\newcommand{\Hom}{{\tt Hom}}

\newcommand{\cO}{\mathcal{O}}
\newcommand{\End}{{\tt End}}

\newcommand{\im}{{\tt im}}

\newcommand{\cF}{\mathcal{F}}

\newcommand{\Imm}{{\tt Im}}

\newcommand{\ee}{{\tt e}}

\newcommand{\tr}{{\tt tr}}

\newcommand{\coker}{{\tt coker}}
\newcommand{\id}{{\tt id}}

\newcommand{\nat}{\mathbb{N}}

\def\imath{{i}}

\def\hB{\hspace*{\fill}$\Box$ \newline\noindent}

\newcommand{\ind}{{\tt index}}

\newcommand{\restr}{{\tt Res}}

\def\hB{\hspace*{\fill}$\Box$ \\[0.5cm]\noindent}

 \newcommand{\cG}{\mathcal{G}}

 \newcommand{\cX}{\mathcal{X}}

\newcommand{\bW}{\mathbf{W}}
\newcommand{\pr}{{\tt pr}}

\newcommand{\ch}{{\mathbf{ch}}}

\newcommand{\ev}{{\tt ev}}
\newcommand{\bV}{\mathbf{V}}

\newcommand{\hA}{\hat{\mathbf{A}}}

 \newcommand{\Ab}{{\mathrm{Ab}}}

\begin{document}

\maketitle
\begin{abstract}
We construct differential   $K$-theory of representable smooth orbifolds as
a ring valued functor with the usual properties of a differential extension of a
cohomology theory. For proper submersions (with smooth fibres) we construct a
push-forward map in differential  orbifold $K$-theory. Finally, we construct a
non-degenerate intersection pairing with values in $\C/\Z$ for the subclass of
smooth orbifolds which can be written as global quotients by a finite group
action.  We construct a real subfunctor of our theory, where the
pairing restricts to a non-degenerate $\R/\Z$-valued pairing.

\end{abstract}

\tableofcontents

\section{Introduction}

In this paper we give the construction of a model of differential $K$-theory
for orbifolds. It generalizes the model for smooth manifolds
 \cite{bunke-2007}. Major features are the constructions of the cup-product and  the push-forward with all desired properties,
and the localization isomorphism.

Our construction includes a model of equivariant differential $K$-theory for Lie group actions
with finite stabilizers. {However, a  construction in the realm of orbifolds  not
only covers more general objects, but is stronger also for group actions.
The additional information 
is the independence of the choice of presentations}. In equivariant
terms, this means that {differential $K$-theory} has induction and descend isomorphisms.

One of the motivations for the consideration of differential $K$-theory
came from mathematical physics, in particular from type-II superstring theory. Here it 
was used as a host of certain fields with
differential form field strength, see e.g. \cite{freed-2007-322}, \cite{Witten:1998cd},
\cite{Minasian:1997mm}. 
For the   theory on orbifolds one needs
the corresponding 
generalization of  differential $K$-theory \cite{szabo-2007}.  To serve this goal is one
 of the  motivations of this paper.
As explained in  \cite{szabo-2007}, the intersection pairing in
differential $K$-theory on compact $K$-oriented orbifolds is an
important aspect of the theory. In the present paper we construct a
non-degenerated $\C/\Z$-valued paring.  
Note that because of the nature of the equivariant Chern character orbifold differential $K$-theory naturally works with complex valued forms. We will show that it admits a  
real subfunctor, and the pairing restricts to a
non-degenerated  $\R/\Z$-valued pairing on this subfunctors.

In this paper, we use the terminology ``differential $K$-theory''
throughout. In previous publications like \cite{bunke-2007}, we used the
synonym ``smooth 
$K$-theory''. Dan Freed convinced us that the analogy with
differential forms implies that the first expression is more
appropriate. \textcolor{black}{\cite{Bunke-2011} is a nice survey on the general
  theory of differential K-theory; we try to cover as much of its aspects in
  the equivariant or orbifold situation as possible.}
 
We now describe the contents of the paper. In Section \ref{tztewqze} we
construct the model of differential $K$-theory and verify its basic properties. We
first review the relevant orbifold and stack notation. Then we define differential $K$-theory
for orbifolds by cycles and relations as a direct generalization of
the construction for  manifolds \cite{bunke-2007}, with some extra care
in the local analysis. In the sequel, we will refer to the case of
smooth manifolds as the ``non-singular case'' or the ``smooth case''. 

Section \ref{idwiqdqwwd} is devoted to the cup-product and the push-forward.
These are again direct generalizations of the corresponding
constructions in \cite{bunke-2007}.  In Subsection
\ref{localization} we prove the localization theorem in differential
$K$-theory for global quotients by finite group actions.

In Section \ref{uwhdiodhwqidwqdwqdqwd} we prove two results. The first is
Theorem \ref{uiudqdqwdwq} which identifies the flat part of differential
$K$-theory  as $K$-theory with coefficients in
$\C/\Z$. The result is a generalization of
\cite[Proposition 2.25]{bunke-2007}, though the proof requires \textcolor{black}{new fundamental} ideas.  Finally we show in 
Theorem \ref{trezreer}  that the intersection pairing is
  non-degenerate.

The final Section \ref{sec5} contains some interesting explicit
calculations  and important bordism formulas which are
crucial for any calculations.

In \textcolor{black}{2009}, the preprint \cite{ortiz} appeared. It gives another construction of differential
equivariant $K$-theory for finite group actions along the lines of
\cite{MR2192936}\footnote{\textcolor{black}{At the time of writing} the paper \cite{ortiz} is pending for revisions.}.
It defines a push-forward to a point. 
 The main difference between the two approaches is that our constructions are mainly analytical,
  whereas his are mainly homotopy theoretic.

 Ortitz there raises the interesting question
\cite[Conjecture 6.1]{ortiz} of identifying this push-forward in
analytic terms. Note that in our model, in view of the geometric
construction of the push-forward and the analytic nature of the
relations,  the conjectured relation 
 is essentially a tautology. See \cite[Corollary 5.5]{bunke-2007} for a more
 general statement in the non-equivariant case.  \cite[Conjecture 6.1]{ortiz}
 would be an immediate consequence of a theorem stating that any two models of
 equivariant differential $K$-theory for finite group actions are canonically
 isomorphic  (see \cite{bunke-2009} for the non-equivariant version) in a
 \textcolor{black}{way compatible with} integration. 
 It seems to be plausible that the method of \cite{bunke-2009} extends to the equivariant case though we have not checked the details.

\emph{Acknowledgement: A great part of the material of the present paper has
  been worked out around 2003. Motivated by \cite{szabo-2007} and fruitful
  personal discussions with \textcolor{black}{Richard} Szabo and
  \textcolor{black}{Alessandro} Valentino 
we transferred the theory to the case of orbifolds and worked out the details of the intersection pairing.
}

\section{Definition of differential K-theory via cycles and relations}\label{tztewqze}

\subsection{Equivariant forms and orbifold $K$-theory}

\newcommand{\Mf}{{\tt Mf}}
\newcommand{\Sh}{{\tt Sh}}
\newcommand{\Stack}{{\tt Stack}}

\subsubsection{}\label{stack1}

{In the present paper we use the language of stacks in order to talk about orbifolds and maps between them. This language is by now well-developed and we  refer to  \cite{bss}, \cite{dm}, \cite{vistoli} or \cite{heinloth} for details.  For the sake of readers with less experience with stacks we will recall some basic notions and constructions.}

{
 We consider the category $\Mf$ of smooth manifolds. By the  Yoneda embedding
$\Mf\hookrightarrow \Sh(\Mf)$ manifolds can be considered as sheafs of sets on $\Mf$ equipped with the usual Grothendieck topology 
given by open coverings. \textcolor{black}{Because sets are special kinds of
groupoids, namely those which have only identity morphisms,} the category of
sheaves of sets embeds 
in the two-category of sheaves \textcolor{black}{of groupoids on $\Mf$ denoted $\Stack(\Mf)$ whose
objects are called stacks}. By this embedding $\Sh(\Mf)\hookrightarrow
\Stack(\Mf)$, a manifold $M$ can be considered as a stack which associates to
each test manifold $T\in \Mf$ the set (considered as a groupoid) of smooth
maps from $T$ to $M$, i.e.~we have $M(T)=C^{\infty}(T,M)$. More generally, if
$G$ is a Lie group acting on $M$, then we can consider the quotient stack 
$[M/G]$ which associates to each test manifold $T$ the groupoid $[M/G](T)$ of pairs $(P\to T,\phi)$ of $G$-principal
bundles $P\to T$ and $G$-equivariant maps $\phi\colon P\to M$. If $G$ acts freely and properly with quotient manifold $M/G$, then we have a natural isomorphism of stacks 
$[M/G]\cong M/G$. 
 If $G\subset H$ is an inclusion of Lie groups, then we have a natural isomorphisms of stacks $[M/G]\cong [M\times_{G}H/H]$. By functoriality, a definition of differential $K$-theory for stacks takes these
 isomorphisms into account automatically. This is one of our motivations to prefer the stack language.
 }

\subsubsection{}\label{stack22}

The groupoid of maps $\Hom_{\Stack(\Mf)}(X,Y)$ between two stacks $X,Y\in
\Stack(\Mf)$  is \textcolor{black}{by definition} just the groupoid of maps between sheaves of groupoids on $\Mf$. Its objects are called morphisms or maps of stacks, and its morphisms are called two-morphisms between morphisms.
It is important to understand that the Yoneda embedding gives the equivalence of groupoids
$X(T)\cong \Hom_{\Stack(\Mf)}(T,X)$.

In stacks we can form arbitrary two-categorial fibre products. 
A map between stacks $f\colon X\to Y$ is called representable if the stack 
$T\times_{Y}X$ is isomorphic to a manifold for every manifold $T$ and map
$T\to Y$. 
Many properties of maps between smooth manifolds are preserved by
pull-backs. This includes the conditions of being \textcolor{black}{\textit{a locally trivial
  fibre bundle, open, closed, proper, submersion, or surjective}.}
These properties can be defined for representable maps by requiring them for all the induced maps of manifolds
$T\times_{Y}X\to T$.
For example,  a locally trivial fibre bundle is a representable map $f\colon X\to Y$ such that the induced maps between manifolds
$T\times_{Y}X\to T$ are locally trivial fibre bundles in the ordinary sense. In this case the fibres of $f$ are smooth manifolds, and all stackyness of $X$ comes from the base $Y$.

\newcommand{\Vect}{{\tt Vect}}\newcommand{\univE}{{\tt E}}

{
For a map $f$, being a vector bundle  is an additional structure. A vector bundle structure on a map between stacks can be given in two equivalent ways. One way is to use  classifying stacks.
There exists a stack $\Vect(n,\R)\in \Stack(\Mf)$ whose evaluation on the test manifold $T\in \Mf$ is the groupoid of $n$-dimensional  real vector bundles $V\to T$ and isomorphisms. Then an $n$-dimensional real  vector bundle on a stack $Y$ is, by definition, a map of stacks $Y\to \Vect(n,\R)$.
In order to describe the underlying bundle  we consider the universal vector bundle
$\univE(n,\R)\to \Vect(n,\R)$. The evaluation $\univE(n,\R)(T)$ is the groupoid of pairs $(V\to T,v)$
of \textcolor{black}{an $n$-dimensional} vector bundle on $T$ and a section $v\in C^{\infty}(T,V)$. The map $\univE(n,\R)\to \Vect(n,\R)$ forgets the section.
It is representable and a locally trivial fibre bundle since for every map
$(g\colon T\to \Vect(n,\R))\in \Vect(n,\R)(T)$  the pull-back
$T\times_{\Vect(n,\R)} \univE(n,\R)$ is equivalent to the manifold given by
the total space of the vector bundle classified by $g$.  We can now say that a
map $f\colon X\to Y$ between stacks is
an $n$-dimensional real vector bundle if it \textcolor{black}{comes with a
  (class of) morphisms and two-morphisms making
the right  square of diagram \eqref{above} two-cartesian }
\begin{equation}\label{above}\xymatrix{T\times_{Y}X\ar@{..>}[r]\ar[d]&X\ar[r]\ar[d]^{f}&\univE(n,\R)\ar[d]\\ T\ar@{::>}[ur]\ar@{..>}[r]&Y\ar[r]\ar@{:>}[ur]&\Vect(n,\R)}\ .\end{equation}
Note that $f$ is necessarily representable and a locally trivial fibre
bundle.}

\textcolor{black}{
The other, equivalent, way to  define the structure of an $n$-dimensional real vector bundle
on the map $f\colon X\to Y$ is as a family of $n$-dimensional real vector bundle structures on the family of maps
$(T\times_{Y}X\to T)_{T\to Y}$ which is compatible with pull-backs along morphisms of manifolds over $Y$, i.e. for 
pairs $(f,\phi)$
of a smooth map $f$ and a two-morphism $\phi$
$$\xymatrix{T\ar[rr]^{f}\ar[dr]&\ar@{:>}[d]^{\phi}&T^{\prime}\ar[dl]\\&Y&}\ .$$
Indeed, the datum of such a family is the same as a map of stacks $Y\to \Vect(n,\R)$. On the other hand, given
this map, we get the compatible family of vector bundles by forming the left cartesian squares in the diagram (\ref{above}).
}

\textcolor{black}{The same philosophy allows to define additional differential-geometric structures like
fibrewise metrics or connections. Let us explain this in detail for  vertical Riemannian metrics.
}

\newcommand{\BDiff}{{\tt BDiff}}
We consider the stack $\BDiff$ of locally trivial fibre bundles whose evaluation
on a test manifold $T$ is the groupoid of locally trivial \textcolor{black}{smooth} fibre bundles
$F\to T$ and bundle isomorphisms. As in the case of vector bundles it carries a universal bundle $\univE\to \BDiff$ such that $\univE(T)$ is the groupoid of
pairs $(F\to T,s)$ of a fibre bundle  and a section $s\in C^{\infty}(T,F)$, and the bundle projection forgets the section.
We can now form the stack $\BDiff(g^{T^v})$ whose evaluation on $T$
is the groupoid of pairs $(\pi\colon F\to T,g^{T^{v}\pi})$
\textcolor{black}{consisting} of a locally trivial fibre bundle and a vertical
Riemannian metric, and whose morphisms are isometric bundle isomorphisms. We
again have a forgetful map $\BDiff(g^{T^v})\to \BDiff$ and define  
$\univE(g^{T^{v}}):=\BDiff(g^{T^v})\times_{ \BDiff}\univE$.

A map $f\colon X\to Y$ is a locally trivial fibre bundle   if it fits into a two-cartesian diagram
$$\xymatrix{X\ar[d]^{f}\ar[r]&\univE\ar[d]\\Y\ar@{:>}[ur]\ar[r]&\BDiff}\ .$$
A vertical Riemannian metric on $f$ is then a refinement
 to
$$\xymatrix{&&\\X\ar[r]\ar[d]^{f}\ar@/^1.5cm/[rr]&\univE(g^{T^{v}})\ar@{:>}[u]\ar[r]\ar[d]&\univE\ar[d]\\Y\ar[r]\ar@{:>}[ur]\ar@/_1.5cm/[rr]&\BDiff(g^{T^{v}})\ar@{:>}[ur]\ar[r]&\BDiff\\&\ar@{:>}[u]&}\ .$$ 
 Equivalently, a vertical Riemannian metric on $f\colon X\to Y$ can be understood as a  collection of
vectical Riemannian metrics on the bundles 
$T\times_{Y}X\to T$ for all maps $T\to Y$ from smooth manifolds $T$ which is
compatible for pull-backs along maps of test manifolds $(f,\phi)\colon T^\prime\to T$ over $Y$.

A similar idea works for horizontal distributions using the stack 
$\BDiff(T^{h})$ which classifies bundles with horizontal distributions.
For connections and metrics on a vector bundle we work with the corresponding stacks
of vector bundles with connections, metrics or both. See also \ref{fuwefiu} where we apply these ideas to principal bundles and connections.


\subsubsection{}\label{stack2}

A map $A\to X$ from a manifold $A$ to a stack $X$ is called an atlas if it
is representable, surjective and a submersion. A stack is called smooth if it admits an atlas.
For example, the quotient stack $[M/G]$ defined in \ref{stack1} is smooth since the map
$M\to [M/G]$ is an atlas. As a counter example, the stack $\BDiff$ is not smooth.   In general, every smooth stack $X$ is isomorphic to the quotient stack of the action of a groupoid. Indeed, given an atlas $A\to X$, we can form the groupoid
 $\cA:=(A\times_XA\rightrightarrows A)$. This groupoid acts on $A$, and there is a natural
 isomorphism $X\cong [A/\cA]$, where $[A/\cA]$ denotes the quotient stack of the action of $\cA$ on $A$ defined by an extension of the notion of a quotient stack for a group action explained in \ref{stack1}, see \cite{heinloth} for details. 
 In the example above we get the action groupoid
 $M\times G\rightrightarrows M$. 
 The subcategory of smooth stacks can be obtained as a localization of the category of
 groupoids in manifolds by formally inverting a certain class of morphisms.
 There is the option to define differential $K$-theory on the level of groupoids in manifolds
 and to show that it descends to smooth stacks by verifying that the inverted morphisms between groupoids induce isomorphisms in differential $K$-theory. In the present paper we prefer to work with the stacks directly.\footnote{{The difference between these two options resembles the situation in 
 differential geometry, where objects can be defined in charts or globally. The first choice requires frequent  verifications that constructions are independent of the choice of coordinates.}} This choice of language has the advantage that for many definitions (e.g. of a vector bundle or geometric family) and for  many arguments we can just use the same words and symbols as in the non-singular case since their meaning and properties naturally \textcolor{black}{extend} to the case of stacks.

\subsubsection{}

In the present paper we consider differential $K$-theory for orbifolds.
By definition an orbifold is a stack $X$ in smooth manifolds  which admits an orbifold atlas
$A\to X$. An orbifold atlas is an atlas $A\to X$  
 such that the groupoid $A\times_XA\rightrightarrows A$ is {proper and}  \'etale.
 {Recall that a groupoid $A^{1}\rightrightarrows A^{0}$ is called proper if the map
 $(s,t)\colon A^{1}\to A^{0}\times A^{0}$ is proper (preimages of compacts are compact), and it is called \'etale if the range and target maps
 $s,t\colon A^{1}\to A^{0}$ are local diffeomorphisms. In other words,
 orbifolds are stacks which are equivalent to quotient stacks of actions of
 smooth proper \'etale groupoids on a smooth manifolds. A description of
 orbifolds in terms of groupoids has been given in \cite{pronk-2007} or
 \cite{md}.
 }

In the older literature an orbifold is often defined as a topological space together with a compatible collection of 
orbifold charts. In the language of stacks this space would be referred to as  the coarse moduli space.
In this picture the obvious notion of a map between orbifolds would be a map of the coarse moduli spaces which has smooth representatives in the charts. In general this notion is strictly larger
than the notion of  a morphism of orbifolds defined here as a map of stacks. 
Our morphisms of orbifolds are called strong or good maps in \cite{MR1993337}.

\subsubsection{}
A major source of orbifolds are actions of discrete groups on smooth manifolds.
Let $G$ be a discrete group which acts on a smooth manifold $M$. The action
$\mu\colon M\times G\to M$ is called proper if the map $(\id_M,\,u)\colon
M\times G\to M\times M$ is proper.  
If the action is proper, then the quotient stack $[M/G]$ is an orbifold.
The map $M\to [M/G]$ is an orbifold atlas. The associated groupoid is the action groupoid
$M\times G\rightrightarrows M$.
\begin{ddd}
An orbifold of the form $[M/G]$ for a proper action of a discrete group on a smooth manifold
is called good.
\end{ddd}

Another source of examples arises from actions of compact Lie groups $G$ on smooth manifolds $M$ with finite stabilizers. In this case the quotient stack $[M/G]$ is a smooth stack with an atlas
$M\to [M/G]$, but this atlas is not an orbifold atlas since the groupoid
$M\times G\rightrightarrows M$ is not \'etale. In order to find an orbifold
atlas we choose for every point 
$m\in M$ a transversal slice $T_m\subset M$ such that $T_m\times_{G_m}G \to M$
is a tubular 
neighbourhood of the orbit of $m$, where $G_m\subseteq G$ is the finite stabilizer of $m$.
Then the composition $\bigsqcup_{m\in M} T_m\to M\to [M/G]$ is an orbifold atlas.
\begin{ddd}
An orbifold of the form $[M/G]$ for an action of a compact Lie group $G$ with finite stabilizers
on a smooth manifold $M$ is called presentable. A presentable orbifold is called compact,  
if the manifold $M$ in its presentation can be chosen compact.
\end{ddd}

 Note that, by definition, a presentation $[M/G]$ of an orbifold involves a \emph{compact}
group $G$.

Let $X$ be an orbifold with orbifold atlas $A\to X$. It gives rise to the \'etale groupoid
$\cA\colon A\times_XA\rightrightarrows A$. The frame bundle of a manifold can be defined by a construction which is  functorial under local diffeomorphisms. Since the groupoid  $\cA$ is  \'etale the frame bundle $\Fr(A)\to A$ is
$\cA$-equivariant.   We can now define the  frame bundle of the orbifold $X$  as the quotient stack
$\Fr(X):=[\Fr(A)/\cA]$. It does not depend on the choice of the atlas up to natural equivalence.

\begin{ddd}
An orbifold $M$ is called effective if the total space of its frame bundle $\Fr(X)\to X$ is
equivalent to a smooth manifold.
\end{ddd} 

It is known that an effective orbifold is presentable.
On the other hand it is an open problem whether all orbifolds are presentable, see
\cite{MR2048526}.

\subsubsection{}\label{uidwqdqwd}

For a stack $X$ we define the inertia stack
$$LX:=X\times_{X\times X }X$$ by forming the two-categorial fibre product of
two copies of the diagnal $\diag\colon X\to X\times X$. If $X$ is an orbifold, then
the inertia stack $LX$ is again an orbifold (\textcolor{black}{compare}
\cite[\textcolor{black}{Lemma 2.33}]{math.KT/0609576} for an argument).
 In the case of a good orbifold of the form $[M/G]$ with a discrete group $G$, the inertia orbifold $L[M/G]$ is equivalent to the quotient stack 
$[\hat M/G]$, where $\hat M:=\bigsqcup_{g\in G} M^g$, $M^g\subseteq M$ is the
smooth submanifold of fixed points of $g$, and the  element $h\in G$ defines a
map $M^g\to M^{h^{-1}gh}$ in the natural way.
The $G$-space $\hat M$ is sometimes called the Brylinski space.

\subsubsection{}

\newcommand{\Site}{{\tt Site}}

For a stack $X$ we consider the site $\Site(X)$ of manifolds over $X$ (see
in\cite[\textcolor{black}{Section 2.1}]{bss}). Its objects are representable
submersions
$T\to X$ from smooth manifolds, and its morphisms are  pairs $(f,\phi)$
of a smooth map and a two-morphism
$$\xymatrix{T\ar[rr]^{f}\ar[dr]&\ar@{:>}[d]^{\phi}&T^{\prime}\ar[dl]\\&X&}\ .$$
The topology is given by open coverings of the manifolds $T$.
 We thus have a category $\Sh(X)$ of sheaves on $X$ (see
 \cite[Section 2.1]{bss} for details). A natural example of a
 sheaf on $X$ is the de Rham complex \textcolor{black}{of complex-valued forms}
 $\Omega_X\in \Sh(X)$, which is a sheaf of 
 differential graded commutative algebras over $\C$ and given by
$\Omega_{X}(T\to X):=\Omega(T)$ (the \textcolor{black}{de Rham complex of
  complex-valued forms of} $T$). Its structure maps are given by 
$(f,\phi)^{*}=f^{*}\colon \Omega(T^{\prime})\to \Omega(T)$. 
For a sheaf $F\in \Sh(X)$ we define the set (or group, ring, or differential
graded algebra depending \textcolor{black}{on the target category of} $F$) of global sections by
$$F(X):=\lim_{(T\to X)\in \Site(X)} F(T\to X)\ .$$

In the case of the de Rham complex we write  $\Omega(X):=\Omega_X(X)$\footnote{Observe that this does not introduce any notational conflict if $X$ is a manifold itself.}.
In particular we can consider the global sections of the de Rham complex $\Omega(LX)$ of the intertia stack. By definition, its cohomology
is the delocalized orbifold de Rham cohomology
$$H_{dR,deloc}(X):=H(LX):=H^*(\Omega(LX)),$$ see \textcolor{black}{\cite[Section
  3.2]{math.KT/0609576}}.
In the case of a good orbifold $X=[M/G]$ the forms on the inertia orbifold
coincide with the $G$-invariant forms on the smooth manifold $\hat{M}$:
\begin{equation}\label{formdef}
\Omega(LX) \cong \Omega(\hat M)^G \ .\end{equation}
Note that the left-hand side of this equality  has a definition which is manifestly independent of the presentation of $X$ as a quotient $X={[M/G]}$.

\subsubsection{}\label{z23r23r32r}

Let $E\to X$ be a complex vector bundle over an orbifold $X$. Recall from \ref{stack22} that this means that
$E$ is a stack and  the projection $E\to X$ is a representable map such that
$T\times_{X}E\to T$ is a complex vector bundle for all maps $T\to X$ compatibly with pull-backs along maps
$T^{\prime}\to T$ over $X$. One can check that $E$ is an orbifold, too.

Further recall that a connection $\nabla^{E}$ on $E$ can be understood as a compatible collection of connections on
the vector  bundles $T\times_{X}E\to T$.
In order to construct connections on $E$ we choose an orbifold atlas $A\to X$.  We consider the  associated proper and \'etale groupoid $\cA\colon A\times_XA\rightrightarrows A$.  The vector bundle gives rise to an $\cA$-equivariant vector bundle $E_A:=E\times_XA\to A$, where the action is a fibrewise linear map
$$\xymatrix{(A\times_XA)\times_{\pr_2, A}E_A\ar[rr]\ar[dr]^{\pr_1}&& E_A\ar[dl]\\&A&}\ .$$
A connection on $E$ induces, by definition,  an $\cA$-invariant connection on
$E_A$. On the other hand, one can check that an $\cA$-invariant connection on
$E_{A}$ uniquely determines a connection on $E$.
Thus to construct a connection on $E\to X$ first choose an arbitrary connection on $E_A$, and then average over $\cA$ in order to make it invariant.

We choose a connection $\nabla^E$.

\subsubsection{}\label{t6z111}

Consider a two-categorial pull-back
$$\xymatrix{A\ar[r]\ar[d]&B\ar[d]\\C\ar@{:>}[ur]\ar[r]^{f}&D}$$
in a two-category like $\Stack(\Mf)$.
Then  we have a natural action of the group
of two-automorphisms $\Aut(f)$ on $A$.
Furthermore,
given a morphism $g\colon X\to Y$, the natural map
$g\colon X\times_{Y}X\to Y$ comes with a canonical two-automorphism
$\phi_{g} \in \Aut(g)$. 
If we apply this to the inertia object ($g=\diag$)
$$\xymatrix{LX\ar@/^2cm/[rrdd]^{i}\ar[d]\ar[r]&X \ar[d]^{\diag}&\\
X\ar@{:>}[ur]\ar[r]^{\diag}&X\times X\ar[dr]^{\pr_{1}}&\\&&X}\ , 
$$ then we get a natural automorphism
\begin{equation}\label{defphi}\phi:=\pr_{1*}\phi_{\diag}\in \Aut(i) .\end{equation}

Let $E_L\to LX$ be the vector bundle defined by the pull-back
$$\xymatrix{E_L\ar[d]\ar[r]&E\ar[d]\\LX\ar@{:>}[ur]\ar[r]^i&X}\ .$$
The {two-automorphism $\phi\in \Aut(i)$}  induces an automorphism of vector bundles of $E_L$
$$\xymatrix{E_L\ar[dr]\ar[rr]^\rho&\ar@{:>}[d]&E_L\ar[dl]\\&LX&}\ .$$

The connection $\nabla^{E}$ induces by pull-back a connection $\nabla^{E_{L}}$.
Using the curvature $R^{\nabla^{E_L}}\in \Omega(LX,\End(E_L))$ of the connection $\nabla^{E_L}$
we define the Chern form
\begin{equation}\label{cherndef}\ch({\nabla^E}):=\Tr \: \rho\: \ee^{-\frac{1}{2\pi i} R^{\nabla^{E_L}}}\in \Omega(LX)\ .\end{equation}
                                                                                                                   
This form is closed and represents the Chern character of $E$ in delocalized cohomology $H_{dR,deloc}(X)$.


\subsubsection{}

In order to motivate this definition of the Chern form we consider the example
of quotient stacks. 
 If $X=[M/G]$ for a discrete group $G$ then  we have $LX\cong [\hat M/G]$ as
 above (see \ref{uidwqdqwd}) with $\hat M=\bigsqcup_{g\in G}M^g$. The map
 $i\colon LX\to X$ is represented by the map of groupoids 
 $(\hat M\times G\rightrightarrows \hat M)\to (M\times G\rightrightarrows M)$ which 
  on morphisms is given by  $(x\in M^{g},h)\mapsto (x,h)$. In this picture
the automorphism $\phi\in \Aut(i)$ is given by 
$\hat M\to M\times G$, 
$(x\in M^g)\mapsto (x,g)$.
 
We consider a $G$-equivariant vector bundle $\tilde E\to M$. Then $$E:=[\tilde E/G]\to X=[M/G]$$ is a vector bundle in stacks.
The bundle
$E_{L}\to LX$ is represented by the maps of groupoids
$(\hat E \times G\rightrightarrows \hat E) \to (\hat M \times G\rightrightarrows \hat M)$,
where $\hat E\to \hat M$ is the $G$-equivariant vector bundle defined as the pull-back of $\tilde E$ along the map $\hat M\to M$, $(x\in M^{g})\mapsto x$. The automorphism $\rho$ of $E_{L}$ is represented by the bundle automorphism $\hat \rho\colon \hat E\to \hat E$ which reduces to the action of $g$ on each fibre
$\hat E_{(x\in M^{g})}\cong \tilde E_{x}$.

We choose a $G$-invariant connection $\nabla^{\tilde E}$. It induces  connections $\nabla^{E}$ and $\nabla^{\hat E}$.
In this case the Chern form
$\ch(\nabla^{E})$ defined in (\ref{cherndef}) is given by the
invariant form
$$\Tr \:\hat \rho \:\ee^{-\frac{2}{2\pi i} R^{\nabla^{\hat E}}}\in \Omega(\hat M)^{G}\stackrel{(\ref{formdef})}{\cong} \Omega(LX)\ .$$
This is exactly the definition of the Chern form given by Baum and Connes in
\cite{bc}.

\subsubsection{}\label{uiqduwqdwqdqwdqwd}
\newcommand{\can}{{\tt can}}
The inertia orbifold $i\colon LX\to X$  has the structure of a group-object in
the two-category of stacks over $X$, see \cite[\textcolor{black}{Lemma
  2.23}]{math.KT/0609576}.   
The group structure is easy to describe in the case of a quotient stack $X=[M/G]$ for a discrete group $G$. In this case  $LX\cong \left[\left(\bigsqcup_{g\in G}M^g\right)/G\right]$, and  the multiplication and inversion $I$ are  given by 
$(x,g)(x,h):=(x,gh)$ for $x\in M^g\cap M^h$, and $I(x,g):=(x,g^{-1})$.

In general, there is a canonical isomorphism $\can\colon i\circ I\Rightarrow i$.
If $\phi\colon i\Rightarrow i$ is the natural two-automorphism of  $i$
 as in \eqref{defphi}, then 
\begin{equation}\label{udqidwqdwqdqwd}
 \phi^{-1}=\can\circ \phi\circ I\circ \can^{-1}\ 
\end{equation}
in $\Aut(i)$.
 
We use the inversion $I$ in order to define a real structure $Q$ on $\Omega(LX)$ by
$Q(\omega):=I^*\overline{\omega}$. We define the subcomplex of real forms
  $\Omega_\R(LX)\subseteq \Omega(LX)$ as the subspace of $Q$-invariants.
  The isomorphism $\can\colon  i\circ I\Rightarrow i$ induces an
     isomorphism of bundles
  $E_L\cong i^*E \stackrel{\sim}{\to} I^* i^* E\cong I^*E_L$.
It follows from (\ref{udqidwqdwqdqwd}) that
$$\xymatrix{E_L\ar[r]^\rho\ar[d]^{\can}_\cong&E_L\ar[d]^{\can}_\cong\\I^*E_L\ar[r]^{I^*\rho^{-1}} &I^*E_L}$$
commutes.

This is easy to check  directly in the case $X=[M/G]$ with $G$ finite. In this case
$\rho$ is given by the induced action of $g$ on $(E_L)_x\cong E_x$
for  $x\in M^g$,  but $I(x) \in M^{g^{-1}}$, and thus $I^*\rho$ is given by
the induced action of $g^{-1}$ on $(E_L)_{I(x)}\cong E_x$.

We can choose a hermitean
metric on $E$ and a onnection compatible with this metric.  
In fact, one can choose an orbifold atlas $f\colon A\to
  X$  and a metric and  metric connection on $f^* E\to A$. 
By averaging  one can make these invariant under the groupoid $A\times_XA\rightrightarrows A$. The invariant metric and connection  give  a metric and a metric connection on $X$, see \ref{z23r23r32r}.

The metric on $E$ induces a metric on $E_L$, and
the morphism $\rho$ is unitary. Furthermore,   the curvature
of a metric  connection takes values in the antihermitean endomorphisms.
Because the connection pulls back from $X$ we have  $I^*\nabla^{E_L}=\nabla^{E_L}$
 under the canonical isomorphism $I^*E_L\cong E_L$. A similar equality holds true for the curvature.
Combining all these facts we see  that the Chern form for a metric connection is real, i.e.~we that 
$$\ch(\nabla^E)\in \Omega_\R(LX)\ .$$

\subsubsection{}\label{dqwekdwqdqwdwq}

\newcommand{\Top}{{\tt Top}}
Using the methods of \cite{MR2119241} or \cite{freed-2007} one can define complex $K$-theory for local quotient stacks. Here we consider stacks on the site of topological spaces $\Top$ with the open covering topology, see \cite[Chapter 6]{period} and the corresponding notions of representability of maps and locally trivial bundles.
A local quotient stack is a  stack which admits a covering by open substacks of the form
$[U/G]$ where $U$ is a locally compact space and the topological group $G$ is compact.

Let us explain, for example, the definition of $K$-theory according 
\textcolor{black}{to} \cite[\textcolor{black}{Section 3.4}]{freed-2007}. It is
based on the notion of a universal
bundle of separable Hilbert spaces $H\to X$. Here universality is the property that for every
other bundle of separable Hilbert spaces $H_1\to X$ we have an isomorphism $H\oplus H_1\cong H$. Let $\Fred(H)\to X$ be the associated bundle of Fredholm operators. It gives rise to a sheaf of sections
which can naturally be enhanced to a sheaf of spaces, e.g. using simplicial methods.
By $\Gamma(X,\Fred(H))$ we denote the space of global sections of $\Fred(H)\to X$.
Then one defines $K^{-*}(X)$ as the  homotopy group $\pi_*(\Gamma(X,\Fred(H)))$. One can also directly define the group $K^{-1}(X)$ as the 
group $\pi_0(\Fred^*(H))$ of homotopy classes of sections of selfadjoint Fredholm operators with
infinite \textcolor{black}{dimensional
positive and negative spectral subspace}.

A stack in manifolds $X\in \Stack(\Mf)$ in general can not be considered as a topological stack
since it is not clear how to evaluate $X$ on test spaces $T$ which are not manifolds. 
However, the inclusion $\Mf\hookrightarrow \Top$, $M\mapsto M_{\Top}$, of the category of manifolds in the category of topological spaces extends to smooth stacks as follows.
If $X$ is a smooth stack, then we can choose an atlas $A\to X$ and obtain a natural isomorphism
$X\cong [A/\cA]$, where $\cA=(A\times_{X}A\rightrightarrows A)$.
The smooth groupoid $\cA$ has an underlying topological groupoid $\cA_{\Top}$, and we obtain the topological stack $X_{\Top}:=[A_{\Top}/\cA_{\Top}]$. The stack $X_{\Top}$ does not depend on the choice of the atlas up to natural isomorphisms. 

Hence we can apply this construction to orbifolds. If $X$ is an orbifold, then $X_{\Top}$ is a local quotient stack, and we can define its $K$-theory by 
\begin{equation}\label{kh56}K^{*}(X):=K^{*}(X_{\Top})\ .\end{equation}

\subsubsection{}\label{cgd222}

For the present paper this set-up is too general since we want to do local index theory.
In our case we want to represent $K$-theory classes by indices of families of
Dirac operators, or in the optimal case, by vector bundles. For compact
presentable orbifolds a construction of $K$-theory in terms of vector bundles
has been given in \cite[\textcolor{black}{Definition 4.1}]{MR1993337}. Note that
a vector bundle on an orbifold as defined in
Subsection \ref{stack22} is an orbifold vector bundle in the terminology of
\cite{MR1993337}\footnote{{As an illustration, let $\Z/2\Z$ act on $\R$ by
    reflection at $0$. Then the map of orbifolds $[\R/(\Z/2\Z)]\to *$ is not a
    vector bundle.}}. At the moment, for general (not presentable) orbifolds,
it is not clear that the definition
(\ref{kh56}) is equivalent to a definition based on vector bundles.

For presentable orbifolds we can also use equivariant $K$-theory.
Let $X$ be an orbifold and consider a presentation  $[M/G]\cong X$. Then
 the category of vector bundles over $X$ is equivalent to the category of $G$-equivariant vector bundles over $M$. The Grothendieck group of the latter is $K^0_G(M)$, and we have
$K^0(X)\cong K^0_G(M)$, see \cite[Proposition 4.3]{MR1993337}.
The isomorphism $K(X)\cong K_G(M)$
can  be taken as an alternative  definition since independence of the
presentation follows e.g.~from \cite[\textcolor{black}{Proposition
  4.1}]{pronk-2007}.

For a compact presentable orbifold $B$ 
the description of $K^{0}(B)$ in terms of vector bundles over $B$  shows that the construction of Chern forms (\ref{cherndef}) induces a natural transformation
$$\ch_{dR}\colon K^{0}(B)\to H^{ev}_{dR,deloc}(B)$$
in the usual manner. The odd case
$$\ch_{dR}\colon K^{1}(B)\to H^{odd}_{dR,deloc}(B)$$ is obtained from the even case using suspension by $S^{1}$.

\subsection{Cycles}
\subsubsection{}
 
In this paper we construct the differential $K$-theory of compact presentable orbifolds. 

The restrition to compact orbifolds is due to the fact that we work with absolute $K$-groups. One could in fact modify the constructions in order to produce compactly supported differential $K$-theory or relative differential 
$K$-theory. But in the present paper, for simplicity, we will not discuss relative differential cohomology theories.

We restrict our attention  to presentable orbifolds since we want to use equivariant techniques. We do not know if our approach extends to general compact orbifolds, see \ref{uidqwdqwdqwdqwdd}.

\subsubsection{}

We define the differential $K$-theory $\hat K(B)$ as the group completion of a quotient of a semigroup
of isomorphism classes of cycles by an equivalence relation. We start with the description of cycles.

\begin{ddd}
Let $B$ be a compact presentable orbifold, possibly with boundary.
A cycle for a differential $K$-theory class over $B$ is a pair
$(\cE,\rho)$, where $\cE$ is a geometric family, and $\rho\in
\Omega(LB)/\im(d)$ is a class of differential forms.
\end{ddd}

\subsubsection{}

\newcommand{\geomfam}{{\tt GeomFam}}
In the smooth case 
the notion of a geometric family has been introduced in
\cite[\textcolor{black}{Definition 2.2.2}]{math.DG/0201112} 
in order to have a short name for the data needed to define a Bismut
super-connection \cite[Proposition 10.15]{bgv}. In the present paper we need the straightforward 
generalization of this notion to orbifolds. In fact, one can consider the
stack $\geomfam$
which associates to a test manifold $T$ the groupoid of geometric families and isomorphisms over $T$.
Then a geometric family over a stack $X$ is just a map $X\to \geomfam$. Let us spell-out this in greater detail.
Let $B$  be an orbifold (or more general, an arbitrary stack on $\Mf$).
\begin{ddd}
A geometric family over $B$ consists of the following data:
\begin{enumerate}
\item a proper representable submersion with closed fibres $\pi\colon E\to B$,
\item a vertical Riemannian metric $g^{T^v\pi}$ {as in
    \ref{stack22}}, 
\item a horizontal distribution $T^h\pi$ (see \ref{stack22}) 
\item a family of Dirac bundles $V\to E$,
\item an orientation of $T^v\pi$.
\end{enumerate}
\end{ddd}
Here, a family of Dirac bundles consists of\begin{enumerate}
\item a hermitean vector bundle with connection $(V,\nabla^V,h^V)$ on $E$,
\item a Clifford multiplication $c\colon T^v\pi\otimes V\to V$,
\item on the components where
$\dim(T^v\pi)$ has even dimension a $\Z/2\Z$-grading $z$.
\end{enumerate}
We require that the restrictions of the family of Dirac bundles to the
fibres $E_b:=\pi^{-1}(b)$, $b\in B$, give Dirac bundles in the usual
sense as in \cite[Definition  3.1]{math.DG/0201112}, \textcolor{black}{namely}: 
\begin{enumerate}
\item The vertical metric induces the Riemannian structure on $E_b$.
\item The Clifford multiplication turns $V_{|E_b}$ into a Clifford module (see \cite[Definition 3.32]{bgv}) which is graded if $\dim(E_b)$ is even.
\item The restriction of the connection $\nabla^V$ to $E_b$ is  a Clifford connection   (see \cite[Definition 3.39]{bgv}). 
\end{enumerate}

Since $\pi$ is representable and a locally trivial fibre bundle its fibres are smooth manifolds.
All stackyness of $E$ is ``induced'' from $B$.
In particular all fibrewise structures, but also the notions of a connection   $\nabla^{V}$ and a horizontal distribution $T^{h}\pi$, are well-defined as   explained in Subsection \ref{stack22}.

It is also useful to understand 
a geometric family on $B$, i.e. a map $B\to \geomfam$, as a collection of geometric families
($\cE_{T\to B})_{T\to B}$ defined for all maps $T\to B$ from smooth manifolds
$T$  together with isomorphisms
$(f,\phi)^{*}\colon \cE_{T\to B}\stackrel{\sim}{\to} f^{*}\cE_{T^{\prime}\to B}$ for all
pairs of a smooth map and a two-morphism
$$\xymatrix{T\ar[rr]^{f}\ar[dr]&\ar@{:>}[d]^{\phi}&T^{\prime}\ar[dl]\\&B&}$$
which are compatible with compositions.

If $B$ is a smooth stack, then
using an atlas $f\colon A\to B$ we can give a third equivalent definition of a geometric family. We can form the groupoid $\cA:=(A\times_BA \rightrightarrows A)$ which represents the stack $B$. The pull-back of the geometric family along
$f$ is the  geometric family  $\cE_{A\to B}$ in the non-singular setting which in addition carries an action of the groupoid $\cA$. 
We can define a geometric family over $B$ as an $\cA$-equivariant geometric family over $A$.

Let $[M/G]\cong B$ be a presentation and $\cE$ be a geometric family over $B$.
Then $M\times_BE\to M$ is the underlying bundle of a $G$-equivariant geometric family $M\times_B\cE$ over $M$.
Vice versa, a $G$-equivariant geometric family $\cF$ over $M$ induces a geometric family $\cE:=[\cF/G]$ over $B$. If $F\to M$ is the underlying $G$-equivariant bundle, then
the underlying bundle of $\cE$ is the map of quotient stacks $[F/G]\to [M/G]\cong B$.

A geometric family is called even or odd, if $T^v\pi$ is
even-dimensional or odd-dimensional, respectively.

\subsubsection{}
\label{sec:Dirac}

Let $\cE$ be an even geometric family over a presentable compact orbifold  $B$. 
It gives rise to a bundle of graded separable Hilbert spaces $H_1\to B$ with fibre $H_{1,b}\cong L^2(E_b,V_{|E_b})$.
We furthermore have an associated family of Dirac operators which gives rise to a section
$F_1:=D^+(D^2+1)^{-\frac{1}{2}}\in \Fred(H_1^+,H_1^-)$.
Let $H\to B$ be the universal Hilbert space bundle as in \ref{dqwekdwqdqwdwq}. We choose isomorphisms
$H_1^\pm\oplus H\cong H$. Extending $F$ by the identity we get a section 
$F\in \Gamma(B,\Fred(H))$. \textcolor{black}{By definition, its} homotopy class represents the index
$$\ind(\cE)\in K^0(B)$$ of the geometric family.

Alternatively we can use a presentation $[M/G]\cong B$. Then
$M\times_B\cE$ is a $G$-equivariant geometric family over $M$. The index of the associated equivariant family of Dirac operators 
$\ind(M\times_N\cE)\in K_G^0(M)$ represents $\ind(\cE)\in K^0(B)$ under the isomorphism
$K^0(B)\cong K_G^0(M)$.

The index of an odd geometric family can be understood in a similar manner.

As an illustration let us consider the case where $B=BG=[*/G]$ for a finite group $G$. In this case a geometric family $\cE$ over $B$ is the same as a $G$-equivariant geometric family $\cE_{*\to B}$ over $*$. The universal Hilbert bundle is given by a universal separable Hilbert representation $H_{univ}$ of $G$ which contains each representation with infinite multiplicity. We write $$H_{univ}\cong \bigoplus_{\rho\in \hat G} H(\rho)\otimes V_{\rho} ,$$ where
$H(\rho)$ is the space of multiplicities of the irreducible unitary representation
$\rho\in \hat G$ on $V_{\rho}$.  The space
$\Gamma(B,\Fred(H))$ \textcolor{black}{then is identified} with  the space of
$G$-invariant Fredholm operators on $H_{univ}$ which decomposes into a
product $$\Gamma(B,\Fred(H))\cong \prod_{\rho\in \hat G} \Fred(H(\rho))\
,\quad F=\prod_{\rho\in \hat G}F_{\rho}\ .$$
It follows that
$$K^{0}(B)\cong \prod_{\rho\in \hat G} \Z\cong R(G)\ ,$$
and the index of $F$ is given by 
$\prod_{\rho\in \hat G} \ind(F_{\rho})$.
Hence, the index of the geometric family $\cE$ is exactly the $G$-equivariant index of the Dirac operator associated to $\cE_{*\to B}$ 
which takes values in the representation ring $R(G)$ of $G$.

\subsubsection{}\label{zerofibre}

Here is a simple example of a geometric family $\cV$ with zero-dimensional fibres.
Let $\pi \colon V\to B$ be a complex $\Z/2\Z$-graded vector bundle. Note that the projection of a vector bundle $\pi$ is by definition representable so that the fibres $V_b$ for $b\in B$ are complex vector spaces.

Assume that $V$ comes with a hermitean metric
$h^V$ and a hermitean connection $\nabla^V$ which are compatible with the $\Z/2\Z$-grading.
The geometric bundle $(V,h^V,\nabla^V)$ will usually be denoted by $\bV$.

Using a presentation of $B$ it is easy to construct a metric and a connection on a given vector bundle $V\to B $.
 Indeed, let $[M/G]\cong B$ be a presentation. Then $M\times_BV\to V$ is a $G$-equivariant vector bundle over $M$. We now can choose some metric or connection (by glueing local choices using a partition of unity). Then we can avarage these choices in order to get $G$-equivariant structures.
These induce corresponding structures on the quotient $V\cong [M\times_BV/G]$.

Alternatively one could use an orbifold atlas $A\to B$ and   
choose a metric or connection on the bundle $A\times_BV\to V$.
Again we can average these objects with respect to the action of the groupoid
$A\times_BA\rightrightarrows A$ in order to get equivariant geometric structures.
These induce corresponding structures on $V\to B$.

The underlying bundle of $\cV$ is the submersion $\pi:=\id_B\colon B\to B$. In this case the vertical bundle
is the zero-dimensional bundle which has  a canonical vertical Riemannian metric $g^{T^v\pi}:=0$. 
Let us describe the horizontal distribution of $\cV$. For every map $A\to B$ from a manifold $A$ the underlying bundle of $\cV_{A\to B}$ is the bundle $\id_{A}\colon A\to A$. 
The horizontal distribution $T^{v}\pi$ specializes to the  $TA\to A$.

Furthermore, there is a canonical orientation of $\pi$.
The geometric bundle $\bV$   can naturally be interpreted as a family of Dirac bundles on $B\to B$. 
In this way $\bV$ gives rise to a geometric family $\cV$ over $B$.

This construction shows that we can realize every class in $K^0(B)$ for a presentable orbifold $B$ as the index of a geometric family. We choose a presentation $B\cong  [M/G]$ so that
$K^0(B)\cong K_G^0(M)$.  If $x\in K^0(B)$, then there exists a $G$-equivariant
$\Z/2\Z$-graded vector bundle $W\to M$ which represents the image of $x$ in
$K^0_G(M)$. Let $V:=[W/G]\to B$ be the induced vector bundle over $B$ and
$\cV$ be the associated geometric family. Then we have $\ind(\cV)=x$.

\subsubsection{}

In order to define a representative of the negative of the differential $K$-theory class represented by a cycle $(\cE,\rho)$ we introduce the notion of the opposite geometric family.

\begin{ddd}\label{oppdef} The opposite $\cE^{op}$ of a geometric family $\cE$ is obtained by reversing
the signs of the Clifford multiplication and the grading (in the even case) of the underlying family of Clifford bundles, and of the orientation of the vertical bundle. 
\end{ddd}

\subsubsection{}

Our differential $K$-theory groups will be $\Z/2\Z$-graded. On the level of cycles the grading is reflected
by the notions of even and odd cycles.

\begin{ddd}
A cycle $(\cE,\rho)$ is called even (or odd, resp.), if $\cE$ is even (or odd, resp.) and $\rho\in \Omega^{odd}(LB)/\im(d)$
(or $\rho\in \Omega^{ev}(LB)/\im(d)$, resp.).
\end{ddd}

\subsubsection{}

Let $\cE$ and $\cE^\prime$ be two geometric families over $B$. 
 An isomorphism $\cE\stackrel{\sim}{\to} \cE^\prime$ is a two-isomorphism
 $\cE\Rightarrow \cE^{\prime}$ between maps of stacks $B\to \geomfam$.
 In explicit terms such a two-isomorphism
 consists of the following data:
$$\xymatrix{V\ar[d]\ar[rr]^F&\ar@{:>}[d]^{\psi}&V^\prime\ar[d]\\E\ar[dr]^\pi\ar[rr]^{\quad
    f}&\ar@{:>}[d]^{\phi}&E^\prime\ar[dl]_{\pi^\prime}\\&B&}\,$$
where
\begin{enumerate}
\item $(f,\phi)$ is an isomorphism over $B$,
\item $(F,\psi)$ is a bundle isomorphism over $f$,
\item $f$ preserves the horizontal distribution, the vertical metric, and the
  orientation. 
\item $F$ preserves the connection, Clifford multiplication, and the grading.
 \end{enumerate}
Compared with the non-singular case the new ingredients are the two-isomorphisms  $\phi$ and $\psi$
 which are parts of the data. Alternatively one could define the notion of an isomorphism between $\cE$ and $\cE^{\prime}$ as a collection of isomorphisms of geometric families $(\cE_{T\to B}\cong \cE^{\prime}_{T\to B})_{T\to B}$ which is compatible with pull-backs along maps 
$$\xymatrix{T\ar[rr]\ar[dr]&\ar@{:>}[d]&T^{\prime}\ar[dl]\\&B&}$$
of manifolds over $B$.

\begin{ddd}\label{may20113}
Two cycles $(\cE,\rho)$ and $(\cE^\prime,\rho^\prime)$ are called isomorphic if
$\cE$ and $\cE^\prime$ are isomorphic and $\rho=\rho^\prime$.
We let $G^*(B)$ denote the set of isomorphism classes of cycles over $B$ of parity $*\in \{ev,odd\}$.
\end{ddd}

\subsubsection{}
Given two geometric  families $\cE$ and $\cE^\prime$ we can form their sum $\cE\sqcup_B \cE^\prime$
over $B$. The underlying proper submersion with closed fibres  of the sum is $\pi\sqcup \pi^\prime\colon E\sqcup E^\prime\to B$.
The remaining structures of $\cE\sqcup_B\cE^\prime$ are induced in the obvious way.

\begin{ddd}\label{isodeff}
The sum of two cycles $(\cE,\rho)$ and $(\cE^\prime,\rho^\prime)$  is defined by
$$(\cE,\rho)+(\cE^\prime,\rho^\prime):=(\cE\sqcup_B \cE^\prime,\rho+\rho^\prime)\ .$$
\end{ddd}
The sum of cycles induces on $G^*(B)$ the structure of a graded abelian semigroup.
The identity element of $G^*(B)$ is the cycle $0:=(\emptyset, 0)$, where
$\emptyset$ is the empty geometric family.

\subsection{Relations}

\subsubsection{}\label{pap1}

In this subsection we introduce an equivalence relation $\sim$ on $G^*(B)$.
We show that it is compatible with the semigroup structure so that we get a semigroup 
$G^*(B)/\sim$. We then define the differential $K$-theory 
$\hat K^*(B)$ as the group completion of  this quotient.

In order to define $\sim$ we first introduce a simpler relation "paired" which has a nice local  index-theoretic  meaning. The relation $\sim$ will be the equivalence relation generated by ``paired''.

\subsubsection{}\label{uidqwdqwdqwdqwdd}

 The main ingredients of our definition of ``paired'' are the notions
of a taming of a geometric family $\cE$ introduced in \cite[Definition  4.4]{math.DG/0201112}, and the $\eta$-form of a tamed family \cite[Definition  4.16]{math.DG/0201112}.

In this paragraph we shortly review the notion of a taming and the construction of the eta forms.  In the present paper we will use $\eta$-forms as a black box with a few important  properties which we explicitly state at the appropriate places below.
 
If $\cE$ is a geometric family over $B$, then we can form a family of Hilbert spaces $H(\cE)\to B$ with fibre $H_b:=L^2(E_b,V_{|E_b})$. If $\cE$ is even, then this family is in addition $\Z/2\Z$-graded.

A pretaming of $\cE$ is a smooth  section $Q\in \Gamma(B,B(H(\cE)))$ such that $Q_b\in B(H_b)$ is selfadjoint
given by a smooth integral kernel $Q\in C^\infty(E\times_B E,V\boxtimes V^*)$.
In the even case we assume in addition that $Q_b$ is odd, i.e. that it anticommutes with the grading $z$. The geometric family $\cE$ gives rise to a family of Dirac operators
$D(\cE)$,  where $D(\cE_b)$ is an unbounded selfadjoint operator on $H_b$, which is odd in the even case.

The pretaming is called a taming if $D(\cE_b)+Q_b$ is invertible for all $b\in B$.

 In the above description we followed the philosophy that all notions
 involved have a natural meaning if $B$ is an orbifold. For example, the datum
 of  a (pre)taming of $\cE$ is equivalent to a collection of (pre)tamings of
 the geometric families $\cE_{T\to B}$ (the non-singular case) for all maps
 $T\to B$ from smooth manifolds $T$ which is compatible with pull-backs.

The family of Dirac operators $D(\cE)$
has a $K$-theoretic  index which we denoted in \ref{sec:Dirac} by
$$\ind(\cE)\in K(B)\ .$$

If the geometric family $\cE$ admits a taming, then the associated family of Dirac operators admits an invertible compact perturbation, and hence $\ind(\cE)=0$.
In the non-singular case the converse is also true.  Assume that $B$ is a smooth manifold. If  $\ind(\cE)=0$ and
$\cE$ is not purely zero-dimensional then $\cE$ admits a taming.
The argument is as follows.
The bundle of Hilbert spaces $H(\cE)\to B$ is universal. 
If $\ind(\cE)=0$ then  the section of unbounded Fredholm operators $D(\cE)$ admits an invertible compact perturbation $D(\cE)+\tilde Q$. We can approximate $\tilde Q$ in norm
by pretamings. A sufficiently good approximation of $\tilde Q$ by a pretaming is a taming.

In the orbifold case the situation is more complicated. In general, the bundle
$H(\cE)\to B$ is not universal. Therefore we may have to stabilize.
It is at this point that we use the assumption that the orbifold is presentable.
\begin{lem}\label{uiufwefewfwfewf}
If $\ind(\cE)=0$, then there exists a geometric family $\cG$ (of the same parity of $\cE$) such that
$\cE\sqcup_B\cG \sqcup_B \cG^{op}$ has a taming.
\end{lem}
\proof
We first consider the even case.
Let $[M/G]\cong B$ be a presentation and
 $\cF:=M\times_B\cE$ be the corresponding equivariant geometric family.
Let $H^+$  be a universal $G$-Hilbert space, i.e.\ a $G$-Hilbert space isomorphic to
$l^2\otimes L^2(G)$. We consider the $\Z/2\Z$-graded space $H:=H^+\oplus \Pi H^+$, where for a $\Z/2\Z$-graded vector space $U$
the symbol $\Pi U$ denotes the same underlying vector space equipped
with the opposite grading. The sum  $H(\cF)\oplus
H\times M$ is now a universal equivariant Hilbert space bundle.
Since $\ind(\cE)=0$, the extension $D(\cF)\oplus 1$ of $D(\cF)$ to
 $H(\cF)\oplus H\times M$ has an equivariant compact
selfadjoint odd invertible perturbation $D(\cF)\oplus 1+\tilde Q$.

In the next step we cut down $H$ to a finite-dimensional subspace.
Let $(P^+_n)$ be a sequence of invariant projections on $H^+$ such
that $P^+_n\xrightarrow{n\to\infty} \id_{H^+}$
strongly. These exist because $G$ is compact and so
$L^2(G)$ is a sum of finite dimensional irreducible representations. We set $P_n:=P_n^+\oplus P_n^+$ on $H=H^+\oplus \Pi H^+$.
 Using compactness of $M$, for sufficiently large $n$ the operator
$(1\oplus P_n)((D(\cF)\oplus 1)+\tilde Q)(1\oplus P_n)$ is invertible on $\im(1\oplus P_n)$.
Hence we have found a finite-dimensional $G$-representation $V:=P_nH$
of the form $V=V^+\oplus \Pi V^+$ such that the perturbation
$D(\cF)\oplus 1+\hat  Q$ of $D(\cF)\oplus 0$ by the equivariant
compact odd selfadjoint $\hat Q:=1\oplus P_n+(1\oplus P_n) \tilde Q
(1\oplus P_n)$ is invertible on $H(\cF)\oplus V\times M$.  
Finally we approximate
$\hat Q$ by a family $Q$ represented by a smooth integral kernel,
 where we think of $V\times M$ as a bundle over an
additional one-point component of the fibers of the new family, see below.

 Denote by $\cV^+$ the
 equivariant zero-dimensional geometric family based on the
trivial bundle $M\times V\to M$. Then we set $\cG:=[\cV^+/G]$.  
The operator $Q$ constructed above provides the taming
of $\cE\sqcup_B \cG\sqcup_B\cG^{op}$.

In the odd-dimensional case we argue as follows.
We again choose a presentation $[M/G]\cong B$ and form $\cF:=M\times_B\cE$ as above. In this case we let $H:=H^+$
be an ungraded universal $G$-Hilbert space.

Since $\ind(\cE)=0$ 
the extension $D(\cF)\oplus 1$ of $D(\cF)$ to
$H(\cF)\oplus H\times M$ admits an 
  equivariant compact
selfadjoint invertible perturbation $D(\cF)\oplus 1+\tilde Q$.
We can again find a finite-dimensional projection $P_n$ on $H$ such that
$(1\oplus P_n)(D(\cF)\oplus 1+\tilde Q)(1\oplus P_n)$ is still invertible.
We get the invertible operator
$D(\cF)\oplus 1+\hat Q$ on $H(\cE)\oplus V$ with
 $V:=P_n H\times M$ and  $\hat Q:=1\oplus P_n+(1\oplus P_n) \tilde Q (1\oplus P_n)$. We again approximate $\hat Q$ by an operator $Q$ with smooth kernel. 

We now choose an odd geometric family $\cX$ over a point such that $\dim\ker(D(\cX))=1$
and form the  $G$-equivariant family $\cY:=p^*\cX\otimes V$, where $p\colon M\to *$.
The kernel of $D(\cY)$ is isomorphic to $M\times V$.
Using this identification we can define $Q$ on $H(\cF)\oplus \ker(D(\cY))$. Its extension by zero on
$H(\cF)\oplus H(\cY)=H(\cF\sqcup_M\cY)$ is a taming of
$\cF\sqcup_M \cY$. 

Let $R$ be the projection onto $\ker(D(\cY))$.
The operator $D(\cY)+R$ is invertible so that we can consider
$R$ as a taming  of $\cY^{op}$. All together,
$Q\oplus R$ defines a $G$-equivariant taming of $ \cF\sqcup_M \cY\sqcup_M\cY^{op}$.
We now let $\cG:=[\cY/G]$ and get a taming of
$\cE\sqcup_B\cG\sqcup_B\cG^{op}$.
\hB

%
%

\begin{ddd}\label{tgeomf}
A geometric family $\cE$ together with a taming will be denoted by $\cE_t$ and called a tamed geometric family.
\end{ddd}
Let  $\cE_t$ be a taming of the geometric family $\cE$ by the family $(Q_b)_{b\in B}$.
\begin{ddd} The opposite tamed family  $\cE_t^{op}$ is given by the taming
$-Q\in \Gamma(B,B(H(\cE)))$ of $\cE^{op}$.
\end{ddd}
Note that the bundles of Hilbert spaces $H(\cE)\to B$ and $H(\cE^{op})\to B$
associated to $\cE$ and $\cE^{op}$ are canonically isomorphic (up to reversing
the grading in the even case) so that this formula makes sense.

\subsubsection{}\label{uzu1}

The local index form $\Omega(\cE)\in \Omega(LB)$ is a closed differential form
canonically assciated to a geometric family. It represents the Chern character
of the index of $\cE$. \textcolor{black}{To define and to analyze it, we use
  superconnections and the other tools of local index theory.}
Let $A_t(\cE)$ denote the family of rescaled Bismut superconnections on $H(\cE)\to B$.
We define $H(\cE)_L\to LB$ as the pull-back 
$$\xymatrix{H(\cE)_L\ar[r]\ar[d]&H(\cE)\ar[d]\\LB\ar[r]&B}\ .$$
Let $A_t(\cE)_L$ denote the pull-back of the superconnection.
As explained in \ref{z23r23r32r} the bundle
$H(\cE)_L$ comes with a canonical automorphism
$\rho_{H(\cE)_L}$. For $t>0$ the form
$$\Omega(\cE)_t:=\varphi\:\Tr_s\:\rho_{H(\cE)_L} \: \ee^{-A_t^2(\cE)_L}\in
\Omega_{{\R}}(LB)$$ 
is closed and  real by the argument
  given in \ref{uiqduwqdwqdqwdqwd}.
Here $\varphi$ is a normalization operator. It acts on $\Omega(LB)$ and is defined by  
$$\varphi:=\left\{\begin{array}{cc}
(\frac{1}{ 2\pi i})^{\deg/2}& \mbox{even case}
\\
\frac{-1}{\sqrt{\pi}} (\frac{1}{ 2\pi i})^{\frac{\deg-1}{2}}& \mbox{odd case}
\end{array} \right.
\ .$$

All the analysis here is fibrewise and the fibres are smooth. The theory
developed e.g.~in the book \cite{bgv} applies without changes. The stackyness
of $B$ or $LB$ is only reflected by additional \textcolor{black}{invariance
  properties}. The technical way to translate to the classical situation is
again to work with
the  compatible collection of superconnections $(A_{t}(\cE_{T\to LB})_{L})_{T\to LB}$ for all
maps $T\to LB$ from smooth manifolds. The theory of   \cite{bgv} applies to the specializations $A_{t}(\cE_{T\to B})_{L}$ immediately. For example, the collection of forms $(\Omega(\cE_{T\to LB})_t)_{T\to LB}$ is compatible and therefore
 defines an element $\Omega(\cE)_{t}$ of
$\Omega(LE)=\lim_{(T\to LB)}\Omega(T)$.  
A similar reasoning is applied in order to interpret the arguments below.

The methods of local index theory show that $\Omega(\cE)_t$ has a limit as $t\to \infty$.
\begin{ddd}\label{may20112}
We define the local index form
$\Omega(\cE)\in \Omega_{{\R}}(LB)$ of the geometric family $\cE$ over $B$ as the limit
$$\Omega(\cE):=\lim_{t\to 0}\Omega(\cE)_t\ .$$ 
\end{ddd}
We have the following special case of Theorem \ref{diqwdwqdddwqd},
which also covers families of manifolds with boundary.
\begin{theorem} \label{thm2}
$$\ch_{dR}(\ind(\cE))=[\Omega(\cE)]\in H_{dR,deloc}(B)\ .$$
\end{theorem}
In the following we give a differential geometric description of $\Omega(\cE)$.
The automorphism $\rho_{H(\cE)_L}$ comes from the canonical automorphism
$\rho_{\cE}$ of the pull-back $\cE_L:=LB\times_B\cE$. The usual finite
progagation speed estimates show that as $t$ tends to zero the supertrace
$\Tr_s\rho_{H(\cE)_L} \ee^{-A_t^2(\cE)_L}$ localizes at the fixed points of $\rho_{\cE}$.

Let $\pi\colon E\to B$ be the underlying fibre bundle of $\cE$, and let $V\to E$ be the Dirac bundle.
If we apply the loops functor  to the projection $\pi$ we get a diagram
$$\xymatrix{LE\ar[d]^{L\pi}\ar[r]&E\ar[d]^\pi\\LB\ar@{:>}[ur]\ar[r]&B}\ .$$
The fibre bundle
$LE\to LB$ is exactly the bundle of fixed points of $\rho_{\cE}$.
Therefore the local index form is given as an integral
$$\Omega(\cE)=\int_{LE/LB}I(\cE)$$ for some  $I(\cE)\in \Omega_{{\R}}(LE)$.
Let $\cU\to LE$ be a tubular neighbourhood of the local
  embedding $i\colon LE\to E$.

We let $V_L:=LE\times_EV\to LE$ be the pull-back of $V\to E$.
Similarly, we  let $T^v\pi_L\to LE$ be the pull-back of the vertical bundle $T^v\pi\to E$.
Both bundles come with canonical automorphisms (see \ref{z23r23r32r})
$$\rho_{T^v\pi_L}\colon T^v\pi_L\to T^v\pi_L\ ,\quad \rho_{V_L}\colon V_L\to V_L\ .$$
The  automorphism $\rho_{T^v\pi_L}$ preserves the orthogonal decomposition
$$T^v\pi_L\cong T^vL\pi\oplus N\ ,$$
where   $T^vL\pi=\ker(dL\pi)=\ker(1-\rho_{T^v\pi_L})$.
We let $\rho^N$ denote the restriction of $\rho_{T^v\pi_L}$ to  the  normal bundle.

Then we have (see \cite[Sec. 6.4]{bgv} for similar
arguments) \begin{equation} \label{eq:Uli_for_reference}
\lim_{t\to 0}\Tr_s\rho_{H(\cE)_L} \ee^{-A_t^2(\cE)_L}=\lim_{t\to 0}
\int_{LE/LB}\int_{\cU/LE}\tr_s  \rho_{V_L}
K_{\ee^{-A_t^2(\cE)}}((x,\rho^Nn),(x,n))\ ,
\end{equation}

where $\tr_s$ the the local super-trace of the integral kernel 
 $K_{\ee^{-A_t^2(\cE)}}((x,n),(x^\prime,n^\prime))$ of $\ee^{-A_t^2(\cE)_L}$, $x\in LE$, and $n\in \cU_x$.
The form $I(\cE)$ is thus given by
$$I(\cE)(x)=\lim _{t\to 0}\int_{\cU/LE}\tr_s  \rho_{V_L} K_{\ee^{-A_t^2(\cE)}}((x,\rho^Nn),(x,n))\ .$$

{The explicit form of the local index density will not be needed in
 rest of the  present paper. If necessary, it can be derived from the local index formulas for $G$-equivariant families 
 \cite[\textcolor{black}{Definition 1.3 and Theorem 1.1}]{lm00}}


\subsubsection{}\label{pap2}

Let $\cE_t$ be a tamed geometric family (see Definition \ref{tgeomf}) over $B$.
The taming is used to modify the Bismut superconnection $A_\tau(\cE)$ for
$\tau>1$ in order to make the zero form degree part invertible. For $\tau\ge
2$ we \textcolor{black}{set} 
$A_\tau(\cE_t)=A_\tau(\cE)+\tau Q$, for $\tau\in (0,1)$ we
\textcolor{black}{set} $A_\tau(\cE_t)=A_\tau(\cE)$, and
on the interval $\tau\in (1,2)$ we interpolate smoothly between these
two.
The taming has the effect that the integral kernel of $\ee^{-A_\tau(\cE_t)^2}$ vanishes exponentially for $\tau\to \infty$ in the $C^\infty$-sense.
The $\eta$-form $\eta(\cE_t)\in \Omega_{{\R}}(LB)$ is defined  by
\begin{equation}\label{udqwidwqdodiwodopop}
\eta(\cE_t):= \tilde \varphi\int_0^\infty \Tr_s \: \rho_{\cE}\:\partial_\tau  A_\tau(\cE_t)_L\: \ee^{-A_\tau(\cE_t)_L^2}\:d \tau\ ,
\end{equation}
where
$\tilde \varphi$ again acts on $\Omega(LB)$ and is defined by
$$\tilde \varphi=\left\{
\begin{array}{cc}
(2\pi i)^{-\frac{\deg+1}{2}}&\mbox{even case}\\
\frac{-1}{\sqrt{\pi}} (2\pi i)^{-\deg/2}&\mbox{odd case}\ .
\end{array}
\right.
$$
Note \textcolor{black}{that even and odd refer} to the dimension of the fibre. The corresponding $\eta$-form has the opposite parity.

Convergence at $\tau\to \infty$ is due to the taming. The convergence at
$\tau\to 0$ follows from the standard  equivariant local index theory for the
Bismut superconnection.
The same methods imply
 \begin{equation}\label{detad}
d\eta(\cE_t)=\Omega(\cE)\ .
\end{equation}
%
%
%

\subsubsection{}

Now we can introduce the relations ``paired'' and $\sim$.
\begin{ddd}\label{uuu1}
We call two cycles $(\cE,\rho)$ and $(\cE^\prime,\rho^\prime)$ paired if there exists a taming
$(\cE\sqcup_B \cE^{\prime op})_t$ such that
$$\rho-\rho^\prime=\eta((\cE\sqcup_B \cE^{\prime op})_t)\ .$$
We let $\sim$ denote the equivalence relation generated by the relation ``paired''.
\end{ddd}

\begin{lem}\label{lem1}
The relation ``paired'' is symmetric and reflexive.
\end{lem}
\proof
We can copy the argument of the corresponding \textcolor{black}{\cite[Lemma
  2.11]{bunke-2007}} 
literally. 
\hB

\begin{lem}\label{lem5}
The relations ``paired'' and $\sim$ are compatible with the semigroup structure on $G^*(B)$.
\end{lem}
\proof We can copy the argument of the corresponding
\cite[Lemma 2.12]{bunke-2007} literally. \hB

\begin{lem}\label{lem3}
If $(\cE_0,\rho_0)\sim (\cE_1,\rho_1)$, then 
there exists a cycle $(\cE^\prime,\rho^\prime)$
such that
$(\cE_0,\rho_0)+(\cE^\prime,\rho^\prime)$ is paired with $(\cE_1,\rho_1)+(\cE^\prime,\rho^\prime)$.
\end{lem}
\proof We can copy the argument of the corresponding
\textcolor{black}{\cite[Lemma 2.13]{bunke-2007}} literally. \hB  
The three proofs above only depend on formal properties  of
geometric families,
tamings and the associated local index- and $\eta$-forms which also hold true
in the present case.  The same remark applies to the proofs of the first three
lemmas in the next subsection.

\subsection{Differential orbifold $K$-theory}

\subsubsection{}

In this subsection we define the assignment 
$B\to \hat K(B)$ from compact presentable orbifolds to $\Z/2\Z$-graded abelian groups. Recall  Definition \ref{isodeff} of the semigroup of 
isomorphism classes of cycles. By Lemma \ref{lem5} we can form the semigroup $G^*(B)/\sim$.

 \begin{ddd}
We define the differential $K$-theory $\hat K^*(B)$ of $B$ to be the group completion of
the abelian semigroup $G^*(B)/\sim$.
 \end{ddd}
If $(\cE,\rho)$ is a cycle, then we let $[\cE,\rho]\in \hat K^*(B)$ denote the corresponding class in differential $K$-theory.

We now collect some simple facts which are helpful for computations in $\hat K(B)$ on the level of cycles.
\begin{lem}\label{lem22}
We have
$[\cE,\rho]+[\cE^{op},-\rho]=0$.
\end{lem}
\proof We can copy the argument of the corresponding \textcolor{black}{\cite[Lemma 2.15]{bunke-2007}} literally. \hB

\begin{lem}
Every element of $\hat K^*(B)$ can be represented in the form
$[\cE,\rho]$.
\end{lem}
\proof
We can copy the argument of the corresponding \textcolor{black}{\cite[Lemma
2.16]{bunke-2007}} literally.  \hB 

\begin{lem}\label{lem4}
If $[\cE_0,\rho_0]=[\cE_1,\rho_1]$, then there exists a cycle
$(\cE^\prime,\rho^\prime)$ such that
$(\cE_0,\rho_0)+(\cE^\prime,\rho^\prime)$ is paired with $(\cE_1,\rho_1)+(\cE^\prime,\rho^\prime)$.
\end{lem}
\proof We can copy the argument of the corresponding
\textcolor{black}{\cite[Lemma 2.17]{bunke-2007}} literally. \hB

\subsubsection{}

 In this paragraph we extend $B\mapsto \hat K^*(B)$ to a contravariant functor
from compact orbifolds  to $\Z/2\Z$-graded groups.
Let $f\colon B_1\to B_2$ be a morphisms of orbifolds. 
Then we define
$$f^*\colon  \hat K^*(B_2)\to  \hat K^*(B_1)$$ by
$$f^*[\cE,\rho]:=[f^*\cE,Lf^*\rho]\ ,$$
where
$f^*\cE=B_1\times_{B_2}\cE$ and $Lf\colon LB_1\to LB_2$ is obtained from $f$  by an application
of the loops functor. For the details of the construction of the pull-back of
geometric families we refer to \cite[\textcolor{black}{2.3.2}]{bunke-2007}. 
It is easy to check that the construction is well-defined and additive.
At this point we use in particular the relation
\begin{equation}\label{eq3}
\eta(f^*\cE_t)=f^*\eta(\cE_t)\ .
\end{equation}
If $g\colon B_0\to B_1$ is the second morphisms of compact presentable orbifolds, then we have the relation
$$f^*\circ g^*=(f\circ g)^*\colon \hat K(B_2)\to \hat K(B_0)\ .$$

Note that the morphisms between the orbifolds $B_{1}$ and $B_{2}$ form a groupoid. If two morphisms  $f,f^{\prime}\colon B_{1}\to B_{2}$ are two-isomorphic, then
we have the equality \begin{equation}\label{twomrph}f^{*}=f^{\prime *}\colon \hat K^{*}(B_{2})\to \hat K^{*}(B_{1})\ .\end{equation} Indeed, a two-isomorphism $\phi\colon f\Rightarrow f^{\prime}$ induces an isomorphism $f^{*}\cE\stackrel{\sim}{\to} f^{\prime *}\cE$, and we have
$Lf^{*}\rho=Lf^{\prime *}\rho$.

 \subsection{Natural transformations and exact sequences}
 
\subsubsection{}

In this subsection we introduce the transformations $R,I,a$, and we show that they turn the functor
$\hat K$ into a differential extension of $(K,\ch_\C)$ in the sense of the natural  generalization
of the definition \cite[Definition  1.1]{bunke-2007} to the orbifold case.

\subsubsection{}

We first define the  natural transformation
$$I\colon \hat K(B)\to K(B)$$
by
$$I[\cE,\rho]:=\ind(\cE)\ .$$
The proof that this is well-defined can be copied literally  from
 \cite[\textcolor{black}{2.4.2}]{bunke-2007}.
The relation
$\ind(f^*\cE)=f^*\ind(\cE)$  shows that $I$ is a natural transformation of functors
from presentable compact orbifolds to $\Z/2\Z$-graded abelian groups.

We consider the functor
$B\mapsto \Omega^*(LB)/\im(d)$, $*\in \{ev,odd\}$ as a 
functor from orbifolds to $\Z/2\Z$-graded abelian groups.
We construct a parity-reversing natural transformation
$$a\colon \Omega^*(LB)/\im(d)\to \hat K^*(B)$$ by
$$a(\rho):=[\emptyset,-\rho]\ .$$

Let $\Omega_{d=0}^*(LB)$ be the group of closed forms of parity $*$ on $B$. 
Again we consider $B\mapsto \Omega_{d=0}^*(LB)$ as a functor from orbifolds  to $\Z/2\Z$-graded abelian groups.
We define a natural transformation
$$R\colon \hat K(B)\to\Omega_{d=0}(LB)$$
by 
$$R([\cE,\rho])=\Omega(\cE)-d\rho\ .$$
The map $R$ is well-defined by the same argument as in
\cite[\textcolor{black}{2.4.5}]{bunke-2007}.
It follows from
$\Omega(f^*\cE)=f^*\Omega(\cE)$ that $R$ is a natural transformation.

\subsubsection{}

The natural transformations satisfy the following relations:
\begin{lem}\label{tzw}
\begin{enumerate}
\item $R\circ a=d$
\item $\ch_{dR}\circ I=[\dots]\circ R$. 
\end{enumerate}
\end{lem}
\proof
The first relation is an immediate consequence of the definition of $R$ and $a$.
The second relation is the local index Theorem \ref{thm2}. \hB

\subsubsection{}
Via the embedding $H_{dR,deloc}(B)=H_{dR}(LB)\subseteq \Omega(LB)/\im(d)$,
the Chern character $\ch_{dR}\colon K(B)\to H_{dR,deloc}(B)$ can be considered as a natural transformation
$$\ch_{dR}\colon K(B)\to \Omega(LB)/\im(d)\ .$$ 
\begin{prop}\label{prop1}
The following sequence is exact:
$$K(B)\stackrel{\ch_{dR}}{\to} \Omega(LB)/\im(d)\stackrel{a}{\to} \hat K(B)\stackrel{I}{\to} K(B)\to 0\ .$$
\end{prop}
\proof \textcolor{black}{The proof is carried out in the Paragraphs \ref{start} to
  \ref{dhuidwqdwqd}.}

\subsubsection{}\label{start}
We start with the surjectivity of $I\colon \hat K(B)\to K(B)$.
The main point is the fact that every element $x\in K(B)$ can be realized as the index of a geometric family over $B$.   Here we use again that the orbifold is presentable.
Let $[M/G]\cong B$ be a presentation.
Given a class in $K(B)$ let $x\in K_G(M)$ be the corresponding class under the isomorphism
$K(B)\cong K_G(M)$. It suffices to show that $x$ can be realized as the index of  a $G$-equivariant geometric family $\cE$ over $M$.
We first consider the even case.  Then $x$ can be represented by a $\Z/2\Z$-graded
 $G$-vector bundle $V\to M$. As in \ref{zerofibre} we construct a $G$-equivariant geometric family with zero-dimensional fibre $\cV\to M$ such that $\ind(\cV)=x$.

In the odd case we let $y\in K^0_G(S^1\times M,\{1\}\times M)$ be the the class
corresponding to $x$ under the suspension isomorphism
$K^0_G(S^1\times M,\{1\}\times M)\cong K^1_G(M)$.
As above we can find an equivariant geometric family
$\cV$ over $S^1\times M$ such that
$\ind(\cV)\in K^0_G(S^1\times M)$ is the image of $y$ under
$K^0_G(S^1\times M,\{1\}\times M)\to K^0_G(S^1\times M)$.
Using the standard metric on $S^1$ and the canonical horizontal bundle
$TM\subset T(S^1\times M)$ for $p\colon S^1\times M\to M$ we can define a $G$-equivariant geometric family
$p_!(\cV)$ over $M$   such that
$\ind(p_!\cV)=x$.

\subsubsection{}
Next we show exactness at $\hat K(B)$. For $\rho\in \Omega(LB)/\im(d)$  we have
$I\circ a(\rho)=I([\emptyset,-\rho])=\ind(\emptyset)=0$, hence
 $\hat I\circ a=0$.
Consider a class $[\cE,\rho]\in \hat K(B)$ which satisfies $I([ \cE,\rho])=0$.
Using Lemma \ref{uiufwefewfwfewf} and Lemma \ref{lem22}
 we can replace $\cE$ by $\cE\sqcup_B(\tilde\cE\sqcup_B\tilde \cE^{op})$
for some geometric family $\tilde \cE$  without changing the differential $K$-theory class  such that $\cE$
 admits a taming $\cE_t$.
Therefore,
$(\cE,\rho)$ is paired with  $(\emptyset,\rho-\eta(\cE_t))$.
It follows that $[\cE,\rho]=a(\eta(\cE_t)-\rho)$.

 \subsubsection{}
 
In order to prepare the proof of exactness at $\Omega(LB)/\im(d)$ 
we need some facts about the classification of tamings of a geometric family $\cE$.
As in \cite[\textcolor{black}{2.4.10}]{bunke-2007} we introduce the notion of boundary
  taming and will use an index theorem for boundary tamed families
in order to compare tamings. Let $\cF$ be a geometric family with boundary $\cE$ over $B$ and $\cE_t$ be a taming. Then we have a boundary tamed family $\cF_{bt}$ and can consider $\ind(\cF_{bt})\in K(B)$. 
\begin{theorem}\label{diqwdwqdddwqd} In $H_{dR, deloc}(B)$ we have the following equality:
$$\ch_{dR}(\ind(\cF_{bt}))=[\Omega(\cF)+\eta(\cE_t)]\ .$$
\end{theorem}
\proof We first consider the even case.
We use that $B$ is presentable so that we have
$\ind(\cF_{bt})=\ind(\cV)$ for some vector bundle $V\to B$, where $\cV$ is the geometric family associated to $V$ as in \ref{zerofibre}. By definition of the Chern character in \ref{cgd222} we have
$\ch_{dR}(\ind(\cF_{bt}))=\ch_{dR}(\ind(\cV))=[\Omega(\cV)]$. The main part of the proof is to show that
$[\Omega(\cV)]=[\Omega(\cF)+\eta(\cE_t)]$. Here
we can repeat the argument given in \cite[Theorem 4.13]{math.DG/0201112}. The only modifications are
\begin{enumerate}
\item We consider the pull-backs of $\cF$, $\cV$, and $\cE_t$ to $LB$ which
  come with canonical automorphisms $(\rho_{\cF},\rho_{\cE_t})$.
\item We replace $\Tr_s\dots$ by $\Tr_s \rho_{\cF}$, $\Tr_{s}\rho_{V}$,  or $\Tr_s \rho_{\cE_t}$,  respectively. 
\item The small time analysis of this trace takes the localization of the heat
  kernel at the fibrewise fixed points of the canonical automorphisms into
  account. To write out all the details here is of course a lengthy and
  tedious matter, but all necessary technical details of the local heat kernel
  analysis are well documented in \cite{bgv}. See also \cite{lm00} for the equivariant situation without boundary. \end{enumerate}  
The odd case is reduced to the even case by suspension as in  \cite[Theorem 4.13]{math.DG/0201112}.
\hB

In view of this theorem we can argue as in
\cite[\textcolor{black}{2.4.10}]{bunke-2007} that if 
$\cE_t$ and $\cE_t^\prime$ are two tamings of a geometric family, then
the difference of the associated $\eta$-forms is closed and we have
$$[\eta(\cE_t)-\eta(\cE_t^\prime)]\in \ch_{dR}(K(B))\subset H_{dR,deloc}(B)\ .$$

We now show exactness at $\Omega(LB)/\im(d)$.
Let  $\rho\in \Omega(LB)/\im(d)$ be such that $a(\rho)=[\emptyset,-\rho]=0$.
Then by Lemma \ref{lem4} there exists a cycle $(\hat \cE,\hat \rho)$ such that
$(\hat \cE,\hat \rho-\rho)$ pairs with $(\hat \cE,\hat \rho)$.
Using Lemma \ref{lem5} we can add a copy $\hat \cE^{op}$ and see that
$(\cE,\hat \rho-\rho)$ is paired with $(\emptyset,\hat \rho)$, where $\cE=\hat \cE\sqcup_B \hat \cE^{op}$.
The taming which induces this relation will be denoted by $\cE_t^\prime$. We have
$\eta(\cE_t^\prime)=-\rho$.  Because of the odd $\Z/2\Z$-symmetry the family $\cE$ admits another taming 
 $\cE_t$ with vanishing $\eta$-form.  
Therefore
$$\rho=-[\eta(\cE_t)]\in   \ch_{dR}(K(B))\ .$$

\subsubsection{}\label{dhuidwqdwqd}

It remains to show that for $x\in K(B)$ we have
$a\circ \ch_{dR}(x)=0$.  Note that
$a\circ \ch_{dR}(x)=[\emptyset,-\ch_{dR}(x)]$.
The proof is accomplished by   showing that there exists a geometric family $\cE=\hat \cE\sqcup_B \hat \cE^{op}$
which admits tamings
$\cE_t$ and $\cE_t^\prime$ such that
$\eta(\cE_t^\prime)-\eta(\cE_t)=\ch_{dR}(x)$. More
precisely, we will get
$\ind((\cE\times I)_{bt})=x$, where the boundary taming $(\cE\times I)_{bt}$
is induced by $\cE_t$ and $\cE_t^\prime$ and then use Theorem \ref{diqwdwqdddwqd}. 
 
To this end we modify the corresponding argument
given in \cite[\textcolor{black}{2.4.10}]{bunke-2007}.  To be specific, let us
consider the even case.
First of all, using a presentation $B\cong [M/G]$, we will actually consider the equivariant problem.
Let $H$ be a universal $G$-Hilbert space. Then the $G$-space $GL_1(H)\subset
GL(H)$ of invertible operators of the form $1+K$ with compact $K$
 has the homotopy
  type of the classifying space of $K_G^1$. Let $x\in K^1_G(M)$ be
represented by an equivariant map
$x\colon M\to GL_1(H)$. If $(P_n)$ is \textcolor{black}{an equivariant  strong}
approximation of
the identity of $H$ then, for sufficiently large $n$, \textcolor{black}{by
  compactness of $M$,} the $G$-map
$$(1-P_n)+P_nxP_n\colon M\to GL_1(H)$$ is $G$-homotopic to $x$.
Let $\cV$ be the equivariant geometric family on $M$ constructed from
the $\Z/2\Z$-graded $G$-vector bundle $V:=\im(P_n)\times M$. The matrices
$$Q:=\left(\begin{array}{cc}0&P_nx^*P_n\\P_nxP_n&0
\end{array}\right)\ ,\quad Q^\prime:=\left(\begin{array}{cc}0&\id_{V}\\\id_{V}&0
\end{array}\right)$$
represent tamings of $\cE:=\cV\sqcup_M \cV^{op}$. 
We use $Q$ and $Q^\prime$ at $\cE\times \{0\}$ and $\cE\times \{1\}$ in order to define
$(\cE\times I)_{bt}$.
As in \cite[\textcolor{black}{2.4.10}]{bunke-2007} we can now show that
$\ind((\cE\times I)_{bt})=x$. Because of the product structure 
 we have  $\Omega(\cE\times I)=0$, so that
by Theorem \ref{diqwdwqdddwqd} $\ch_{dR}(x)=\eta(\cE_t^\prime)-\eta(\cE_t)$.
 
The odd case is similar.
 \hB

\subsubsection{}
We define a real structure $\hat Q$ on $\hat K(B)$ by
$Q([\cE,\rho]):=[\cE,Q(\rho)]$, where
\ $Q(\rho)=I^*(\overline{\rho})$  is as in
\ref{uiqduwqdwqdqwdqwd}. Since the local  index forms and eta forms  are real,
$\hat Q$ is well-defined. We define the real subfunctor 
$$\hat K_\R(B):=\{x\in \hat K(B)\:|\: \hat Q(x)=x\}\ .$$
By restriction we get
natural transformations
$$R\colon \hat K_\R(B)\to \Omega_\R(LB)\ ,\quad a\colon  \Omega^*_\R(LB)/\im(d)\to \hat  K_\R(B)$$
\textcolor{black}{such} that
\begin{equation}\label{eq:LES}
K(B)\stackrel{\ch_{dR}}{\to} \Omega_{\R}(LB)/\im(d)\stackrel{a}{\to} \hat K_\R(B)\stackrel{I}{\to} K(B)\to 0
\end{equation}
is exact.

 \subsection{Calculations for $[*/G]$} 

 \subsubsection{}
Let $G$ be a finite group. We consider the orbifold $[*/G]$.
Note that $K_G^0([*/G])\cong K_G^0(*)\cong R(G)$ as rings, where $R(G)$ denotes the representation
ring of $G$.  Moreover, $K^1_G([*/G])\cong 0$.
We have $L[*/G]=[G/G]$, where $G$ acts on itself by conjugation. Therefore
$$\Omega(L[*/G])\cong \C[G]^G\cong H_{dR,deloc}([*/G])$$ is the ring of
conjugation invariant complex valued functions on $G$.
The Chern character fits into the diagram
$$\xymatrix{K^0([*/G])\ar[d]^\cong\ar[r]^(0.45){\ch}&H^*_{dR,deloc}([*/G])\ar[d]^{\cong}\\
R(G)\ar[r]^{\Tr}&\C[G]^G }\ .$$
\begin{lem}\label{calpunk}
We have 
$$\hat K^*([*/G])\cong \left\{\begin{array}{cc} R(G) & *  = 0\\
\C[G]^G/R(G) & *=1\end{array}\right.\ .$$
\end{lem}
\proof
We use the exact sequence given by  Proposition \ref{prop1}. \hB
Note that $\Tr\colon R(G)\otimes_\Z\C\to \C[G]^G$ is an isomorphism so that,
\textcolor{black}{with $\T=\C/\Z$,}
$$\C[G]^G/R(G) \cong R(G)\otimes_\Z \T\ .$$ 
It restricts to an isomorphism
$R(G)_\R:=R(G)\otimes_\Z \R\stackrel{\sim}{\to}\Omega_\R(L[*/G])\subset \C[G]^G$.
\begin{kor}
We have 
$$\hat K_\R^*([*/G])\cong \left\{\begin{array}{cc} R(G) & *  = 0\\
R(G)_\R/R(G)\cong \R(G)\otimes_\Z \R/\Z & *=1\end{array}\right.\ .$$
\end{kor}

\textcolor{black}{ \subsection{Calculation for $[M/G]$ if $G$ acts trivially}
 \label{sec:trivial_action}}

 \subsubsection{}
\textcolor{black}{ Let $G$ be a finite group and $M$ a compact manifold. We consider the
 orbifold $[M/G]$ where $G$ acts trivially on $M$. Then $L[M/G]=[M\times
 G/G]$ where $G$ acts by conjugation on itself (and trivially on $M$). Therefore
 \begin{equation*}
   \Omega(L[M/G])=\Omega(M\times G)^G = \Omega(M)\otimes \C[G]^G;\qquad
   H_{dR,deloc}(L[M/G])\cong H_{dR}(M)\otimes \C[G]^G.
 \end{equation*}
}

Observe that $[M/G]=M\times [*/G]$. From the cup product of
  Section \ref{jkjdkdqwdqdwd} we therefore get a product
  \begin{equation*}
  R(G)\otimes \hat K(M)=    \hat K^0([*/G])\otimes \hat K(M)\to \hat K([M/G]),
\end{equation*}
compatible along $R$, $I$, and $\ch_{dR}$ with the corresponding maps on forms
and on ordinary $K$-theory
\begin{equation}\label{eq:tensprod}
  R(G)\otimes K(M)\to K([M/G])\ ;\quad \C[G]^G\otimes
  \Omega(M) {\cong} \Omega([*/G])\otimes \Omega(M)\to \Omega([M/G]).
\end{equation}
Because the maps in \eqref{eq:tensprod} are isomorphisms and {$R(G)$ is a free $\Z$-module the  
exact sequence \eqref{eq:LES} shows by the $5$-lemma \textcolor{black}{that $R(G)\otimes \hat K(M)\to \hat
K([M/G])$ is} an isomorphism, as well.}

\section{Push-forward  and $\cup$-product}\label{idwiqdqwwd}

\subsection{Equivariant $K$-orientation}

\subsubsection{}\label{fuwefiu} 

The notion of a $Spin^c(n)$-reduction of an $SO(n)$-principal bundle
extends directly from the smooth case to the orbifold case using the appropriate
notions of principal bundles in the realm of stacks which we explain in the following. 
Let $G$ be a Lie group. Then we can form the quotient stack
$BG:=[*/G]$. By definition (see \ref{stack1}), the evaluation $BG(T)$ on a test manifold $T$ is the groupoid of $G$-principal bundles on $T$.  A $G$-principal bundle on a  stack $X$ is then, by definition,
a morphism $p\colon X\to BG$. Its underlying fibre bundle is determined by the right pull-back square of $$\xymatrix{T\times_{X}P\ar@{.>}[r]\ar[d]&P\ar[d]\ar[r]&{*}\ar[d]\\T\ar@{.>}[r]\ar@{:>}[ur]&X\ar[r]^{p}\ar@{:>}[ur]&BG}\ .$$
There is an equivalent definition of a $G$-principal bundle on $X$ as a collection
of $G$-principal bundles 
 $(T\times_{X}P\to T)_{g\colon T\to X}$  for all maps $g\colon T\to X$ from smooth manifolds 
 which is compatible with further pull-backs along maps of manifolds over $X$. 
 These bundles are obtained by further pull-backs as indicated by the left square of the diagram above.

If $H\to G$ is a homomorphism of Lie groups, then we get a map
of quotient stacks $BH\to BG$. For a test manifold $T\in \Mf$ it is the functor
$BH(T)\to BG(T)$ which maps
 the $H$-bundle $(P\to T)\in BH(T)$ to the $G$-bundle $(P\times_{H}G\to T)\in BG(T)$.
By definition, an $H$-reduction of the $G$-principal  bundle
$p\colon X\to BG$ is a pair $(q,\psi)$
$$\xymatrix{X\ar@{=}[d]\ar[r]^{q}&BH\ar[d]\\X\ar@{:>}[ur]^{\psi}\ar[r]^{p}&BG}$$
of a morphism of stacks $q$ and a two-morphism $\psi$ filling the above square.
Let us spell out this definition in terms of  compatible collections of principal bundles for
maps $T\to X$ from smooth manifolds. The $H$-reduction of $p$ is then given by a compatible collection of $H$-principal bundles $(Q\to T)_{T\to X}$ (this is the datum of $q$) together with a collection of isomorphisms of $G$-principal bundles $(Q\times_{H}G\to T)\stackrel{\sim}{\to} (P\to T)$
compatible with pull-backs (this is the datum of $\psi$).

For later use let us discuss connections at this point. 
We can form the stack
$BG_{\nabla}$ of $G$-principal bundles with connections whose evaluation $BG_{\nabla}(T)$ is the groupoid of pairs $(P\to T,\nabla^{P})$ of $G$-principal bundles $P\to T$ with a connection $\nabla^{P}$, and whose morphisms are connection-preserving isomorphisms of principal bundles.
There is a natural morphism of stacks $BG_{\nabla}\to BG$ which forgets the connection.
For a homomorphism of Lie groups $\phi\colon H\to G$ we get a commutative diagram
$$\xymatrix{BH_{\nabla}\ar[r]^{B_{\nabla}\phi}\ar[d]&BG_{\nabla}\ar[d]\\ BH\ar@{:>}[ur]^{\id}\ar[r]^{B\phi}&BG}$$
where $B_{\nabla}\phi\colon B_{\nabla}H(T)\to B_{\nabla}G(T)$ is the functor which maps the pair
$(Q\to T,\nabla^{Q})$ to the pair $(Q\times_{H}G,\nabla^{Q\times_{H}G})$, where
$ \nabla^{Q\times_{H}G}$ is the connection induced from $\nabla^{Q}$.

If $X$ is a stack, then by definition a connection on the $G$-principal bundle $p\colon X\to BG$ is a lift 
$$\xymatrix{X\ar@{=}[d]\ar[r]^{p_{\nabla}}&B_{\nabla}G\ar[d]\\X\ar@{:>}[ur]\ar[r]^{p}&BG}$$
of $p$.
We will often use the notation $\nabla^{\cdots}$ for a connection where the decoration $\cdots$ 
should indicate its origin.

Given a diagram
$$\xymatrix{X\ar@/^1cm/[rr]^{p_{\nabla}} \ar@{=}[d]\ar[r]^{q_{\nabla}}&BH_{\nabla}\ar[r]^{B_{\nabla}\phi}\ar[d]&BG_{\nabla}\ar[d]\\X\ar@{:>}[ur]\ar[r]^{q}& BH\ar[r]^{B\phi}\ar@{:>}[ur]&BG}$$
we say that the $H$-connection $q_{\nabla}$ reduces to the $G$-connection $p_{\nabla}$, or that
$q_{\nabla}$ extends $p_{\nabla}$.

\subsubsection{}
 
Let $p\colon W\to B$ be \textcolor{black}{representable morphism of stacks which is a} locally trivial fibre bundle   with $n$-dimensional fibres. Its vertical bundle $T^vp$  is an  $n$-dimensional real vector bundle. 
Its frame bundle $\Fr(T^{v}\pi)$ is a $GL(n,\R)$-principal bundle.
Let $GL(n,R)_{1}\subset GL(n,\R)$ be the connected component of the identity.
An orientation of $T^{v}\pi$ is, by definition, a $GL(n,R)_{1}$-reduction of $\Fr(T^{v}\pi)$.
A choice of a vertical metric $g^{T^vp}$ is equivalent to a further
$SO(n)$-reduction of the frame bundle which we denote by
$SO(T^{v}\pi)\to W$.

A map between smooth manifolds is called $K$-oriented if its stable normal
bundle is equipped with a  $K$-theory Thom class. It is a well-known fact
\cite{MR0167985}  that the choice of a $Spin^c$-structure 
on the stable normal bundle determines a $K$-orientation, and the $K$-orientability is equivalent to the existence of a $Spin^{c}$-structure. Note that isomorphism classes of choices of $Spin^c$-structures on $T^vp$ and the stable normal bundle of $p$ are in bijective correspondence.

In the equivariant or orbifold situation this is more complicated. For the
purpose of the present paper we will work with vertical structures along the
morphisms $p\colon W\to B$. 

Let $p\colon W\to B$ be a  representable morphism of stacks which is a locally trivial fibre bundle    with $n$-dimensional fibres.
\begin{ddd}\label{topkor}
A topological $K$-orientation of $p$ is a $Spin^c(n)$-reduction of the $SO(n)$-principal bundle  $SO(T^vp)\to W$.
\end{ddd} 
In general, the stack $W$ may decompose as a sum of substacks
$W=\bigsqcup_{\alpha} W_{\alpha}$ such that the restriction  $p_{\alpha}\colon W_{\alpha}\to B$ of $p$ is a bundle with fibre dimension $n_{\alpha}$. A topological $K$-orientation of $p$ in this case is a collection of topological $K$-orientations for the components $p_{\alpha}\colon W_{\alpha}\to B$. The same idea will be
applied without mentioning for other constructions below.

\subsubsection{}\label{ldefr}

If $f\colon E\to A$ is a locally trivial fibre bundle in manifolds, then 
the choice of a vertical metric $g^{T^{v}f}$ and a horizontal distribution
$T^{h}f$ naturally induce a connection $\nabla^{T^{v}f}$ on the vertical bundle which restricts to the Levi-Civita connection on the fibres, see \cite[Chapter 9]{bgv}. 
The construction of $\nabla^{T^{v}f}$ is compatible with pull-back along maps
$A^{\prime}\to A$. Hence it extends to the case of locally trivial fibre bundles in stacks.
The natural construction of the connection on $T^{v}f$  can be formulated as a construction of the upper two squares  of the diagram 
 $$\xymatrix{&&B_{\nabla}O(n,\R)\ar[r]&BO(n,\R)\ar[r]&BGL(n,\R)\\W\ar@{.>}[urr]^{\nabla^{T^{v}p}}\ar@{.>}[rr] \ar[d]^{p}&&\univE(n,g^{T^{v}},T^{h})\ar[r]\ar[d]\ar[u]\ar@{:>}[ur]&\univE(n,g^{T^{v}})\ar[r]\ar[d]\ar@{:>}[ur]\ar[u]&\univE(n)\ar[d]\ar[u]\\B\ar@{.>}[rr]^{g^{T^{v}p},T^{h}p}\ar@{:>}[urr]\ar@{.>}@/_1.5cm/[rrrr]^{p}&&\BDiff(n,g^{T^{v}},T^{h})\ar@{:>}[ur]\ar[r]&\BDiff(n,g^{T^{v}})\ar@{:>}[ur]\ar[r]&\BDiff(n)\\&&&\ar@{:>}[u]&}\ ,$$
 where $\BDiff(n)$ denotes the stack of locally trivial fibre bundles with
 $n$-dimensional fibres, and the remaining notation is
  self-explaining.

Let $p\colon W\to B$ \textcolor{black}{be a} representable morphism of stacks
which is a  locally trivial fibre bundle with $n$-dimensional fibres. 
If we choose a vertical metric $g^{T^{v}p}$ and a horizontal distribution $T^hp$, then by the above construction we get a connection
$\nabla^{T^vp}$ which restricts to the Levi-Civita connection along the
fibres. This is indicated in the left part of the diagram above. 

If $p$ is oriented, then by restriction the connection $\nabla^{T^vp}$ can be
considered as an $SO(n)$-principal bundle connection on the frame bundle
$SO(T^vp)$.
Given a topological $K$-orientation of $p$, i.e. a 
$Spin^{c}(n)$-reduction of $SO(T^{v}p)$,   
 we can choose a $Spin^c$-reduction $\tilde \nabla$ of $\nabla^{T^vp}$ (see \ref{fuwefiu}).
Observe that, in contrast to the $Spin$-case, $\tilde \nabla$ is not unique.

\subsubsection{}\label{origeo}

The $Spin^c$-reduction of $\Fr(T^vp)$ determines a spinor bundle $S^c(T^vp)$,
and the choice of $\tilde \nabla$ turns $S^c(T^vp)$ into a family of Dirac
bundles. 
In this way the choices of the $Spin^c$-structure and the geometric
structures $(g^{T^vp},T^hp,\tilde \nabla)$ turn
$p\colon W\to B$ into a geometric family $\cW$.

We define the closed form
\begin{equation}
 \hA_\rho^c(\tilde \nabla):=I(\cW)\in \Omega_{{\R}}(LW)\ ,            \label{eq:Ahatc}
 \end{equation}
see Subsection \ref{uzu1}  for a description of
 the form $I(\cW)$.
Its cohomology class will be denoted by
$\hA^{{c}}_\rho(LW)\in H_{dR}(LW)$.

\subsubsection{}

The dependence of the form $\hA_{\rho}^c(\tilde \nabla)$ on the data is described in terms of the transgression form.
Let $(g_i^{T^vp},T_i^hp,\tilde \nabla_i)$, $i=0,1$, be two choices of geometric data. Then we can choose
geometric data $(\hat  g^{T^vp},\hat T^hp, \hat {\tilde \nabla})$ on $\hat p=\id_{[0,1]}\times p\colon [0,1]\times W\to [0,1]\times B$
(with the induced $Spin^c$-structure on $T^v\hat p$) which restricts to
$(g_i^{T^vp},T_i^hp,\tilde \nabla_i)$ on $\{i\}\times B$ for $i=0,1$.
The class
$$\tilde \hA_\rho^c(\tilde \nabla_1,\tilde \nabla_0):=\int_{[0,1]\times LW/LW}\hA_\rho^c(\hat{\tilde \nabla})\in \Omega_{{\R}}(LW)/\im(d)$$ 
is independent of the extension and satisfies
\begin{equation}\label{eqq7}
d\tilde \hA_\rho^c(\tilde \nabla_1,\tilde\nabla_0)=\hA_\rho^c(\tilde \nabla_1)-\hA_\rho^c(\tilde \nabla_0)\ .
\end{equation}
\begin{ddd}\label{uzu4}
The form $\tilde \hA_\rho^c(\tilde \nabla_1,\tilde \nabla_0)$ is called the transgression form.
\end{ddd}
Note that we have the identity
\begin{equation}\label{eq7}
\tilde \hA_\rho^c(\tilde \nabla_2,\tilde \nabla_1)+\tilde \hA_\rho^c(\tilde \nabla_1,\tilde\nabla_0)=\tilde \hA_\rho^c(\tilde \nabla_2,\tilde\nabla_0)\ .
\end{equation}
As a consequence we get  
\begin{equation}\label{eq8}
\tilde \hA_\rho^c(\tilde \nabla,\tilde \nabla)=0\ ,\quad \tilde \hA_\rho^c(\tilde \nabla_1,\tilde \nabla_0)=-\hA_\rho^c(\tilde \nabla_0,\tilde \nabla_1)\ .
\end{equation}

\subsubsection{}\label{pap201}

We can now introduce the notion of a differential $K$-orientation of a
representable map $p\colon W\to B$ between orbifolds which is a locally
trivial fibre bundle.
We fix an 
 underlying topological $K$-orientation of $p$ (see Definition \ref{topkor}) which is given by a $Spin^c$-reduction of $SO(T^vp)$ after choosing an orientation and a metric on $T^vp$.  

We consider the set $\cO$ of tuples
$(g^{T^vp},T^hp,\tilde \nabla,\sigma)$ where the first three entries have the
same meaning as above (see \ref{ldefr}), and
$\sigma\in \Omega^{odd}(LW)/\im(d)$. We introduce a relation $o_0\sim o_1$ on $\cO$:
Two tuples
$(g_i^{T^vp},T_i^hp,\tilde \nabla_i,\sigma_i)$, $i=0,1$ are related if and only if
$\sigma_1-\sigma_0=\tilde \hA^c_\rho(\tilde \nabla_1,\tilde \nabla_0)$.
We claim that
$\sim$ is an equivalence relation. In fact, symmetry and reflexivity follow from (\ref{eq8}), while transitivity
is a consequence of (\ref{eq7}).

\begin{ddd}\label{smmmmzuz}
The set of differential $K$-orientations which refines a fixed underlying topological $K$-orientation
of $p\colon W\to B$ is the set of equivalence classes $ \cO/\sim$.
\end{ddd}

\textcolor{black}{Now} $\Omega^{odd}(LW)/\im(d)$ acts on the set of differential 
$K$-orientations. 
If $\alpha\in \Omega^{odd}(LW)/\im(d)$ and $(g^{T^vp},T^hp,\tilde \nabla,\sigma)$ represents 
a differential $K$-orientation, then the translate of this orientation by $\alpha$ is represented by
$(g^{T^vp},T^hp,\tilde \nabla,\sigma+\alpha)$.
As a consequence of (\ref{eq7}) we get:
\begin{kor}
The set of differential $K$-orientations refining a fixed underlying topological $K$-orientation is a torsor over
$\Omega^{odd}(LW)/\im(d)$, \textcolor{black}{i.e.\ the action is free and
  transitive}.
\end{kor}

If $o=(g^{T^vp},T^hp,\tilde \nabla,\sigma)\in \cO$ represents a differential
$K$-orientation then we will write
\begin{equation}\label{may20111}\hA^c(o):=\hA_\rho^c(\tilde \nabla)\ ,\quad \sigma(o):=\sigma\ .\end{equation}

\subsection{Definition of the push-forward}

\subsubsection{}\label{pap8}

We consider a representable morphism  $p\colon W\to B$ of
orbifolds which is a proper submersion,  or
equivalently, a locally trivial fibre bundle with compact smooth fibres. We
fix a topological $K$-orientation for $p$.
Let $o=(g^{T^vp},T^hp,\tilde \nabla,\sigma)$ represent a differential $K$-orientation which refines the given  topological one. 
To every geometric family $\cE$ over $W$ we now associate a geometric family $p_!\cE$ over $B$ as follows.
 
Let
$\pi\colon E\to W$ denote the underlying fibre bundle of $\cE$ which comes with the  geometric data $g^{T^v\pi}$, $T^h\pi$ and the family of Dirac bundles $(V,h^V,\nabla^V)$. 
Then the underlying fibre bundle of $p_!\cE$ is given by the composition
$$q:=p\circ\pi\colon E\rightarrow B\ .$$

In the following,
when we talk about horizontal bundles or connections  we think of compatible collections of horizontal bundles or connections
for all pull-backs along maps  from smooth manifolds to the respective base as explained in \ref{stack22}. 
So in the technical sense the following natural constructions are applied to all these pull-backs simultaneously.

The horizontal bundle of $\pi$ admits a 
decomposition $T^h\pi\cong \pi^* T^vp\oplus \pi^*T^hp$,
where the isomorphism is induced by $d\pi$. 
We define $T^hq\subseteq T^h\pi$ such that $d\pi\colon T^hq\cong \pi^* T^hp$.
Furthermore we have an identification
$T^vq=T^v\pi\oplus \pi^* T^vp$.
Using this decomposition we define the vertical metric
$g^{T^vq}:=g^{T^v\pi}\oplus \pi^* g^{T^vp}$.
These structures give a {metric} connection $\nabla^{T^vq}$ which in general differs from the sum
$\nabla^{T^v\pi}\oplus \pi^*\nabla^{T^vp}\textcolor{black}{=:\nabla^\oplus}$.

The orientations of $T^v\pi$ and $T^vp$ induce an orientation of $T^vq$.

Finally we must construct the Dirac bundle $p_!\cV\rightarrow E$.
Locally on $E$ we can choose a   $Spin^c$-structure on $T^v\pi$ with
spinor bundle $S^c(T^v\pi)$ and with a  $Spin^c$-connection 
$\tilde \nabla_\pi$ which refines the connection $\nabla^{T^v\pi}$.
We define the twisting bundle
$$Z:=\Hom_{\Cliff(T^v\pi)}(\textcolor{black}{S^c}(T^v\pi),V)\ .$$
The connections $\tilde \nabla_\pi$ and $\nabla^V$ induce a connection $\nabla^Z$.

The local $Spin^c$-structure of $T^v\pi$ together with the $Spin^c$-structure
of $T^vp$ induce a $Spin^c$-structure on $T^vq\cong T^v\pi\oplus \pi^*T^vp$.
\textcolor{black}{We get an induced connection $\tilde \nabla^{\oplus}$ from
  $\tilde\nabla_\pi$ and $\tilde\nabla^{T^vp}$} which refines the direct sum
connection {$\nabla^{\oplus}$}.
Let $$\omega:= \textcolor{black}{\nabla^{T^vq}}- \nabla^{\oplus}\in
  \Gamma(E,\Lambda^1(T^vq)^*\otimes \End(T^vq)^{a})$$ 
be the difference of the two metric connections, a one form with coefficients in antisymmetric endomorphisms.
 We define  
$$\tilde \nabla_{q}:=\tilde \nabla^{\oplus}+\frac{1}{2}c(\omega)\ .$$ 
This is a $Spin^c$-connection on $T^vq$ which
refines $\nabla^{T^vq}$ and has the same central curvature as $\tilde \nabla^{\oplus}$.
Locally we can define the family of Dirac bundles
$p_!V:=S(T^vq)\otimes Z$. \textcolor{black}{One can show}
 that this bundle is well-defined independent of the choices of local $Spin^c$-structure and therefore a globally defined family of Dirac bundles.

 \begin{remark}
   Note that the notion of locality in the realm of orbifolds is more
   complicated than it might appear at first glace. To say that we choose
   a local $Spin^c$-structure means that we use an orbifold atlas $A\to B$ and
   choose a $Spin^c$-structure after pulling the family back to $A$. Thus in
   particular we do not (and can not) require that it is equivariant with
   respect to the local automorphism groupoid $A\times_BA\Rightarrow A$.
   Therefore our twisting bundle $Z$ is not equivariant, too.  On the other
   hand, the tensor product $S^c(T^vq)\otimes Z$ is completely canonical and
   thus is equivariant.
 \end{remark}

 \begin{ddd}\label{ddd7771}
Let
$p_!\cE$ denote the geometric family given by $q\colon E\to B$ and $p_!V\to E$ with the geometric structures
defined above.
\end{ddd}

\subsubsection{}\label{pap55}

Let $p\colon W\to B$ be a representable morphism  between orbifolds which is a
locally trivial fibre bundle with compact fibres and equipped with a
differential $K$-orientation represented by $o$.
In \ref{pap8} we have constructed for each geometric family $\cE$ over $W$ a push-forward $p_!\cE$.
Now we introduce a parameter $a\in (0,\infty)$ into this construction.
\begin{ddd}
For $a\in (0,\infty)$ we define  the geometric family $p_!^a\cE$ as in \ref{pap8} with the only difference that 
the metric on $T^vq=T^v\pi\oplus \pi^*T^vp$ is given by $g_a^{T^vq}= a^2g^{T^v\pi}\oplus  \pi^*g^{T^vp}$.
\end{ddd}
The family of geometric families $p^a_!\cE$ is called the adiabatic deformation of $p_!\cE$.
There is a natural way to define a geometric family $\cF$ on $(0,\infty)\times B$ such that its restriction
to $\{a\}\times B$ is $p_!^a\cE$. In fact, we define $\cF:=(\id_{(0,\infty)}\times p)_!((0,\infty)\times \cE)$ with the exception that we take the appropriate vertical metric.

Although the vertical metrics of $\cF$ and $p^a_!\cE$ collapse as $a\to 0$ the induced connections and the curvature tensors
on the vertical bundle $T^vq$ converge and simplify in this limit. This fact is heavily used in local index theory, and we refer to \cite[Sec 10.2]{bgv} for details. In particular, the integral
$$\tilde \Omega(a,\cE):=\int_{(0,a)\times LB/LB}\Omega(\cF)$$
converges, and we have (see Definition  \ref{may20112} and (\ref{may20111}) for notation)
\begin{equation}\label{eq88}
\textcolor{black}{\Omega(p_!^a\cE)\xrightarrow{a\to
    0}}\int_{LW/LB}\hA^c(o)\wedge \Omega(\cE)\ ,\quad
\Omega(p_!^a\cE)-\int_{LW/LB}\hA^c(o)\wedge \Omega(\cE)=d\tilde \Omega(a,\cE)\
.
\end{equation}

\subsubsection{}

Let $p\colon W\to B$ be a representable morphism  between presentable compact
orbifolds  which is 
a locally trivial fibre bundle with compact fibres and equipped with a
differential $K$-orientation represented by  $o$. 
 We now start with the construction of the push-forward $p_!\colon \hat K(W)\to \hat K(B)$. For $a\in (0,\infty)$
and  a cycle $(\cE,\rho)$ we define 
\begin{equation}\label{eq300}
\hat p^a_!(\cE,\rho):=\left[p^a_!\cE\:\:,\:\:\int_{LW/LB}  \hA^c(o)\wedge \rho + \tilde
\Omega(a,\cE)+\int_{LW/LB}\sigma(o) \wedge R([\cE,\rho])\right]\in \hat K(B)\ . 
\end{equation} 
Since $\hA^c(o)$ and $R([\cE,\rho])$ are closed forms, the map
$$\Omega(LW)/\im(d)\ni \rho\mapsto  \int_{LW/LB}  \hA^c(o)\wedge \rho\in
\Omega(LB)/\im(d)$$
and the element
$$  \int_{LW/LB}\sigma(o) \wedge
R([\cE,\rho])\in \Omega(LB)/\im(d)$$ are well-defined. It immediately follows
from the definition that 
$p_!^a\colon G(W)\to \hat K(B)$ is a homomorphism of semigroups
(\textcolor{black}{$G(W)$ was
introduced in Definition \ref{may20113}}).

\subsubsection{}\label{pap101}

The homomorphism $p^a_!\colon G(W)\to \hat K(B)$ commutes with pull-back.
More precisely, let $f\colon B^\prime\to B$ be a morphism of  presentable
compact orbifolds.
Then we define the submersion $p^\prime\colon W^\prime\to B^\prime$ by the two-cartesian
diagram 
$$\xymatrix{W^\prime\ar[d]^{p^\prime}\ar[r]^F&W\ar[d]^p\\B^\prime\ar@{:>}[ur]\ar[r]^f&B}\ .$$
The  differential of the morphism $F\colon W^{\prime}\to W$
induces an isomorphism
$dF\colon T^vW^\prime\stackrel{\sim}{\to} F^*T^vW$.
Therefore the metric, the orientation, and the $Spin^c$-structure of $T^v\pi$ induce by pull-back corresponding structures on $T^vp^\prime$. 
We have furthermore an induced horizontal distribution $T^hp^\prime$. 
Finally we set $\sigma^\prime:=LF^*\sigma \in \Omega^{*}(LW^{\prime})/\im(d)$.
The representative of a differential $K$-orientation given by these structues will be denoted by 
$o^\prime:=f^*o$. An inspection of the definitions shows:
\begin{lem}\label{djdgejwdewd}
The pull-back of representatives of differential $K$-orientations preserves  equivalence and hence
 induces a pull-back of differential $K$-orientations. 
\end{lem}
Recall from \ref{origeo} that the representatives $o$ and $o^\prime$ of the differential $K$-orientations enhance $p$ and $p^\prime$ to geometric families $\cW$ and $\cW^\prime$. We have $f^*\cW\cong \cW^\prime$.

Note that we have
$LF^*\hA^c(o)=\hA^c(o^\prime)$.
If $\cE$ is a geometric family over $W$, then an inspection of the definitions shows that
$f^*p_!(\cE)\cong p^\prime_!(F^*\cE)$. The following lemma now follows immediately from the definitions
\begin{lem}\label{lem100}
We have $f^*\circ \hat p^a_!=\hat{
p^\prime}^a_!\circ F^*\colon G(W)\to \hat K(B^\prime)$.
\end{lem}

\subsubsection{}

\begin{lem}
The class $\hat p_!^a(\cE,\rho)\in \hat K(B)$ does not depend on $a\in (0,\infty)$.
\end{lem}
 \proof The proof can be copied literally from \cite[Lemma
 3.11]{bunke-2007}. \hB

In view of this Lemma we can omit the superscript $a$ and write
$p_!(\cE,\rho)$ for $p_!^a(\cE,\rho)$.

\subsubsection{}\label{pap66}

Let $\cE$ be a geometric family over $W$ which admits a taming $\cE_t$.
Recall that the taming is given by a family of smoothing operators $(Q_w)_{w\in W}$.
%
%

The family of operators along the fibres of $p_!\cE$ induced by $Q$ is not a taming of $p_!^a\cE_t$ since it is not given by a smooth integral kernel but rather by a family of fibrewise smoothing operators.  Nevertheless it can be used in the same way as a taming in order to define e.g. the $\eta$-forms which we will denote by $\eta(p_!^a\cE_t)$. 
To be precise, we add the term $\chi(ua^{-1})ua^{-1}Q$ to the rescaled
superconnection $A_u(p^a_!\cE)$, where $\chi$ vanishes near zero and
is equal to $1$ on $[1,\infty)$. This means that we switch on
 $Q$ at time $u\sim a$, and we rescale it in the same
way as the vertical part of the Dirac operator.
In this situation we will speak of a generalized taming.
We can control the behaviour of
$\eta(p^a_!\cE_t)$ in the adiabatic limit $a\to 0$.
\begin{theorem}\label{adia1}
$$\lim_{a\to 0}\ \eta(p^a_!\cE_t)=\int_{LW/LB} \hA^c(o)\wedge \eta(\cE_t)\
.$$
 \end{theorem}
\proof
The proof of this theorem can be obtained by combining standard methods of equivariant local index theory with the adiabatic techniques developed by the school of Bismut. Note again,
that the analysis here is fibrewise. By pull-back along morphisms $T\to B$ from smooth manifolds $T$ we reduce to the case of a fibre-bundle over a smooth manifold.
\hB 

Since the geometric family $p_!^a\cE$ admits a generalized taming  it follows that 
$\ind(p_!^a\cE)=0$. Hence we can also choose a taming $(p_!^a\cE)_t$.
The latter choice together with the generalized taming
induce a generalized boundary taming of the family
$p_!^a\cE\times [0,1]$ over $B$.
 We have, as in \cite[Lemma 3.13]{bunke-2007}, the following assertion.
\begin{lem}\label{lem33}
The difference of $\eta$-forms 
$\eta((p_!^a\cE)_t)-\eta(p^a_!\cE_t)$ 
is closed. Its de Rham cohomology class satisfies
$$[\eta((p_!^a\cE)_t)-\eta(p^a_!\cE_t)]\in \ch_{dR}(K(B))\ .$$
\end{lem}

\subsubsection{}\label{sec:we-now-show}

We now show that
$p_!\colon G(W)\to \hat K(B)$ passes through the equivalence relation $\sim$.
Since $p_!$ is additive it suffices by Lemma \ref{lem3} to show the following assertion.
\begin{lem}\label{pass}
If $(\cE,\rho)$ is paired with $(\tilde \cE,\tilde \rho)$,
then $\hat p_!(\cE,\rho)=\hat p_!(\tilde \cE,\tilde\rho)$.
\end{lem}
\proof
The proof can be copied from \cite[Lemma 3.14]{bunke-2007} since it again only
uses formal properties of local index- and $\eta$-forms which hold true in the
present case.  \hB

\subsubsection{}

 We let
\begin{equation}\label{uiqwduqwdqwdwqdwqdd}
\hat p_!\colon \hat K(W)\to \hat K(B)
\end{equation} denote the map induced
by  the construction (\ref{eq300}).

\begin{ddd}\label{def313}
We define the integration of forms 
$p^o_!\colon \Omega(LW)\to \Omega(LB)$
by
$$p^o_!(\omega)=\int_{LW/LB} (\hA^c(o)-d\sigma(o))\wedge \omega$$
\end{ddd}
Since $  \hA^c(o)-d\sigma(o) $ is closed  we also have a factorization
 $$p_!^o\colon \Omega(LW)/\im(d)\to \Omega(LB)/\im(d)$$  denoted by the same
 symbol.

Our constructions of the homomorphisms
$$\hat p_!\colon \hat K(W)\to \hat K(B)\ ,\quad p^o_!\colon \Omega(LW)\to \Omega (LB)$$
involve an explicit choice of a representative $o=(g^{T^vp},T^hp,\tilde \nabla,\sigma)$
of the differential $K$-orientation lifting the given topological $K$-orientation of
$p$. But 
both push-forward maps are actually independent of the choice of the representative.

\begin{lem}\label{pap99}
The homomorphisms $\hat p_!\colon \hat K(W)\to \hat K(B)$ and 
$p^o_!\colon \Omega(W)\to \Omega (B)$
only depend on the differential $K$-orientation represented by $o$.
\end{lem}
\proof The proof can be copied literally from \cite[Lemma 3.17]{bunke-2007}. \hB

\subsubsection{}

Let $p\colon W\to B$ be a representable morphism between orbifolds which is a
locally trivial fibre bundle with closed fibres and equipped  with a
differential $K$-orientation represented by $o$.
 We now have constructed the  homomorphism (\ref{uiqwduqwdqwdwqdwqdd}).
In the present paragraph we \textcolor{black}{obtain} the compatibilty of this
construction with the curvature $R\colon \hat K\to \Omega_{d=0}$
\textcolor{black}{by copying the calculations from \cite[Lemma 3.16]{bunke-2007}:}

\begin{lem}\label{lem24}
For $x\in \hat K(W)$ we have
$$R(\hat p_!(x))=p^o_!(R(x))\ .$$
\end{lem}

\subsubsection{}

Let $p\colon W\to B$ be a representable morphism between orbifolds which is a locally trivial fibre bundle with closed fibres and equipped  with a topological $K$-orientation.
We choose a differential $K$-orientation which refines
the given topological $K$-orientation.  In this case we say that $p$ is differentiably $K$-oriented.
\begin{ddd}\label{ddd1}
We define the push-forward
$\hat p_!\colon \hat K(W)\to \hat K(B)$ to be the map induced by (\ref{eq300})
for some choice of a representative of the differential $K$-orientation \textcolor{black}{
and $a > 0$.}
\end{ddd}
We also have well-defined maps
$$p^o_!\colon \Omega(LW)\to \Omega(BL),\quad p^o_!\colon \Omega(LW)/\im(d)\to
\Omega(LB)/\im(d)\ .$$
Let us state the result about the compatibility of $p_!$ with the structure maps of differential $K$-theory as follows.
\begin{prop}\label{mainprop}
The following diagrams commute:
\begin{equation}\label{uppersq}
  \begin{CD}
    K(W) @>{\ch_{dR}}>> \Omega(LW)/\im(d) @>a>>\hat K(W) @>{I}>> K(W)\\
    @VV{p_!}V     @VV{p^o_!}V     @VV{\hat p_!}V     @VV{p_!}V \\
    K(B) @>{\ch_{dR}}>> \Omega(LB)/\im(d) @>a>>\hat K(B) @>{I}>> K(B)\\
  \end{CD}
\end{equation}
\begin{equation}\label{lowersq}
  \begin{CD}
    \hat K(W) @>{R}>> \Omega_{d=0}(LW)\\
    @VV{\hat p_!}V      @VV{p_!^o}V \\
    \hat K(B) @>{R}>> \Omega_{d=0}(LB)
  \end{CD}
\end{equation}

\end{prop}
\proof We can copy the proof of \cite[Proposition 3.19]{bunke-2007} literally
since it only uses formal properties of the objects involved which hold true
in the present situation. \hB

\subsection{Functoriality}

\subsubsection{}\label{pap200}

We now discuss the functoriality of the push-forward with respect to
iterated fibre bundles.
Let $p\colon W\to B$ be as before together with a representative of a differential $K$-orientation
$o_p=(g^{T^vp},T^hp,\tilde \nabla_p,\sigma(o_p))$. 
Let $r\colon B\rightarrow A$ be another representable morphism between
presentable compact orbifolds which is a locally trivial fibre bundle with
compact fibres. We assume that it is equipped
with a topological $K$-orientation which is refined by a differential $K$-orientation represented by $o_r:=(g^{T^vr},T^hr,\tilde \nabla_r ,\sigma(o_r))$.

We can consider the geometric family
$\cW:=(W\to B,g^{T^vp},T^hp,S^c(T^vp))$ and apply the construction \ref{pap55} 
in order to define the geometric family $r^a_!(\cW)$ over $A$.
The underlying submersion of this family is
$q:=r\circ p\colon W\to A$. Its vertical bundle has a metric $g_a^{T^vq}$, and
is equipped with a horizontal distribution $T^hq$.
The topological $Spin^c$-structures of $T^vp$ and $T^vr$ induce a topological $Spin^c$-structure on
$T^vq=T^vp\oplus p^*T^vr$. The family of Clifford bundles of $\textcolor{black}{r}_!\cW$ is the
spinor bundle associated to this $Spin^c$-structure. 

In order to understand how the connection $\tilde \nabla_q^a$ behaves as
$a\to 0$ we choose local spin structures on $T^vp$ and $T^vr$. Then we write
$S^c(T^vp)\cong S(T^vp)\otimes L_p$ and $S^c(T^vr)\cong S(T^vr)\otimes L_r$
for one-dimensional twisting bundles \textcolor{black}{ $L_p$, $L_r$ with
  connections}. The two local spin structures induce a local spin structure  on
$T^vq\cong T^vp\oplus p^*T^vr$. We get $S^c(T^vq)\cong S(T^vq)\otimes L_q$
with $L_q:=L_p\otimes p^*L_r$.
The connection $\nabla_q^{a,T^vq}$ converges as $a\to 0$. Moreover,
the twisting connection on $L_q$ does not depend on $a$ at all.
Since $\nabla_q^{a,T^vq}$ and $\nabla_q^{L}$ determine $\tilde \nabla^a_q$ (see \ref{origeo})
we conclude that the connection $\tilde \nabla^a_q$ converges as $a\to 0$. We introduce the following notation for this adiabatic limit:
$$\tilde \nabla^{adia}:=\lim_{a\to 0}\tilde \nabla^a_q\ .$$

\subsubsection{}

We keep the situation described in \ref{pap200}.
 \begin{ddd}\label{def100}
We define the composite
$o^a_q:=o_r\circ_a o_p$ of the representatives of differential $K$-orientations of
$p$ and $r$ 
by
$$o^a_q:=(g_a^{T^vq},T^hq,\tilde \nabla^a_q,\sigma(o^a_q))\ ,$$
where
$$\sigma(o_q^a):=\sigma(o_p)\wedge p^*\hA^c_\rho(o_r)+\hA_\rho^c(o_p)\wedge
p^*\sigma(o_r)- \tilde\hA_\rho^c (\tilde \nabla^{adia},\tilde \nabla^a_q)-d\sigma(o_p)\wedge p^*\sigma(o_r)\ .$$
\end{ddd}
\begin{lem}\label{lem19123}
This composition of representatives of differential $\hat K$-orientations
preserves the equivalence relation introduced in  \ref{pap201} and induces a
well-defined composition of differential $K$-orientations which is independent
of $a$.
\end{lem}
\proof The proof is the same as the one of \cite[Lemma 3.22]{bunke-2007}. \hB

\subsubsection{}\label{tz123}

We consider the composition of the  $K$-oriented locally trivial fibre bundles
$$\xymatrix{W\ar@/_0.5cm/[rr]_q\ar[r]^p&B\ar[r]^r&A}$$
with representatives of differential $K$-orientations $o_p$ of $p$ and $o_r$
of $r$.  We let $o_q:=o_p\circ p_r$ be the composition of differential
$K$-orientations.
These choices define push-forwards $\hat p_!$, $\hat r_!$ and $\hat q_!$ in differential $K$-theory.

\begin{theorem}\label{funktt}
We have the equality of homomorphisms $\hat K(W)\rightarrow \hat K(A)$
$$\hat q_!=\hat r_!\circ \hat p_!\ .$$
\end{theorem}
\proof The proof only depends on the formal properties of transgression forms. It can be copied from \cite[Theorem 3.23]{bunke-2007}. \hB

\subsubsection{}\label{udiduwqdqwdwqoidwiqdwd}
We call a representative $o=(g^{T^vp},T^hp,\tilde \nabla_p,\sigma(o_p))$ of a
differential $K$-orientation of $p\colon W\to B$  real,  if and only if
$\sigma(o_p)\in \Omega_\R^{odd}(LW)/\im(d)$.  Furthermore, we observe that
being real is a property of the equivalence class of $o$. 
If $o$ is real, then it immediately follows from (\ref{eq300}) that the
associated push-forward preserves the real subfunctors, i.e. that
 by restriction  we get  integration  homomorphisms
$$\hat p_!\colon \hat K_\R(W)\to \hat K_\R(W)\ ,\quad \hat p_!^o\colon  \Omega_\R(LW)\to \Omega_\R(LW).$$

\subsection{The cup product}\label{jkjdkdqwdqdwd}

\subsubsection{}
In this section we define and study the cup product 
$$\cup\colon \hat K(B)\otimes \hat K(B)\rightarrow \hat K(B)\ .$$
It turns differential $K$-theory into a functor on compact presentable orbifolds with 
values in $\Z/2\Z$-graded rings.

\subsubsection{}\label{wefwefwefwefwf}

Let $\cE$ and $\cF$ be geometric families over $B$.
The formula for the product involves the product $\cE\times_B\cF$ of geometric families over $B$.
The detailed description of the product is easy to guess, but let us employ the following trick in order 
to give an alternative definition.

The underlying proper submersions of $\cE$ and $\cF$ give rise to a diagram
$$\xymatrix{E\times_BF\ar[d]\ar[r]^\delta&F\ar[d]^p\\E\ar[r]&B}\ .$$
Let us for the moment assume that the vertical
metric, the horizontal distribution, and the orientation of $p$ are complemented by a topological
$Spin^c$-structure together with a $Spin^c$-connection $\tilde \nabla$ as in \ref{pap8}.
The Dirac bundle $\cV$ of $\cF$ has the form
$\cV\cong W\otimes S^c(T^vp) $ for a twisting bundle $W$ with a hermitean metric and unitary connection 
(and $\Z/2\Z$-grading in the even case), which is uniquely determined up to isomorphism.
Let $ p^*\cE\otimes W$ denote the geometric family which is obtained from $p^*\cE$ by twisting its
Dirac bundle with $\delta^*W$.  Then we have
$$\cE\times_B \cF\cong p_!(p^*\cE\otimes W)\ .$$

In the description of the product of geometric families we could interchange the roles of $\cE$ and $\cF$.

If the vertical bundle of $\cE$ does not have a global $Spin^c$-structure, then it has at least a local one. In this case the description above again gives a complete description of the local geometry of $\cE\times_B\cF$ (see the Remark in \ref{pap8}).

\subsubsection{}

We now proceed to the definition of the product in terms of cycles.
In order to write down the formula we assume that the cycles
$(\cE,\rho)$ and $(\cF,\theta)$ are homogeneous of degree
$e$ and $f$, respectively.

\begin{ddd}\label{proddefin}
We define
$$(\cE,\rho)\cup (\cF,\theta):=[\cE\times_B\cF\:\:,\:\:(-1)^e\Omega(\cE)\wedge
\theta +\rho\wedge \Omega(\cF)-(-1)^ed\rho\wedge \theta]\ .$$
\end{ddd}
\begin{prop}\label{prooow}
The product is well-defined. It turns $B\mapsto \hat K(B)$ into a functor from compact presentable orbifolds to unital graded-commutative rings. {By restriction it induces a ring structure on the real subfunctor $\hat K_\R(B)$.}
\end{prop}
\proof  The proof can be copied from \cite[Proposition 4.2]{bunke-2007}  since it only uses formal properties of the involved objects which extend to hold true in the orbifold case.
That the product preserves the real subspace immediately  follows from the
definitions.\hB
\subsubsection{}

In this paragraph we study the compatibility of the cup product in differential $K$-theory with the cup product in topological K-theory and the wedge product of differential forms.
\begin{lem}\label{rring}
For $x,y\in \hat K(B)$ we have 
$$R(x\cup y)=R(x)\wedge R(y)\ ,\quad I(x\cup y)=I(x)\cup I(y)\ .$$
Furthermore, for $\alpha\in \Omega(LB)/\im(d)$ we have
$$a(\alpha)\cup x=a(\alpha\wedge R(x))\ .$$
\end{lem}
\begin{proof}
Straightforward calculation using the definitions and
that $\ind(\cE\times_B\cF)=\ind(\cE)\cup \ind(\cF)$ and
$\Omega(\cE\times_B\cF)=\Omega(\cE)\wedge \Omega(\cF)$.
\end{proof}


\subsubsection{}

Let $p\colon W\rightarrow B$ be a proper submersion with closed fibres with a
differential $K$-orientation.  In \ref{sec:we-now-show} we defined the push-forward
$\hat p_!\colon \hat K(W)\to \hat K(B)$. The explicit formula in terms of cycles is (\ref{eq300}).
The \textcolor{black}{following \emph{projection formula}} states the compatibility of the push-forward
with the $\cup$-product.
\begin{prop} \label{projcl}
Let $x\in \hat K(W)$ and $y\in \hat K(B)$. 
Then
$$\hat p_!(\hat p^*y\cup x)=y\cup \hat p_!(x)\ .$$
\end{prop}
The proof can again be copied from \cite[Proposition 4.5]{bunke-2007} for the
same reason as in the case of Proposition \ref{prooow}. \hB

 \subsubsection{}\label{ex:int_on_trivial_action}
 We continue the example started in \ref{sec:trivial_action}.
    Let $G$ be a finite group, $M$ a smooth \textcolor{black}{compact} $Spin^c$-manifold and consider the
    stack $[M/G]$ where $G$ acts trivially on $M$. Let $\hat p$ be a differential K-orientation of the projection $p 
     \colon M\to *$ and pull it back to a 
        differential K-orientation $\hat p_G$ of $p_{G}\colon [M/G]\to [*/G]$ along the
    map $\pi\colon [*/G]\to *$. 
    \begin{lem} We have a commutative diagram
    \begin{equation}\label{neu890}
      \xymatrix{
        R(G)\otimes \hat K(M)\ar[r]^{\cong}\ar[d]^{ \id\otimes \hat p_! } &  \hat K([M/G])\ar[d]^{(\hat{p}_G)_!}\\        R(G)\otimes \hat K(*)\ar[r]^{\cong}&  \hat K([*/G]).
      }
    \end{equation}
  Here we use the identification $R(G)\otimes \hat K(M)\cong \hat K([M/G])$ of
  Section \ref{sec:trivial_action}. We get corresponding commutative diagrams
  for the integration of forms and in topological K-theory.\end{lem}
   \begin{proof}
    Let $\pi_M\colon [M/G]\to M$  be the projection and consider elements
    $x\in 
    R(G)\cong\hat K^0([*/G])$ and $y\in \hat K(M)$. Then the corresponding element in
    $\hat K([M/G])$ is $p_G^*x\cup \pi_M^*y$. Therefore, by the projection
    formula (Proposition
    \ref{projcl}) and naturality along pull-backs (Lemma \ref{lem100}) we get
    \begin{equation*}
      (\hat p_G)_!( p_G^*x\cup\pi_M^*y) = x\cup (\hat p_G)_!(\pi_M^*y) =x\cup
      \pi^*\hat p_! y\ . 
    \end{equation*}
    This element 
  corresponds to $x\otimes \hat p_! y\in
   R(G)\otimes \hat K(*)$ under the lower isomorphism of (\ref{neu890}) as desired.
   The proofs for forms and topological $K$-theory work the same way.
  \end{proof}

\subsection{Localization} \label{localization}

\subsubsection{}

In the present subsection we show that a version of Segal's localization
theorem  \cite{MR0234452} holds true for differential $K$-theory. Let $B=[M/G]$ be an
orbifold represented by the action of a \emph{finite} group $G$ on a manifold
$M$. Then we have the projection $\pi\colon [M/G]\to [*/G]$.
For $g\in G$ let $[g]=\{hgh^{-1}|h\in G\}$ denote the conjugacy class $g$.
Note that $M^{g}$ is  a smooth submanifold of $M$, and for $h,l\in G$ we
have a canonical diffeomorphism $h\colon M^l\to M^{h^{-1}lh}$.  
We choose, $G$-equivariantly, tubular neighbourhoods $M^h\subseteq \tilde M^h$ for all $h\in G$, set
$\tilde M^{[g]}:=\bigcup_{h\in [g]} \tilde M^h\subseteq M$, and we consider the open suborbifold
$B^g:=[\tilde M^{[g]}/G]\subseteq B$.
We let $i\colon B^g\to B$ denote the inclusion. Note that
$B^g$ is considered as an orbifold approximation of the orbispace
$ [\bigcup_{h\in [g]} M^h/G]$ in the homotopy category of orbispaces.

\subsubsection{}

Note that $\hat K^0([*/G])\cong R(G)$, see \ref{calpunk}. Therefore $\hat K(B)$ and $\hat K(B^g)$ 
become  $R(G)$-modules via $\pi^*$ and  $\pi^*_g$ \textcolor{black}{and the
  cup-product, where $\pi_{g}\colon B^{g}\to [*/G]$} is the natural map. In this way $i^*\colon
\hat K(B)\to \hat K(B^g)$ is a map of $R(G)$-modules.

If we identify, using the character,  $R(G)$ with a subalgebra of the algebra of class
functions on $G$, $$  R(G)\subset R(G)_\C\cong \C[G]^G\ ,$$ we see
that $[g]$ gives rise to a prime ideal
$I([g])\subset R(G)$ consisting of all class functions 
which  vanish at $[g]$.

For an $R(G)$-module $V$ we denote by $V_{I([g])}$ its localization at
the ideal $I([g])$.

\subsubsection{}

\begin{theorem}\label{uidwqdwqdwqd}
The restriction $i^*\colon \hat K(B)\to \hat K(B^g)$  induces, after localization at $I( [g])$, an isomorphism
$$i^*\colon \hat K(B)_{I([g])}\rightarrow \hat K(B^g )_{I([g])}\ .$$
\end{theorem}
\proof We use the following strategy:
 We will first observe that there is a natural $R(G)$-module structure
on $\Omega(LB)/\im(d)$ such that the sequence
$$\rightarrow K(B)\stackrel{\ch}{\rightarrow} \Omega(LB)/\im(d)\stackrel{a}{\to}
\hat K(B)\stackrel{I}{\rightarrow} K(B)\rightarrow$$
becomes an exact sequence of $R(G)$-modules. Then we  will prove the analog
of the localization theorem for equivariant
forms.  Once this is established, we combine it with
Segal's localization theorem \cite{MR0234452} in ordinary $K$-theory,
$$i^* \colon  K(B)_{I([g])}\stackrel{\sim}{\to} K(B^g)_{I([g])}\ ,$$
  and  the result then follows from the Five Lemma. 

Let us start with the $R(G)$-module structure on $\Omega(LB)$.
The map $\pi\colon B\to [*/G]$ induces a homomorphism $L\pi^*\colon \Omega(L[*/G])\to \Omega(LB)$.
We now use the identification $\Omega(L[*/G])\cong \C[G]^G\cong R(G)_\C$.

\subsubsection{}

\begin{lem}
The natural map
$$Li^*\colon (\Omega(LB)/\im(d))_{I( [g])}\rightarrow (\Omega(LB^{g})/\im(d))_{I([g])}$$
is an isomorphism.
\end{lem}
\proof
Since localization is an exact functor it commutes with taking quotients.
Therefore it suffices to show that
$$Li^*\colon \ker(d_{\Omega(LB)})_{I( [g])}\to\ker(d_{\Omega(LB^{[g]})})_{I( [g])} ,\quad   Li^*\colon \Omega(LB)_{I( [g])}\rightarrow \Omega(LB^{[g]})_{I([g])}$$
are isomorphisms.
We give the argument for the second case. The  argument for the first isomorphism is similar.

Let $CG$ denote the set of conjugacy classes in $G$.  For $[h]\in CG$  we define the $G$-manifold $M^{[h]}:=\bigsqcup_{l\in [h]} M^l$.
Then $$LB \cong \bigsqcup_{[h]\in CG} [M^{[h]}/G]$$ is a
decomposition into a disjoint union of orbifolds.
Accordingly, we obtain a decomposition
$$\Omega(LB)\cong \bigoplus_{[h]\in CG} \Omega(M^{[h]})^G\ .$$

Let now $h\in G$ and $<h>$ be the subgroup generated by $h$.
If $<h>\cap [g]=\emptyset$, then  there
exists an element $x\in R(G)$ with
$x(g)\not=0$,  i.e.~$x\notin I([g])$ and $x_{|<h>}=0$.  As multiplication with $x$ is the zero
map on $\Omega(M^{[h]})^G$ and at the same time an isomorphism after
localization at $I([g])$, we observe that
$\Omega(M^{[h]})^G_{I([g])}=0$. Therefore,
we get
$$\Omega( LB)_{I([g])}\cong \bigoplus_{[h]\in CG,<h>\cap [g]\not=\emptyset}
\Omega(M^{[h]})^G_{I([g])}\ .$$
A similar reasoning  applies to $B^g$ in place of $B$:
$$\Omega(LB^g)_{I([g])}\cong \bigoplus_{[h]\in CG,<h>\cap [g]\not=\emptyset}
\Omega((\tilde M^{[g]})^{[h]})^G_{I([g])}\ .$$

If  $<h>\cap [g]\not=\emptyset$, then
the restriction
$\Omega(M^{[h]})^G\rightarrow \Omega((\tilde M^{[g]})^{[h]})^G$ is an
isomorphism. In fact, the map
$(\tilde M^{[g]})^{[h]}\rightarrow M^{[h]}$ is
a $G$-diffeomorphism. 
\hB 
This finishes the proof of the localization theorem. \hB

\section{The intersection pairing}\label{uwhdiodhwqidwqdwqdqwd}

\subsection{The intersection pairing as an orbifold concept}

\subsubsection{}


We start with the definition of a trace on the complex representation ring
$R(G)$ for a compact group $G$. Note that the underlying abelian group of
$R(G)$ is the free $\Z$-module generated by the  set $\hat G$
of  equivalence classes of irreducible complex  representations of $G$. The
unit $1\in R(G)$ is represented by the trivial representation of $G$ on $\C$.

We define
$$\Tr_G\colon R(G)\to \Z\ ,\quad \Tr_G (\sum_{\pi \in \hat G}n_\pi \pi):=n_1\ .$$

The bilinear form
$$(.,.)\colon R(G)\otimes R(G)\to \Z\ ,\quad (x,y)=\Tr_G(xy)$$
is non-degenerate. In fact, \textcolor{black}{if $\pi$ is an irreducible
  representation of $G$ then}
\begin{equation}\label{uidweqdcqx}
(\pi,\pi^\prime)=\left\{
\begin{array}{cc}
1;&\pi^\prime=\pi^*\\
0;&\mbox{else}
\end{array}\right. \ ,
\end{equation}
where $\pi^*$ denotes the dual representation of $\pi$.

The map $\Tr_G$ extends to the complexifications $R_\C(G):=R(G)\otimes\C$, the map
$$\Tr_G\colon R_\C(G)\to \C$$
will be denoted by the same symbol.

\subsubsection{}

Let $G$ be a finite group. The Chern character gives an isomorphism,  
$$\ch\colon R_\C(G)\cong \C[G]^G$$
via $$\sum_{\pi\in \hat G}n_\pi\pi\mapsto \sum_{\pi\in \hat G} n_\pi
\chi_\pi\ ,$$  where $G$ acts by conjugations on itself, and where  $\chi_\pi\in \C[G]^G$ denotes  the character of
$\pi$. Under this identification,
$$\Tr_G(f)=\frac{1}{|G|}\sum_{g\in G} f(g)\ .$$
Indeed, if $\pi$ is \textcolor{black}{a non-trivial irreducible representation},
then $\frac{1}{|G|}\sum_{g\in G} \chi_\pi(g)=0$, and $\frac{1}{|G|}\sum_{g\in
  G}\chi_1(g)=1$ by the orthogonality relations for characters.

\subsubsection{}

Let $G$ be finite. Note that $L[*/G]=[G/G]$, where $G$ acts on itself by conjugation. We have $\Omega([G/G])\cong \C[G]^G$. We thus define 
$$\Tr_G\colon \Omega(L[*/G])\to \C ,\quad \Tr_G(f):=\frac{1}{|G|}\sum_{g\in G} f(g)\ .$$
Observe that  for $x\in K([*/G])\cong R(G)$ and $f=\ch(x)$ we have
$\Tr_G(f)\in \Z$. Therefore we get an induced map
\begin{equation}\label{idwqdwqd}
\Tr_G\colon \Omega(L[*/G])/\Imm(\ch)\to \C/\Z=:\T\ .
\end{equation}

%
%
%

\subsubsection{}
\label{sec:TrG_trivial_action}
We continue the example developed in \textcolor{black}{Sections} \ref{sec:trivial_action} and
\ref{ex:int_on_trivial_action}. 

     Let $G$ be a finite group and $M$ a smooth manifold with trivial
    $G$-action. Using the isomorphism $K([M/G])\cong R(G)\otimes K(M)$ of  
Section \ref{sec:trivial_action}, we define $$\Tr^{M}_G:= \Tr_G\otimes \id\colon
K([M/G])\to K(M)\ ,$$ and correspondingly for $\hat K([M/G])$, $\Omega(L[M/G])$,
$U([M/G])$, where $U([M/G])$ is the flat part of differential K-theory studied
in Section \ref{sec:flat}.

  \begin{lem}
    Assume that $G$ is a finite group and $M$ a smooth compact manifold with
    trivial $G$-action. Assume that $p\colon M\to *$ \textcolor{black}{is}  
$K$-oriennted by a $Spin^{c}$-structure on $TM$ and
    pull this orientation back to a $K$-orientation of $p_G\colon [M/G]\to [*/G]$. Then
    \begin{equation*}
      \Tr_G\circ {p_G}_! = p_!\circ \Tr^{M}_G.
    \end{equation*}
   Moreover, {$\Tr^{M}_G$, $\Tr^{N}_G$} is compatible with pull-backs along maps of manifolds
   $M\to N$ and the induced map $[M/G]\to [N/G]$.
  \end{lem}
  \begin{proof}
    The first assertion follows immediately from the compatibility of $p_!$
    and ${p_G}_!$ of Example \ref{ex:int_on_trivial_action}, the second is a
    direct consequence of the definition of $\Tr^{M}_G$. 
  \end{proof}

\subsubsection{}

Let $G$ be a compact Lie group and consider a compact $G$-manifold $M$  with a $G$-equivariant $K$-orientation.
In this situation we have a push-forward $f^{G}_!\colon K_G(M)\to K_G(*)$
along the projection $f\colon M\to *$. Note that $[f/G]\colon [M/G]\to [*/G]$
is a representable map between orbifolds which is a locally trivial fibre
bundle with fibre $M$. It carries an induced topological $K$-orientation, and
we
have$$\xymatrix{K_{G}(M)\ar[d]^{f_{!}^{G}}\ar[r]^{\cong}&K([M/G])\ar[d]^{[f/G]_{!}}\\
  K_{G}(*)\ar[r]^{\cong}&K([*/G])}\ .$$
 We define the  intersection form\begin{equation}\label{zuidqwd}
(.,.)\colon K_G(M)\otimes K_G(M)\stackrel{\cup}{\to} K_G(M)\stackrel{f^{G}_!}{\to} K_G(*)\cong R(G)\stackrel{\Tr_G}{\to}\Z\ .\end{equation}

 \subsubsection{}

\newcommand{\indu}{{\mathrm{ind}}}

In certain special cases this intersection form is compatible with induction. Let $G\hookrightarrow  H$ be an inclusion of finite groups.
Then $H\times_GM$ has an induced $H$-equivariant $K$-orientation.
  \begin{prop}
If $G\hookrightarrow H$ is an inclusion of finite groups then
the following diagram commutes:
$$\begin{CD} K_G(M)\otimes K_G(M) @>{(.,.)}>> \Z \\
 @VV{\indu_G^H\otimes \indu_G^H}V @VV{=}V\\
 K_H(H\times_GM)\otimes K_H(H\times_GM) @>{(.,.)}>> \Z\ .\end{CD}$$
 \end{prop}
\proof
 The cup product and the integration are defined on the level of orbifolds. Hence they are compatible with induction, i.e.
$$
\xymatrix{K_G(M)\otimes K_G(M)\ar[r]^{\cup}\ar[d]^{\indu_G^H\otimes \indu_G^H}&K_G(M)\ar[d]^{\indu_G^H}\ar[r]^{f^G_!}&R(G)\ar[d]^{\indu_G^H}\\K_H(H\times_GM)\otimes K_H(H\times_GM)\ar[r]^(.66){\cup}&K_H(H\times_GM)\ar[r]^(.7){f_!^H}&R(H)}$$

commutes. We thus must show that the following diagram commutes
$$\xymatrix{R(G)\ar[d]^{\indu_G^H}\ar[r]^{\Tr_G}&\Z\ar@{=}[d]\\R(H)\ar[r]^{\Tr_H}&\Z}\ .$$
 If $\pi\in \hat G$, then
$$\indu_G^H(\pi)=[\C[H]\otimes V_\pi]^G\ ,$$
where we use the
right $G$-action on $\C[H]$ in order to define the invariants. The $H$-action is induced by the left action.
Since $\restr_H^G1=1$, by Frobenuis reciprocity
$$\Tr_H\indu_G^H(\pi)= \Tr_GV_\pi\ ,$$
as $\Tr_H(V)$ counts the multiplicity of $1$ in $V$.
\hB

If $G/H$ is not zero-dimensional, then
 an  $H$-equivariant $K$-orientation of $M$ does not
necessarily induce a $G$-equivariant $K$-orientation of $G\times_HM$. The
problem is that
$G/H$ does not have, in general, an $H$-equivariant $K$-orientation.

\subsubsection{}\label{zeuirerr}
\newcommand{\inv}{{\tt inv}}
Let $H\subseteq G$ be a normal subgroup of a finite group
which acts freely on a closed equivariantly $K$-oriented $G$-manifold $N$ with quotient $M:=N/H$. Then the group $K:=G/H$ acts on the closed equivariantly $K$-oriented $G$-manifold $M$.
We have a map
$\pi\colon [N/G]\to [M/K]$ of quotient stacks (see \ref{stack1})
which maps the pair $(P\to T,\phi\colon P\to N)\in [N/G](T)$ to the pair
$(P/H\to T,\bar \phi\colon P/H\to M)\in [M/K](T)$, where $P/H$ is the $K$-principal bundle obtaind as quotient of $P$ by $H$, and $\bar \phi$ is the natural factorization of $\phi$.
\textcolor{black}{It is well known that} $\pi$ is an equivalence of stacks. In order to see this we show that its evaluation 
at the smooth manifold $T$ described above is an equivalence of groupoids. Let us construct an inverse.
Given a pair $(Q\to T,\psi)\in [M/K](T)$   we 
define the {associated} $G$-principal bundle
$P:=Q\times_{M}N\to T$.
It carries the diagonal action by $G$ and comes with the $G$-equivariant map
$\hat \psi\colon P\to N$ given by the projection to the second factor.
This construction defines a functor
$[M/K](T)\to [N/G](T)$. We leave it as an exercise to see that these functors induce inverse to each other equivalences.

Let $f^K\colon M\to *$ and $f^G\colon N\to *$ denote the corresponding projections to the point.

If $V$ is a representation of $G$, then $K$ acts on the subspace $\inv^H(V):=V^H$ of $H$-invariants. We therefore get an induced homomorphism $\inv^H\colon R(G)\to R(K)$.
\begin{prop}\label{uiduiweewf}
The following diagram commutes:$$\xymatrix{K_G(N) \ar[r]^{f_!^G}&R(G)\ar[r]^{\Tr_G}\ar[d]^{\inv^H}&\Z\ar@{=}[d]
\\K_K(M)\ar[u]_{\pi^*}^\cong\ar[r]^{f^K_!}&R(K)\ar[r]^{\Tr_K}&\Z}\ .$$
\end{prop}
\proof
It follows from the relation
$\inv^G=\inv^K\circ \inv^H$, that the right square commutes. We now show that the left square commutes, too. We give an analytic argument. Let $x\in K_K(M)$ be represented by a $K$-equivariant geometric family $\cE$.
Then $\pi^*\cE$ is a $G$-equivariant geometric family
over $N$. Then $f^K_!(x)$ is represented by 
 the $K$-equivariant geometric family $f_!^K\cE$ over the point $*$. The corresponding element in $R(K)$ is the representation of $K$ on $\ker(D(f_!^K\cE))$.
Similarly, $f_!^G(x)$ is represented by the representation of $G$ on $\ker(D(f_!^G\pi^* \cE))$.
The projection $f_!^G\cE\to f_!^K\cE$ is a regular covering
  with covering group $H$, respecting all the geometric structure. In particular, we have
$H(f_!^K\cE)=H(f^G_!\pi^*\cE)^H$ (distinguish between the Hilbert
space $H(\dots)$ associated to a geometric family and the group $H$) and 
$\ker(D(f_!^K\cE))=\ker(D(f_!^G\pi^* \cE))^H$
as representations of $K$. This implies the commutativity of the left square. \hB

 \subsubsection{}

In the following theorem we show that the intersection pairing is a well-defined concept at least for orbifolds which admit a presentation as a quotient of a closed equivariantly $K$-oriented  $G$-manifold for a finite group $G$.
\begin{theorem}\label{wudiwqd}
If $B$ is an orbifold which admits a presentation
$B\cong [M/G]$ for a finite group $G$ such that $[M/G]\to [*/G]$ is $K$-oriented, then (\ref{zuidqwd}) induces a well-defined intersection pairing 
$$K(B)\otimes K(B)\to \Z\ .$$
\end{theorem}
\proof
We choose a presentation $B\cong [M/G]$ and define the pairing such that
$$\xymatrix{K(B)\otimes K(B)\ar[d]^\cong\ar[r]^(.6){(.,.)}&\Z\ar@{=}[d]\\
K_G(M)\otimes K_G(M)\ar[r]^(.7){(.,.)}&\Z}$$ commutes.
We must show that this construction does not depend on the choice of the presentation.
Let $B\cong [M^\prime/G^\prime]$ be another presentation. 

We use the setup of \cite{pronk-2007} where the  $2$-category of orbifolds is
identified with a localization of a full subcategory of Lie groupoids,
\cite[\textcolor{black}{Theorem 3.4}]{pronk-2007}.

Let $G\rtimes M$ und $G^\prime \rtimes M^\prime$ be the action groupoids. Since they represent the same orbifold $B$, the isomorphism
$G\rtimes M\cong G^\prime \rtimes M^\prime$
in this localization is represented by  a diagram
$$\xymatrix{&\cK\ar[dr]^{u}\ar[dl]_{v}&\\G\rtimes M&&G^\prime \rtimes M^\prime}$$
where $\cK$ is a Lie groupoid and $v$ und $u$ are essential equivalences. By
\cite[Proposition \textcolor{black}{7.1}]{pronk-2007} this diagram is isomorphic
(in the category 
of morphisms) between $G\rtimes M$ and $G^\prime \rtimes M^\prime$ to a
diagram of the form
$$\xymatrix{&(G\times G^\prime)\rtimes N\ar[dr]^{ v}\ar[dl]_{ u}&\\G\rtimes M&&G^\prime \rtimes M^\prime}\ ,$$
where now $u\colon N\to M$ and $v\colon N\to M^\prime $ are equivariant maps over the projections
$G\times G^\prime\to G$ and $G\times G^\prime\to G^\prime$.

For $\bar x,\bar y \in K(B)$ let $x,y\in K_G(M)$ and $x^\prime,y^\prime\in K_{G^\prime}(M^\prime) $ be the corresponding elements under $K(B)\cong K_G(M)\cong K_{G^\prime}(M^\prime)$. We have $u^*x=v^*x^\prime$ and $u^*y=v^*y^\prime$. The subgroups $G^\prime,G\subseteq G\times G^\prime$
 are normal and act freely on $N$. By Proposition
\ref{uiduiweewf} we get 
\begin{multline*}
\Tr_{G} (f^{G}_! (x\cup y))=\Tr_{G\times G^\prime} (f^{G\times G^\prime}_!
(u^*x\cup u^* y))=\\
\Tr_{G\times G^\prime} (f^{G\times G^\prime}_! ( v^*x'\cup
v^* y' )) =\Tr_{G^\prime} (f^{G^\prime}_! (x^\prime\cup y^\prime))\ ,
\end{multline*}
where
$f^G,f^{G\times G^\prime}$ and $f^{G^\prime}$ are the corresponding projections to the point. 
\hB

\subsection{The flat part and  homotopy  theory}
\label{sec:flat}
\subsubsection{}

If $B$ is a presentable and compact orbifold, then we can consider the flat part 
$$U(B):=\ker\left(R\colon \hat K(B)\to \Omega(LB)\right)$$
 of the differential $K$-theory of $B$. The functor $B\mapsto U(B)$ from compact presentable orbifolds to $\Z/2\Z$-graded abelian groups is
homotopy invariant. The main goal of the present section is to
identify this functor in homotopy-theoretic terms. In the language of
\cite[Definition 5.4]{bunke-2009}, we are going to show that $U$ is
topological.

\subsubsection{}

A $G$-equivariant $Spin^{c}$-structure on  a closed $G$-manifold $M$ induces a
$G$-equivariant $K$-orientation, i.e.~a $G$-equivariant fundamental class
$[M]\in K_{\dim M}^G(M)$. For sake of completeness we will explain the local characterization of  $[M]$
which makes clear why the usual proof of Poincar\'e duality extends from the
non-equivariant to the equivariant case. Let us represent $K$-homology in the
\textcolor{black}{equivariant} $KK$-theory picture (see \cite{bla} for an
  introduction to $KK$-theory). The $G$-action on $M$ induces a $G$-action on
  the $C^{*}$-algebra $C(M)$ of continuous functions on $M$, and have
$$K_n^G(M):=KK^G(C(M),\Cliff(\R^n))\ $$
 where $\Cliff(\R^n)$ is the complex Clifford algebra
of $\R^n$ with the standard Euclidean inner product and trivial $G$-action.
The equivariant fundamental class $[M]\in K_{n}^{G}(M)$ of M is represented by the  equivariant Kasparov module
$(L^2(M,E), D)$, where $E=P\times_{Spin^c(n)}\Cliff(\R^n)$
is the $G$-equivariant Dirac bundle associated to the  equivariant $Spin^c(n)$-principal bundle
 $P\to M$  determined by the equivariant $K$-orientation. Note that the Dirac
operator $D$ of $E$ commutes with  the  action of
$\Cliff(\R^n)$ from the right.

Let $x\in M$ und $G_x$ be its stabilizer  group. Then we have a
$G_{x}$-invariant decomposition $T_xM\cong T_x(Gx)\oplus N$, such that
$T_{x}(Gx)$ is fixed by $G_{x}$, and the only $G_{x}$-invariant vector in the
normal summand $N$ is the zero vector.
  A tubular neighbourhood of the orbit $Gx$ can be identified with
$U_x:=G\times_{G_x}V_x$, where $V_x\subset N$ is a disc.
The restriction of the fundamental class to $U_x$ gives an element
$$[M]_{U_x}\in K_n^G(U_x,\partial U_x)\cong K_n^{G_x}(V_x,\partial V_x)\ .$$
Note that $V_x$ admits a $G_x$-equivariant $Spin^c$-structure.
   It is uniquely determined by  the equivariant $Spin^{c}$-structure of $M$
  up to a choice of a $G_x$-equivariant $Spin^c$-structure on the vector space $T_{[G_x]}(G/G_x)$. 
The $Spin^c$-structure gives an equivariant Thom class and the
Thom isomorphism
$$R(G_x)\cong K_{0}^{G_{x}}(*)\stackrel{Thom}{\cong}  K_n^{G_x}(V_x,\partial V_x)$$
of $R(G_x)$-modules.
The characterizing property of a fundamental class (which is satisfied by the class of $(L^{2}(M,E),D)$) is that
$[M]_{U_x}$ is a generator of the $R(G_x)$-module $K_n^{G_x}(V_x,\partial V_x)$ for every $x\in M$.
This condition does not depend on the choice of the $Spin^c$-structure on $T_{[G_x]}(G/G_x)$.

The equivariant $K$-theory fundamental class induces a Poincar\'e duality
isomorphism 
$$P\colon K^{*}_G(M)\xrightarrow{\cdots\cap [M]} K^G_{n-*}(M)\ .$$
Using the Poincar\'e duality isomorphism
 the intersection pairing \textcolor{black}{of} Theorem \ref{wudiwqd}) can be
 written in the form 
 \begin{equation}\label{eq:3star}
K_G(M)\otimes K_G(M)\xrightarrow{1\otimes P} K_G(M)\otimes K^G(M)\xrightarrow{eval} R(G)\xrightarrow{\Tr_G} \Z\ .
\end{equation}

To see this we use the sequence of equalities
$$eval(y\otimes P(x))=eval(y\otimes (x\cap[M]))=eval((x\cup y)\otimes [M])=f^{G}_{!}(x\cup y)\ , \quad x,y\in K_{G}(M)$$
which relate the Poincar\'e duality isomorphism $P$ with the push-forward $f^{G}_{!}$.

\subsubsection{}

Recall that $\T:=\C/\Z$. 
We define a new $G$-equivariant cohomology theory (compare with
\cite[\textcolor{black}{Section 1}]{brco})
which associates to a $G$-space $M$ the
group  $$k_G^\T(M):=\Hom_\Ab(K^G(M),\T)\ .$$  
In fact, since $\T$ is a divisible and hence injective abelian group, the long exact sequences for $K^G$ induce long exact sequences for $k^\T_G$.

Complex conjugation in $\T$ induces a natural  involution on
$k^\T_G(M)$. Its fixed points will be denoted by
$k_G^{\R/\Z}(M)$. In other words, 
$$k_G^{\R/\Z}(M) := \Hom_\Ab(K^G(M),\R/\Z)\subseteq \Hom_\Ab(K^G(M),\T)\ .$$
\textcolor{black}{In the terminology of \cite[Section 1]{brco}, this} is the
Pontrjagin dual of $K^{G}$.

If $M$ is equivariantly $K$-oriented, then
we have natural pairings
\begin{equation} \label{eiufhwiefcewe}
K_G(M)\otimes k^\T_G(M)\xrightarrow{\ev} \T\ , \quad K_G(M)\otimes k^{\R/\Z}_G(M) \xrightarrow{\ev}\R/\Z \end{equation}
given by
$$x\otimes \phi\mapsto \phi(P(x))\ .$$
Since $P$ is an isomorphism, by Pontryagin duality  this
pairing is non-degenerate in the sense that 
it induces a monomorphism
$$K_G(M)\hookrightarrow \Hom_{\Ab}(k^\T_G(M),\T)$$ and isomorphisms
$$k^\T_G(M)\cong  \Hom_{\Ab}(K_G(M),\T)\ ,\quad K_G(M)\cong
\Hom_{\Ab}(k^{\R/\Z}_G(M),\R/\Z)\ .$$ 
 For the latter, we use only continuous homomorphisms and the usual topology on
$k^{\R/\Z}_G(M)$ as a dual of a discrete group. 

\subsubsection{}
\renewcommand{\Ab}{{\mathrm{Ab}}}

We now define cohomology theories
$K_G^\C$ (the complexification of $K_G$-theory) and  $K_G^\T$ which fit into a
natural Bockstein sequence
\begin{equation}\label{udqwiduhqwuidhqwid}
\dots\to K_G^i(M)\to K_G^{\C,i}(M)\to K_G^{\T,i}(M)\to K_G^{i+1}(M)\to \dots\ .
\end{equation}
For this we work in the stable $G$-equivariant homotopy category whose objects
are called naive $G$-spectra (see \cite{may} for reference).

 {It is known by Brown's representability theorem that $G$-equivariant
 (co)homology theories (on finite $G$-CW-complexes)
and transformations between them can be represented by $G$-spectra and maps
between them. 
In certain cases (e.g. for $K^\C_G$ or $K_G^\T$) we want to know that these
spectra are determined uniquely up to unique isomorphism. Similarly, we want
to know that certain maps between these $G$-spectra are uniquely determined by
the induced transformation of equivariant 
homology theories.}

The abelian group of morphisms between $G$-spectra
$X,Y$ will be denoted by  $[X,Y]$.  
A $G$-spectrum will be called \emph{cell-even} if it can be written as a homotopy colimit over even $G$-cells. 
We will repeatedly use the following fact.
\begin{lem}\label{udiqdqwdwqd}
Let $X,Y$ be $G$-spectra such that $X$ is   cell-even and the odd-dimensional homotopy groups of $Y$ vanish.
If  $f\colon X\to Y$ induces the zero map in homotopy groups, then
$f=0$.
\end{lem}
 \proof
We write $X$ as a homotopy colimit of even $G$-cells $X\cong \hocolim_{i\in I} Z_i$ and consider the Milnor
sequence
$$0\to \lim^1_{i\in I} [\Sigma^{-1}Z_i,Y]\to [X,Y]\to \lim_{i\in I}[Z_i,Y]\to 0\ .$$
Since $f$ induces the zero map in homotopy groups it comes from the $\lim^1$-term.
Since the odd-dimensional  homotopy  groups of $Y$ vanish we have  $[\Sigma^{-1}Z_i,Y]=0$ for all $i\in I$ so that
the $\lim^1$-term vanishes. It follows that $f=0$.\hB

We will represent a  $G$-equivariant cohomology
theory $h^G$ by a $G$-spectrum $\underline{h^G}$.
We start with the  $G$-ring spectrum $\underline{K^G}$ which represents
$G$-equivariant $K$-theory. It is cell-even, as it can be
  built using copies of $BU$ and has only even-dimensional
homotopy groups. 
By Lemma \ref{udiqdqwdwqd} it is uniquely determined up to unique isomorphism
\textcolor{black}{in the homotopy category we're working in}.
Since $R_\C(G):=R(G)\otimes \C$ is a flat
$R(G)$-module we get a new $K_G$-module homology theory
$$K_\C^G(M):=K^G(M)\otimes_{R(G)}R_\C(G)\cong K^G(M)\otimes_\Z \C\ .$$
 {This $G$-equivariant homology theory can be represented by 
 the  $G$-spectrum  $\underline{K^G_\C}=\underline{K^G}\wedge\bM\C$. Here $\bM\C$ is the Moore
spectrum for $\C$. In general, a Moore spectrum $\bM A$ for an abelian group $A$  can be written as a colimit over a system of  (zero-dimensional)
cells.} Therefore $\underline{K^G_\C}$ is again cell-even
and has only even-dimensional homotopy groups. The homology theory $K^G_\C$ thus determines
 the $\underline{K^G}$-module $G$-spectrum $\underline{K^G_\C}$  uniquely
 upto unique  isomorphism. The transformation $K^{G}\to K^{G}_\C$ of homology
 theories induces a morphism of $\underline{K^G}$-module $G$-spectra
$\underline{K^G}\to \underline{K^G_\C}$ which is unique, again by  Lemma \ref{udiqdqwdwqd}.
We choose an extension of this morphism to a distinguished triangle
\begin{equation}\label{iwudqwdqwdqdwqd}
\underline{K^{G} } \to\underline{K^{G}_{\C}} \to  \underline{K^{G}_{\T} }\to  \Sigma \underline{K^{G} }
\end{equation}
 which defines the $\underline{K^G}$-module $G$-spectrum    $\underline{K^{G}_{\T}}$ uniquely upto isomorphism. 
{In fact, we can write $\underline{K^{G}_{\T}}\cong \underline{K^G}\wedge\bM\T$ so that $\underline{K^{G}_{\T}}$ is
  again cell-even. Since  it has only even-dimensional homotopy groups, the  $\underline{K^G}$-module $G$-spectrum   $\underline{K^{G}_{\T}}$   is actually defined
upto unique isomorphism.}

We let
$K_G^\T$ denote the cohomology theory represented by $\underline{K^{G}_{\T}}$.
It is a $K_G$-module theory.  For trivial $G$-spaces, the
  construction and 
properties of Subsection \ref{sec:trivial_action}, Subsection
\ref{ex:int_on_trivial_action} and Subsection \ref{sec:TrG_trivial_action}
 {work also for $K_G^\T$}.

In a similar manner, if we set $K^G_\R(M):=K^G(M)\otimes \R$ and consider the distinguished triangle $$\underline{K^{G} } \to\underline{K^{G}_{\R}} \to  \underline{K^{G}_{\R/\Z} }\to  \Sigma \underline{K^{G} } \ ,$$ then
we uniquely define a $K_G$-module cohomology theory $K_G^{\R/\Z}$.

\subsubsection{}\label{idef456}

The cohomology theory $k_G^\T$ is good for the non-degenerate pairing (\ref{eiufhwiefcewe}).
On the other hand, as an immediate consequence of the fibre sequence
(\ref{iwudqwdqwdqdwqd}), 
 the cohomology theory $K_G^\T$ fits into the Bockstein sequence (\ref{udqwiduhqwuidhqwid}).
Since later in the present paper we need both properties together
we must compare the cohomology theories
$k_G^\T$ and $K_G^\T$. In the present paragraph we start with the definition of a transformation
$i\colon K_G^\T\to k_G^\T$. In Lemma \ref{uzqidhiqwdwqdw} we will give conditions under which $i$ induces an isomorphism.

We extend the cohomology theory $k_G^\T$ to $G$-spectra $X$ in the natural way by  defining
$$k_G^\T(X):=\Hom_{\Z}(K^G(X),\T)\ .$$
The evaluation between homology and cohomology extends to the
complexifications
$$eval_\C\colon K^\C_G(X)\otimes K^G(X)\to K_{G}^{\C}(*)\cong R_\C(G)\ .$$

We consider the
natural transformation between cohomology theories
$$c\colon K^\C_G(X)\to k_G^\T(X)\ , \quad K^\C_G(X)\ni x\mapsto \left\{K^G(X)\ni z\mapsto [\Tr_G(eval_\C(x\otimes z))]\in \T\right\}\ ,$$
where $\Tr_{G}$ has values in $\C$, and the brackets $[\cdots]$ denote the
class    in $\T=\C/\Z$.
The triangle
(\ref{iwudqwdqwdqdwqd}) induces a long exact sequence
$$\dots \to k_G^\T(\Sigma \underline{K^G})\to k_G^\T(\underline{K^G_\T})\stackrel{b}{\to}  k_G^\T(\underline{K^G_\C})\stackrel{a}{\to} k_G^\T(\underline{K^G})\to \dots\ .$$
We let $C:=c(\id_{\underline{K^G_\C}})\in k_G^{\T,0}(\underline{K^G_\C})$.
Since the composition
$$K_G(X)\to K^\C_G(X)\stackrel{c}{\to} k_G^\T(X)$$
vanishes we have  $a(C)=0$. Hence there exists a lift $I\in  k_G^{\T,0}(\underline{K^G_\T})$ such that
$b(I)=C$.  We claim that $k_G^{\T,0}(\Sigma \underline{K^G})=0$. The claim implies that the lift $I$ is uniquely determined.
To see the claim  
we write $\underline{K^G}$ as a homotopy colimit over even $G$-cells 
$$\underline{K^G}\cong \hocolim_{j\in J} Z_j\ .$$

We then have
\begin{eqnarray*}
 k^{\T,0}_G( \underline{\Sigma K^G})&= &\Hom_\Ab(K^G_0(\Sigma \underline{K^G}),\T)\\
&\cong &\Hom_\Ab(K^G_0  (\hocolim_{j\in J}\Sigma Z_j),\T)\\
&\cong&\lim_{j\in J}\Hom_\Ab(K^G_0 (\Sigma   Z_j),\T)\\
&=&0
\end{eqnarray*}
The element $I\in k_G^{\T,0}(\underline{K^G_\T})$ induces the desired natural
transformation of cohomology theories $i\colon K_G^\T\to k_G^\T$.
In a similar manner we define a transformation $i\colon K_G^{\R/\Z}\to k_G^{\R/\Z}$.

\subsubsection{}

We now analyse when the transformation of cohomology theories $i$ defined in
\ref{idef456} is an isomorphism.
 Let $H\subseteq G$ be a closed subgroup. Then we have 
$$K^G(G/H)\cong K^H(*)\cong R(H)\ ,\quad K_\C^G(G/H)\cong K_\C^H(*)\cong R_\C(H) ,$$
and hence,  as the homotopy groups of  our spectra are concentrated in even dimensions,
$$K^G_\T(G/H)\cong R_\C(H)/R(H)\cong R(H)\otimes \T\ .$$
Furthermore
$$k^\T_G(G/H)\cong \Hom_{\Ab}(K^G(G/H),\T)\cong \Hom_{\Ab}(K^H(*),\T)\cong \Hom_{\Ab}(R(H),\T)\ .$$  
Let $x\in K^\C_G(G/H)\cong  R_\C(H)$ and $[x]\in  K^\T_G(G/H)\cong  R_\C(H)/R(H)$ be the induced class.
Then we have for $i\colon K_G^\C(G/H)=R_\C(H)\to k_G^\T(G/H)=\Hom_{\Ab}(R(H),\T)$
\begin{equation}\label{eq:2star}
i(x)(y)=[\Tr_G(yx)]=[(y,x)],\quad \forall y\in R(H)\ .
\end{equation}
Because of  (\ref{uidweqdcqx}) the map $i$ is injective.
It is surjective if and only if $R(H)$ is  a finitely
generated abelian group, i.e.~if $H$ is finite.
\begin{lem}\label{uzqidhiqwdwqdw}
If $G$ is finite, then the transformations $i\colon K^\T_G\to k^\T_G$ and   $i\colon K^{\R/\Z}_G\to k^{\R/\Z}_G$ are equivalences of cohomology theories on finite $G$-$CW$-complexes.
If $G$ is compact and if  $M$ is a compact $G$-manifold or a compact
$G$-$CW$-complex on which $G$ acts with finite stabilizers, then
$i\colon K^\T_G(M)\to k_G^\T(M)$ and   $i\colon K^{\R/\Z}_G(M)\to k^{\R/\Z}_G(M)$  are  isomorphisms.
\end{lem}
\proof
 We only discuss the complex case. The real case is similar.
The first statement follows from the discussion above since $i$ induces an isomorphism for all $G$-cells.
 For the second observe
that a compact $G$-manifold has the structure of a $G$-$CW$-complex. We then proceed by
induction over $G$-cells which are of the form $G/H\times D^n$ with finite
$H\subset G$, using Mayer-Vietoris and again that 
$ i\colon K^\T_G(G/H) \to k_G^\T(G/H)  $
is an isomorphism for finite subgroups $H\subset G$.
\hB

\begin{kor}\label{barcel}
If $G$ is a compact group which acts on a $G$-equivariantly $K$-oriented
closed manifold  $M$ with finite stabilizers, then the pairing
$$<\cdots,\cdots>\colon K_G(M)\otimes K^\T_G(M)\xrightarrow{\cup} K^\T_G(M)\xrightarrow{f_!} R(G)\xrightarrow{\Tr_G} \T$$
is a non-degenerate pairing in the sense that
the induced map
$$K_G(M)\to \Hom_\Ab(K^\T_G(M),\T)$$
is a monomorphism, and that 
$$ \quad K^\T_G(M)\to \Hom_\Ab
(K_G(M),\T)\ ,\quad {K^{\R/\Z}_G(M)\to \Hom_\Ab( K_G(M),\R/\Z)}$$
are isomorphisms.
\end{kor}
\proof
Indeed, under the isomorphism $i\colon K^\T_G(M)\stackrel{\sim}{\to} k^\T_G(M)$ the pairing $<\cdots, \cdots>$ is identified with the evaluation pairing (\ref{eiufhwiefcewe}). \hB 

\subsubsection{}

Let $B$ be a presentable and  compact orbifold.
\begin{ddd}
We define the  flat $K$-theory of $B$  {(or its real part,
  respectively)} as the kernel of the curvature morphisms,
$$U(B):=\ker(R\colon \hat K(B)\to \Omega(LB))\ , \quad  {U^\R(B):=\ker(R\colon \hat K_\R(B)\to \Omega_\R(LB))}\ .$$
If $B=[M/G]$ for a compact Lie group $G$ acting on a compact manifold with finite stabilizers, then we will also write
$$U_G(M):=U([ M/G ])\ ,\quad {U^\R_G(M):=U^\R([ M/G ])}\ .$$
\end{ddd}
It follows from Proposition \ref{prop1} that, as always for differential
cohomology theories, $U(B)$ fits into a long  exact sequence 
$$\dots \to K^{n-1}(B)\to H_{dR}^{n-1}(LB)\to U^n(B)\to K^n(B)\to H_{dR}^n(LB)\to \dots\ .$$
If  $B=[M/G]$ is a presentation, then we use
the notation $H_G(M):=H_{dR}( L[M/G] )$ and $K_G(M)=K([M/G])$. The above long exact sequence now becomes
$$\dots \to K^{n-1}_G(M)\to H^{n-1}_G(M)\to U^n_G(M)\to K^n_G(M)\to H^n_G(M)\to \dots\ .$$

\subsubsection{}

We want to define maps
$$j\colon U_G(M)\to K^\T_G(M)\ ,\quad {j\colon U^\R_G(M)\to K^{\R/\Z}_G(M)}$$
by  constructing  the lower horizontal map in the diagrams
\begin{equation}\label{eq:1star}
\xymatrix{&K^\T_G(M)\ar[dr]^i_\cong&\\
  U_G(M)\ar[rr]^{j_G}\ar[ur]^j&&k^\T_G(M) },\quad
{\xymatrix{&K^{\R/\Z}_G(M)\ar[dr]^i_\cong&\\
    U^\R_G(M)\ar[rr]^{j_G}\ar[ur]^j&&k^{\R/\Z}_G(M) }}\ . 
\end{equation}

Their constriction involves integration 
$$\int^{\hat
    K}_{[M/G]/[*/G]}\colon U_{G}(M)\to U_{G}(*)$$
of flat classes along the
map $[M/G]\to [*/G]$. In order to define this integration we
  first chose a differential refinement of the topological $K$-orientation of
  this map and then use the integration in differential $K$-theory given in
  Definition \ref{ddd1}. By (\ref{lowersq}) the integral 
 preserves the flat subgroup. Moreover, as a consequence of homotopy invariance, the integral of  flat differential $K$-theory classes does only depend on the underlying topological $K$-orientation of the map and not on its differential refinement.
 
In order to stay in the category of orbifolds for $[*/G]$ we must assume that $G$ is a {\em finite} group.
 We set for $\xi\in K^G(M)$, $u\in U_G(M)$
 \begin{equation}
j_G(u)(\xi):=\Tr_G\left(\int^{\hat K}_{[M/G]/[*/G]} u\cup \widehat{P^{-1}(\xi)}\right)\in  \T\ .\label{eq:22star}
\end{equation}

Here
$\widehat{P^{-1}(\xi)}\in \hat K_G(M)$ denotes a differential
refinement of the Poincar\'e dual of $\xi$. Its product
with the flat class $u$ is again a flat class which does not depend on the choice of the differential refinement of $P^{-1}(\xi)$. Furthermore note that the integral has values in
$ U([*/G])\cong \Omega(L[*/G])/\im(\ch)$ (this group \textcolor{black}{is
  concentrated} in odd degree).  Finally, $\Tr_{G}$ is the factorization 
(\ref{idwqdwqd}) of the trace map.

In the following we
 indicate by a superscript  in which theory the integration is understood.
 It is easy to see that $j_G$ restricts to the real parts.

\begin{theorem}\label{uiudqdqwdwq}
Assume that $G$ is a finite group, and that $M$ is a $G$-equivariantly $K$-oriented closed $G$-manifold. Then the maps
$$j\colon U_G(M)\to K^\T_G(M)\ ,\quad {j\colon U^\R_G(M)\to K^{\R/\Z}_G(M)}$$
 are isomorphisms.
 \end{theorem}
\proof  
We discuss the complex case. The real case is similar.
Since $[M/G]$ is a good orbifold,
the Chern character induces an isomorphism (see \cite{bc})
$$\ch_G\colon K_G^\C(M)\stackrel{\sim}{\to} H_G(M)\ .$$
We consider the following diagram with
exact horizontal sequences
\begin{equation}\label{uwhfweffwfwf}
\xymatrix{K_G(M)\ar@{=}[d]\ar[r]^{\ch_G}& H_G(M)\ar[r]^{{-a}}&U_G(M)\ar[d]^j\ar[r]^\beta&K_G(M)\ar[r]^{\ch_G}\ar@{=}[d]&H_G(M)\\K_G(M)\ar[r]& K^\C_G(M)\ar[u]_\cong^{\ch_G}\ar[r]&K^\T_G(M)\ar[r]^\delta&K_G(M)\ar[r]&K^\C_G(M)\ar[u]^{\ch_G}_\cong}\ .
\end{equation}

\begin{lem}\label{udqwdqd}
The diagram commutes.
\end{lem}
If we assume this lemma it follows \textcolor{black}{from} the Five Lemma that
$j$ is an isomorphism.
\hB

 Note that all terms in \eqref{uwhfweffwfwf} are $K_G(M)$-modules and all
transformations are $K_G(M)$-module maps. Moreover, all transformations are
compatible with integration. We will use these facts in the proof of Lemma
\ref{udqwdqd} which occupies the rest of the present Subsection.
 
{ 
The guiding idea of our proof of the most complicated part, the equality \textcolor{black}{$\delta\circ j=\beta$},  is the following.
Morally, we will show how to realize all relevant classes as push-forwards
of classes on $M_n\times M$ along the projection to $M$,
where $M_n$ is the Moore space for $\Z/n\Z$. We will see that the equality 
$\delta\circ j=\beta$ for $M_n\times M$ 
implies the equality for $M$.  
Using the compatibility of the maps with integration and cup products, by integration over $M$
we can reduce to the equality in the non-equivariant case for $M_n$. Indeed, the non-equivariant case is already known from \cite{bunke-2007} or \cite{bunke-2009}.
Since  $M_n$ is the mapping cone of the self map of degree $n$ of $S^1$  and not a closed 
manifold, technically we  will use $S^1$ instead.}
 

\subsubsection{}
 We now give the details of the proof of Lemma \ref{udqwdqd}.
It is clear that the first and the fourth square commute.
Next we show that the second square commutes.

We consider a class  $x\in K^\C_G(M)$.
We must show the equality $- j_G(a(\ch_G(x)))=\phi(x)$, where 
$\phi\colon K^\C_G(M)\to k^\C_G(M)\to k^\T_G(M)$ is the natural map (denoted
by $c$ in \ref{idef456}).
To this end  we compare  the evaluations of both sides at a homology class
$\xi\in K^G(M)$. We have 
\begin{eqnarray*}
\phi(x)(\xi)&=&\Tr_G[\int_M^{K_G} x\cup P^{-1}(\xi)]_{\T}\\
&=&\Tr_G[\int^{H_G}_{L[M/G]/L[*/G]} \hA^c_\rho(LM) \cup \ch_G(x)\cup \ch_G(P^{-1}(\xi))]_{\T}\\
&=&{-}\Tr_G[\int_{[M/G]/[*/G]}^{\hat K} a(\ch_G(x)\cup \ch_G(P^{-1}(\xi)))]_{\T}\\
&=&{-}\Tr_G[\int^{\hat K}_{[M/G]/[*/G]} a(\ch_G(x))\cup \widehat{P^{-1}(\xi)}]_{\T}\\
&=&-j_G(a(\ch_G(x)))(\xi)\ .
 \end{eqnarray*}

\subsubsection{}\label{sec:finally-we-show}
Finally we show that the third square in (\ref{uwhfweffwfwf}) commutes.
The argument is surprisingly complicated.
First of all note that $\im(\beta)=K_G^{tors}(M)\subseteq K_G(M)$ is the torsion subgroup. Let 
$t\in K_G^{tors}(M)$.  Then there exists an integer $n\in \nat$ such that $n t=0$.
 
Let $f\colon S^1\to S^1$ be  the covering of degree $n$.
We form the mapping cone sequence
\begin{equation}\label{zuidudwqdiuqwdiuhi}
\xymatrix{S^1\ar[r]^f&S^1\ar[r]\ar[dr]^\pi&C(f)\\&&M_n\ar[u]^\sim}\ ,
\end{equation}
where $M_n$ is a compact manifold with boundary which
 is homotopy equivalent to the cone $C(f)$. It is a smooth model of the Moore space of $\Z/n\Z$.
Using the long exact sequences of reduced cohomology and $K$-theory
$$\tilde H(S^1)\stackrel{n}{\leftarrow} \tilde H(S^1)\stackrel{\pi^*}{\leftarrow} \tilde H(M_n)\stackrel{\delta}{\leftarrow},\qquad
\tilde K(S^1)\stackrel{n}{\leftarrow} \tilde K(S^1)\stackrel{\pi^*}{\leftarrow} \tilde K(M_n)\stackrel{\delta}{\leftarrow}$$
we get
$$\tilde H^*(M_n)\cong 0\ ,\quad  H^*(M_n)\cong\left\{ \begin{array}{cc} \C &{*}=0\\0&
{*}\ge 1\end{array} \right. \ .$$
and
$$\tilde K^*(M_n)\cong \left\{ \begin{array}{cc} \Z/n\Z &{*}=0\\ 0&{*}=1\end{array} \right. \ ,\quad   K^*(M_n)\cong \left\{ \begin{array}{cc} \Z/n\Z\oplus \Z &{*}=0\\ 0&{*}=1\end{array} \right.\ .$$
This implies that
$$U^0(M_n)\cong  \Z/n\Z\ ,\quad U^1(M_n)\cong \T\ .$$
In particular, we see that
$\beta\colon U^0(M_n)\to K^{0,tors}(M_n)$ is
an isomorphism. 

We now analyse the map
$\pi^*\colon U^0(M_n)\to U^0(S^1)\cong \T$.
We know from \textcolor{black}{\cite[Section 2.5.4]{bunke-2007} and
  \cite[Section 7]{bunke-2009}}
that the map $j\colon U\to K^\T$ induces an isomorphism
of reduced cohomology theories (i.e.\ the non-equivariant version of the Theorem \ref{uiudqdqwdwq} holds).
Since $U$ is a reduced cohomology theory
we have a mapping cone sequence
$$U^0(S^1)\stackrel{n}{\leftarrow} U^0(S^1)\stackrel{\pi^*}{\leftarrow}  U^0(M_n)\stackrel{\delta}{\leftarrow} U^1(S^1)\stackrel{1}{\leftarrow}U^1(S^1)\ ,$$
where we use the known actions of $f^*$ on
$U^0(S^1)\cong H^1(S^1,\Z)\otimes \T$ and
$U^1(S^1)\cong H^0(S^1,\Z)\otimes \T$.
We get 
$$\xymatrix{0\ar[r]&U^0(M_n)\ar[d]^\cong\ar[r]&U^0(S^1)\ar[d]^\cong\ar[r]&U^0(S^1)\ar[d]^\cong\\0\ar[r] & \Z/n\Z\ar[r] &\T \ar[r]^n&\T}\ .$$
In particular we see that the composition
$$I_n:=\int_{S^1}^{U}\circ \pi^*\circ \beta^{-1}\colon \Z/n\Z\cong
K^{\textcolor{black}{0,tors}}(M_n)\to \T\cong U^{-1}(*)$$ is the usual embedding
$\Z/n\Z\hookrightarrow \T$.
Note that in the non-equivariant case we have
$ \delta\circ j=\beta$. Therefore, we also have
$$I_n=\int_{S^1}^{K^\T}\circ \pi^*\circ \delta^{-1}\ .$$

The product of the mapping cone sequence (\ref{zuidudwqdiuqwdiuhi}) with $M$ induces  a long exact sequence
\begin{multline}K_G(S^1\times M,*\times M)\stackrel{(f\times
\id)^*}{\leftarrow} K_G(S^1\times M,*\times M)\\
\stackrel{(\pi\times
\id)^*}{\leftarrow} K_G(M_n\times M,*\times
M)\stackrel{\delta_1}{\leftarrow}K_G(S^1\times M,*\times M)\end{multline}
in equivariant $K$-theory.
Note that $K_G(M_n\times M,*\times M)$
is a torsion group which is a summand in
\begin{equation}\label{wherezis}K_G(M_n\times M)\cong  K_G(M_n\times M,*\times M)\oplus K_G(M)\ .\end{equation}
Further note that
$$K_G(S^1\times M,*\times M)\oplus K_G(M)\cong K_G(S^1\times M)\ .$$
We now consider,  with $t\in K_G^{tors}(M)$ chosen above and $\ori_{S^{1}}\in K^{1}(S^{1})\cong \Z$
the $K$-orientation of $S^{1}$,
$$\ori_{S^1}\times t\in K_G(S^1\times M,*\times M)\ .$$
Since
$$(f\times \id)^*(\ori_{S^1}\times t)=n\cdot \ori_{S^1}\times t=\ori_{S^1}\times nt=0$$
we can choose a class
\begin{equation}\label{eq:z}
z\in  K_G(M_n\times M,*\times M)
\end{equation}

such that
$(\pi\times \id)^*(z)=\ori_{S^1}\times t$.
Since $  K_G(M_n\times M,*\times M)$ is a torsion group, we can further find
an element
$\hat z\in U_G(M_n\times M)$ such that
$\beta(\hat z)=z$.
Since $\beta$ is natural we have
$$\beta\circ (\pi\times \id)^*(\hat z)=\ori_{S^1}\times t\ .$$
Furthermore, we know that $\beta$ intertwines
$\int^{U_G}$ and $\int^{K_{G}}$.
Therefore we have
$$\beta\circ \int^{U_G}_{[S^1\times M/G]/[M/G]}\circ (\pi\times \id)^*(\hat z)=t\ .$$
We define
$$\hat t:= \int^{U_G}_{[S^1\times M/G]/[M/G]}\circ (\pi\times \id)^*(\hat z)\in U_G(M)\ .$$
\textcolor{black}{Because $\im(\beta)=K_G^{tors}$ and $\im(\alpha)=\ker(\beta)$,
  if} we let $t$ run over all torsion classes in $K_G(M)$, then 
the set of corresponding $\hat t\in U_G(M)$ generates
$U_G(M)/\im(a)$.
Therefore, in order to show that the third square in (\ref{uwhfweffwfwf}) commutes, it suffices to show that
$\beta(\hat{t})=\delta(j(\hat{t}))$ for all these classes.

 Let us for the moment assume that the degree of $t$ has the opposite parity as $\dim(M)$.
We calculate, using \textcolor{black}{functoriality of integration, the
  projection formula, and Subsection \ref{sec:TrG_trivial_action}},
\begin{eqnarray}\label{www1122}
 \Tr_G\circ  \int^{U_G}_{[M/G]/[*/G]} \hat t&=& \Tr_G\circ \int^{U_G}_{[S^1\times M/G]/[*/G]} 
 (\pi\times \id)^*(\hat z)\\ 
&=&  \Tr_G\circ \int_{[S^1/G]/[*/G]}^{U_G}\circ \int_{[S^1\times
  M/G]/[S^1/G]}^{U_G} (\pi\times \id)^*(\hat
z)\nonumber\\
&=&\int_{ S^1}^{U}\circ \pi^*\circ  \Tr_G\circ \int_{ [M_n\times
  M]/[M_n/G]}^{U_G} \hat z\nonumber\\&=&
\int_{ S^1}^{U}\circ \pi^*\circ \beta^{-1}\circ \Tr_G\circ \beta\circ \int_{[M_n\times M/G]/[M_n/G]}^{U_G} \hat z\nonumber\\&=&
\int_{ S^1}^{U}\circ \pi^*\circ \beta^{-1}\circ \Tr_G\circ \int_{M_n\times M/M_n}^{K_G} \beta(\hat z)\nonumber\\
&=& \int_{ S^1}^{U}\circ \pi^*\circ \beta^{-1}\circ \Tr_G\circ \int_{M_n\times M/M_n}^{K_G} z\nonumber\\
&=&I_n\left(\Tr_G\circ \int_{M_n\times M/M_n}^{K_G} z\right)\ .\nonumber
\end{eqnarray}

We also know that $\im(\delta)$ is the torsion subgroup. Therefore we can find
$\tilde z\in K^\T_G(M_n\times M)$ such that $ \delta(\tilde z)=z$, where here
$\delta\colon K^{\T}_{G}(M_{n}\times M)\to K_{G}(M_{n}\times M)$ and we
consider $z\in K_{G}(M_{n}\times M)$ using
(\ref{wherezis}). \textcolor{black}{Since $\delta$ is a natural transformation
  we} have 
$$\delta\circ (\pi\times \id)^*(\tilde z)=\ori_{S^1}\times t\ .$$
Furthermore, we have
$$\delta\circ \int_{S^1\times M/M}^{K^{ \T}_G}\circ (\pi\times \id)^*(\tilde z)=t\ .$$
We define $$\tilde t:= \int_{S^1\times M/M}^{K^{\T}_G}\circ (\pi\times \id)^*(\tilde z)\ .$$
Then we have, \textcolor{black}{using the same rules as above},
\begin{eqnarray}\label{www112}
\Tr_G\circ\int^{ K^\T_G}_{M} \tilde t&=&\Tr_G\circ \int^{K^\T_G}_{S^1\times M} 
 (\pi\times \id)^*(\tilde  z)\\
&=&\Tr_G\circ \int_{S^1}^{K^\T_G}\circ \int_{S^1\times M/S^1}^{K^\T_G}
(\pi\times \id)^*(\tilde z)\nonumber\\
&=&\int_{S^1}^{K^\T}\circ \pi^*\circ\Tr_G\circ
\int_{M_n\times M/M_n}^{K^{\textcolor{black}{\T}}_{G}} \tilde  z\nonumber\\
&=& 
\int_{S^1}^{K^\T}\circ \pi^*\circ \delta^{-1}\circ
\Tr_G\circ \delta\circ \int_{M_n\times M/M_n}^{K^{\textcolor{black}{\T}}_G} \tilde  z\nonumber\\
&=&
\int_{S^1}^{K^\T}\circ \pi^*\circ \delta^{-1}\circ \Tr_G\circ \int_{M_n\times M/M_n}^{K_G} \delta(\tilde  z)\nonumber\\
&=& \int_{S^1}^{U}\circ \pi^*\circ \delta^{-1}\circ \Tr_G\circ \int_{M_n\times M/M_n}^{K_G} z\nonumber\\
&=&I_n\left(\Tr_G\circ \int_{M_n\times M/M_n}^{K_G} z\right)\ .\nonumber
\end{eqnarray}

 Let us now go back to consider $t$ of arbitrary parity.
We finally show that
$\delta \circ j(\hat t)=t$.  Because of the $K_G(M)$-module structure,
in the calculation above we can replace $t$ by
$t\cup \pr_M^*(P^{-1}(\xi))$ for  $\xi\in K^G(M)$.
Then
$\hat t$, $\tilde t$ and $z$  get replaced by 
$\hat t\cup \widehat{P^{-1}(\xi)}$,
$\tilde t\cup P^{-1}(\xi)$ and
$z\cup \pr_M^*(P^{-1}(\xi))$.
For all $\xi\in K^G(M)$
such that $\deg(\xi)+\deg(t)\equiv \dim(M)+1$ we
therefore have 
\begin{eqnarray*}
 i(j(\hat t))(\xi)&\stackrel{\textcolor{black}{\eqref{eq:1star}}}{=}&j_G(\hat
 t)(\xi)\\&\stackrel{}~\eqref{eq:22star}{=}&  \Tr_G\circ \int^{U_G}_{[M/G]/[*/G]} \hat t\cup
 \widehat{P^{-1}(\xi)}\\&\stackrel{(\ref{www1122})}{=}&
I_n( \Tr_G\circ \int^{K_G}_{M_n\times M/M_n}(z\cup \pr_M^*(P^{-1}(\xi))))\\&\stackrel{(\ref{www112})}{=}&
  \Tr_G\circ 
\int^{ K^\T_G}_{M} \tilde t\cup
P^{-1}(\xi)\\&\stackrel{\textcolor{black}{\eqref{eq:2star}}}{=}&
i(\tilde t)(\xi) \ .
\end{eqnarray*}
 Since the pairing $k^\T_G(M)\otimes K^G(M)\to \T$
is non-degenerate \textcolor{black}{and $i$ is an isomorphism}, $j(\hat t)=\tilde t$, and consequently
$\delta\circ j(\hat t)=\delta(\tilde t)=t=\beta(\hat t)$. This finishes the
proof of Lemma  \ref{udqwdqd}.
\hB 

\subsection{Non-degeneracy of the intersection pairing}

\subsubsection{}
In this subsection we introduce the notion of a differential $K$-orientation
of an orbifold $B$ (Definition  \ref{uiwqdqwdwqdqwd}) and construct
intersection pairings (Proposition \ref{uidqwdqwdqwdd})  
$$\hat K(B)\otimes \hat K(B)\to \T\ , \quad {\hat K_\R(B)\otimes \hat
  K_\R(B)\to \R/\Z}$$ for a compact differentially $K$-oriented orbifold $B$. 
The main result is Theorem \ref{trezreer} which states that the intersection pairing 
is non-degenerate.

\subsubsection{}

\newcommand{\av}{{\tt av}}
In the following, for a possibly inhomogeneous element $x\in \hat K([*/G])$
\textcolor{black}{we} let $x^1\in \hat K^1([*/G])$ denote \textcolor{black}{the
  component of degree $1$}.

As in \ref{zeuirerr}, we let $G$ be a finite group, $H\subseteq G$ be a normal subgroup, and we define $K:=G/H$. 
We assume that $N$ is a $G$-manifold \textcolor{black}{such that the action of $H$ is free},  and we define the $K$-manifold $M:=N/H$.
In addition we  assume that the locally trivial bundle of orbifolds $f^{G}\colon [N/G]\to [*/G]$ with fibre $N$ has a differential $K$-orientation. This differential $K$-orientation is given by certain data on $[N/G]$ (see \ref{pap201}) which in view of the equivalence $\pi\colon [N/G]\stackrel{\sim}{\to} [M/K]$ induces the data of an induced  differential $K$-orientation on  the locally trivial bundle of orbifolds $f^{K}\colon [M/K]\to [*/K]$ with fibre $M$. Hence the integration maps
$\hat f_!^G$ und $\hat f_!^K$  are defined.

\subsubsection{}

We define the average
$$\xymatrix{\C[G]^{G}\ar[d]^{\cong}\ar[r]^{\av^{H}}&C[K]^{K}\ar[d]^{\cong}\\\Omega(L[*/G])\ar[r]^{\av^{H}}& \Omega(L[*/K])}$$ 
over $H$-orbits by
$$\av^H(f)(Hg):=\frac{1}{|H|}\sum_{h\in H} f(hg)\ .$$
If $V$ is a complex representation of $G$ with character
$\chi_V$, then
$\av^H(\chi_V)$ is the character of the subspace of $H$-fixed points
$V^H\subseteq V$, considered as a representation of $K$. Therefore the left
square in
$$\xymatrix{R(G)\ar[d]^{\inv^H}\ar[r]&\Omega(L[*/G]) \ar[d]^{\av^H}\ar[r]&\hat K^1([*/G])\ar@{.>}[d]^{\av^H}\ar[r]&0\\
R(K)\ar[r]&\Omega(L[*/K])\ar[r]&\hat K^1([*/K])\ar[r]&0}$$
commutes, and this gives the dotted arrow which we also denote by $\av^{H}$.

\subsubsection{}

Recall that \textcolor{black}{by  \ref{zeuirerr}} we have an equivalence
$\pi\colon [N/G]\stackrel{\sim}{\to} [M/K]$  of orbifolds.
\begin{prop}\label{zwqduqwdqwd}
The diagram
\begin{equation}\label{zwqduqwdqwd1}
\xymatrix{\hat K([M/K])\ar[r]^{(\hat f^K_!\dots )^1}\ar[d]_{\cong}^{\pi^*}&\hat K^1([*/K])\ar[r]^(.6){\Tr_K}&\T \ar@{=}[d]\\
\hat K([N/G])\ar[r]^{(\hat f_!^G\dots)^1}&\hat K^1([*/G])\ar[u]^{\av^H}\ar[r]^(.6){\Tr_G}&\T}
\end{equation}
commutes.
\end{prop}
\proof
Since $\Tr_K$ and $\Tr_G$ are given as averages over $K$ and $G$, and the average in stages, first over $H$ and then over $K$, is equal to the average over $G$, we see that the 
right square commutes.

We now show that the left square commutes.
Consider $\hat x=[\cE,\rho]\in \hat K^{1}([M/K])$, where we actually think of $\cE$ as  a  $K$-equivariant geometric family over $M$. 
 According to  \eqref{eq300}, the class
$f^K_!(\hat x)$ is represented by
$$\left[f^K_!\cE\:\:,\:\:\int_{L[M/K]/L[*/K]} \left(\hA^c(o)\wedge \rho+ \sigma(o)\wedge R(\hat x)\right)+\tilde \Omega(1,\cE)\right]\ .$$
The pull-back $\pi^*\cE$ is a $G$-equivariant geometric family over $N$.
The class $f_!^G(\pi^* \hat x)$ is represented by
$$\left[f^G_!\pi^*\cE\:\:,\:\:\int_{L[N/G]/L[*/G]} L\pi^*\left(\hA^c(o)\wedge \rho+ \sigma(o)\wedge R(\hat x)\right)+\tilde \Omega(1,\pi^*\cE)\right]\ .$$

\subsubsection{}

We first show that the left square of (\ref{zwqduqwdqwd1}) commutes on classes of the form
$[\emptyset,\rho]$, i.e. we   show that
$$(\av^H\circ f^G_!\circ L\pi^*)(\rho)=f^K_!(\rho)\ .$$

To this end we make the isomorphism
$L\pi^{*}\colon \Omega(L[M/K])\cong \Omega(L[N/G])$ explicit.
First recall that
$$\Omega(L[M/K])\cong [\bigoplus_{k\in K} \Omega(M^k)]^K\ ,\quad 
\Omega(L[N/G])\cong [\bigoplus_{g\in G} \Omega(N^g)]^G\ .$$
 We write $\omega\in \Omega(L[N/G])$ in the form $\omega=\oplus_{g\in G}
 \omega_g$ with $\omega_g\in \Omega(N^g)$.

Let
$$\hat \pi\colon \bigsqcup_{g\in G}N^g\to \bigsqcup_{k\in K}M^k$$ be the
$G$-equivariant map induced by the projection $N\to M$. 

If $Hg\in K$ fixes an element \textcolor{black}{$nH\in M$ then $n\in N^{gh}$ for a suitable
$h\in H$. Indeed, $ ng=n h^{-1}$ for suitable $h\in H$.}

On the other hand, if $n\in N^{g}$, then $nH\in M^{Hg}$.
Indeed, $nH\cdot H g=nghH=nH$.
It follows that for $nH\in M^{Hg}$ we have
$$\hat \pi^{-1}(nH)=\bigsqcup_{h\in H} (nH\cap N^{gh})\ .$$
\textcolor{black}{Assume that $n\in N^{g}$ and $n\tilde h\in N^{g}$.
Then $ng=n$ and
$n\tilde h g=n\tilde h =ng\tilde h$, hence
$n\tilde h=ng\tilde h g^{-1}$. Since $g\tilde hg^{-1}\in H$ and $H$ acts
freely this implies that 
$\tilde h\in H_{g}$.} Vice versa, if $\tilde h\in H_{g}$ then
with $n\in N^{g}$ we have also $n\tilde h\in N^{g}$.
 We conclude that for $n\in N^{g}$ we have
$nH\cap N^{g}=nH_{g} $, so that
 $$|nH\cap N^{g}|=\left\{\begin{array}{cc}|H_{g}|;&|Hn\cap
     N^{g}|\not=0\\0;&else\end{array}\right.\ .$$
Therefore $N^g\to M^{H_g}$ is a $|H_g|$-fold covering.
Moreover, if $nH\in M^{Hg}$, then
\begin{equation}\label{wiuqiuqjkdqddqwd}|H|=|Hn|=\sum_{h\in
    H,|nH\cap N^{gh}|\not=0}|H_{gh}|\ .\end{equation} 
%
%


We consider $g\in G$ such that $N^g\not=\emptyset$.
Note that $\hat \pi(N^g)\subseteq M^{gH}$ is an open and closed submanifold.
If $\omega\in \Omega(L[M/K])$, then
\begin{eqnarray*}
f^G_!(L\pi^*\omega)(g)&=&\int_{N^g} \hat \pi_{Hg}^*\omega_{Hg}\\
&=&|H_{g}| \int_{\hat \pi(N^{g})} \omega_{Hg|_{\hat \pi(N^g)}}
\end{eqnarray*}
All together
\begin{eqnarray*}
\av^Hf^G_!(L\pi^*\omega)(Hg)&=&
\frac{1}{|H|}\sum_{h\in H }f^G_!(L\pi^*\omega)(gh)\\&=&
\frac{1}{|H|}\sum_{h\in H}|H_{gh}| \int_{\hat \pi(N^{gh})} \omega_{Hg|_{\hat \pi(N^{gh})}}\\
&\stackrel{(\ref{wiuqiuqjkdqddqwd})}{=}& \int_{M^{Hg}} \omega_{Hg}\\
&=&f^K_!(\omega)(Hg)\ .
\end{eqnarray*}
This calculation shows that the left square in (\ref{zwqduqwdqwd}) commutes on elements
of the form $[\emptyset,\rho]$.
\subsubsection{} 
We now consider a geometric family
$\cE$ over $M$. Note that
$\tilde \Omega(\cE,1)=f^K_!(\alpha)$ for some $\alpha\in \Omega(L[M/K])$. It follows from
the locality of $\alpha$ that
$\tilde \Omega(\pi^*\cE,1)=f^G_!(\pi^*\alpha)$.
Hence
$\av^H(\tilde \Omega(\pi^*\cE,1))=\tilde \Omega(\cE,1)$.

We continue with classes of the form
$[\cE,0]$.  As $K^1([*/K])=0$, and as we only consider odd
  classes, we can choose, after stabilization, a $K$-invariant taming $(f^K_!\cE)_t$.
It lifts to a $G$-invariant taming $(f^G_!\pi^*\cE)_t$.
Note that
$$[f^K_!\cE,0]=[\emptyset,-\eta((f^K_!\cE)_t)]\ ,\quad \quad [f^G_!\pi^*\cE,0]=[\emptyset,-\eta((f^G_!\pi^*\cE)_t)]
\ .$$ 
\textcolor{black}{Therefore, w}e must show that
$$\av^H (\eta((f^G_!\pi^*\cE)_t))=\eta((f^K_!\cE)_t)\ .$$
To this end we write out the definition (\ref{udqwidwqdodiwodopop}) of the
eta-invariant. We have
$$\eta((f^G_!\pi^*\cE)_t)(g)=\frac{-1}{\pi}\:\int_0^\infty \Tr \:g \:\partial_t A_\tau \ee^{A_\tau^2}d\tau\ ,$$
where $A_\tau:=A_\tau((f^G_!\pi^*\cE)_t)$ is the family of rescaled tamed
Dirac operators on the $G$-Hilbert space $H(f^G_!\pi^*\cE)$.  The important observation is now that
$H(f^K_!\cE)$ can naturally be identified with the subspace of $H$-invariants
$H(f^G_!\pi^*\cE)^H$, and the restriction of $A_\tau$ to this subspace is $A_\tau((f^K_!\cE)_t)$.
Note that $\frac{1}{|H|}\sum_{h\in H} h$ acts as the projection onto the subspace of $H$-invariants. Therefore
$$\av^H(\eta((f^G_!\pi^*\cE)_t))(Hg)=\frac{1}{|H|}\sum_{h\in H}\eta((f^G_!\pi^*\cE)_t)(hg)=\eta((f^K_!\cE)_t)(Hg)\ .$$
Alltogether we thus have shown that
$$\av^H[f^G_!\pi^*\cE,0]=[f^K_!\cE,0]\ .$$
This finishes the proof of Proposition \ref{zwqduqwdqwd}.
\hB

\subsubsection{}

Let $B$ be an orbifold which admits a presentation
$B\cong [M/G]$  for a finite group $G$.  We further assume that  the map
$[M/G]\to[*/G]$ is differentiably $K$-oriented.
\begin{prop}\label{uzqwgduqwdwqd}
If $B\cong [M^\prime/G^\prime]$ is another presentation of $B$ with a finite group $G^\prime$, then
$[M^\prime/G^\prime]\to [*/G^\prime]$ has an induced differential $K$-orientation. This correspondence preserves reality of differential $K$-orientations.
\end{prop}
\proof
We use the method and notation of the proof of Theorem \ref{wudiwqd}.
The differential $K$-orientation of $[M/G]\to [*/G]$ is given by $G$-invariant data on $M$, see \ref{pap201}.
It lifts to $G\times G^\prime$-equivariant data on $N$, and finally induces the $G^\prime$-equivariant data on $M^\prime$ which gives the induced orientation of $[M^\prime/G^\prime]\to [*/G^\prime]$.
This correspondence respects the equivalence relation between representatives
of differential $K$-orientations \textcolor{black}{and reality}.
\hB

In view of Proposition \ref{uzqwgduqwdwqd}
we can talk about a differential $K$-orientation of an orbifold
which admits a presentation $[M/G]$ with a finite group $G$.
\begin{ddd}\label{uiwqdqwdwqdqwd}
 Assume that $B\cong [M/G]$ is an orbifold presented
with a finite group $G$. A differential $K$-orientation $o$ of an orbifold $B$
is represented by a differential $K$-orientation of the map
$[M/G]\to [*/G]$.

If $o^\prime$ is a differential $K$-orientation  represented by
$[M^\prime/G^\prime]\to [*/G^\prime]$, where $B\cong [M^\prime/G^\prime]$ is a
presentation of $B$ for a another finite group $G^\prime$, then $o^\prime=o$
if $o^\prime$ is equal to the differential $K$-orientation induced on
$[M^\prime/G^\prime]\to [*/G^\prime]$ by $o$  according to Proposition
\ref{uzqwgduqwdwqd}.  The differential $K$-orientation of $B$ is called real if it
is represented by a real differential $K$-orientation of $[M/G]\to [*/G]$.
\end{ddd}

 Note that we only define the concept of a
 differential $K$-orientation of an orbifold if the latter admits a presentation as a quotient of a closed manifold by a finite group.

\begin{prop}\label{uidqwdqwdqwdd} 
We consider an orbifold which admits a presentation
 $B\cong [M/G]$ for a compact manifold $M$ and a finite group $G$, and  which is equipped with   a differential
$K$-orientation (represented by a differential $K$-orientation of
$[M/G]\to [*/G]$). The pairing
$$\hat K(B)\otimes \hat K(B)\stackrel{\cup}{\to} \hat K(B)\cong \hat K([M/G])\xrightarrow{\Tr_G\circ ( \int_{[M/G]/[*/G]}\dots)^1} \T$$
is well-defined  independent of the choice of the representative of the differential $K$-orientation.
If the orientation of $B$ is real, then by restriction we get a well-defined
pairing 
$$\hat K_\R(B)\otimes \hat K_\R(B)\to \R/\Z\ .$$
\end{prop}
\proof
We again use the technique of the proof of Theorem \ref{wudiwqd}.
If $B\cong [M/K]$ and $B\cong [M^\prime/K^\prime]$ are two presentations, then there is a third presentation
$B\cong [N/G]$ such that $K,K^\prime\subset G$ are normal subgroups and
$M\cong N/K^\prime$ and $M^\prime\cong N/K$.
We now use Proposition \ref{zwqduqwdqwd} which gives 
$$\Tr_Kf^K_!(x\cup y)=\Tr_G(f^G_!(\pi^*(x\cup y)))=
\Tr_{K^\prime} f^{K^\prime}_!(x^\prime\cup y^\prime)\ ,$$
where $x,y\in \hat K([M/K])$ and $x^\prime,y^\prime\in \hat K([M^\prime/K^\prime])$ are such that
$\pi^*x=\pr^{\prime *}x^\prime$ and $\pi^*y=\pr^{\prime *}y^\prime$.
\hB

\subsubsection{}

\begin{theorem}\label{trezreer} 
Let $B$  be an orbifold with a differential $K$-orientation. The intersection pairing
$$\hat K(B)\otimes \hat K(B)\xrightarrow{(.,.)} \T $$
is non-degenerate. If the orientation of $B$ is real (see \ref{udiduwqdqwdwqoidwiqdwd}) then the restriction
$$\hat K_\R(B)\otimes \hat K_\R(B)\xrightarrow{(.,.)} \R/\Z$$ is
non-degenerate.

\end{theorem}
\proof
We can apply the argument of the proof of \cite[Proposition B6]{freed-2007-322} using the fact that
\begin{eqnarray*}
 U_G(M)\otimes K_G(M)\stackrel{\cup}{\to} U_G(M)\xrightarrow{\Tr_G\circ (\int_{[M/G]/[*/G]}\dots)^1} \T\\{U^\R_G(M)\otimes K_G(M)\stackrel{\cup}{\to} U^\R_G(M)\xrightarrow{\Tr_G\circ (\int_{[M/G]/[*/G]}\dots)^1} \R/\Z}
\end{eqnarray*}
are non-degenerate pairings by Theorem \ref{uiudqdqwdwq} and Corollary \ref{barcel}. \hB

\section{Examples}\label{sec5}

\subsection{The differential $K$-theory class of a mapping torus}

\subsubsection{}\label{dwuduidwhdiudhuihwdu}

Let $G$ be a finite group. We consider a geometric $\Z/2\Z$-graded $G$-bundle 
 $\bV:=(V,h^V,\nabla^V,z)$ over $S^1$, where we let $G$ act trivially on $S^1$. 
Let $1\in S^1$ be the base point. The group $G$ acts on the fibres $V_1^{\pm}$ of the homogeneous components of $V$.
We assume  that $V_1^+\cong V_1^-$ as representations of $G$. Let $\cV$ denote
the corresponding $G$-equivariant geometric family over $S^1$. Equivalently, we can consider the family $[\cV/G]$ over $[S^1/G]$. 

By Proposition \ref{prop1} we have an exact sequence
$$K^1([S^1/G])\stackrel{\ch}{\to}\Omega^1(L[S^1/G])/\im(d)\stackrel{a}{\to}\hat  K^0([S^1/G])\stackrel{I}{\to} K^0( [S^1/G])\to 0\ .$$ 
We identify, \textcolor{black}{as in Subsection
  \ref{sec:trivial_action}}, $$\Omega^1(L[S^1/G])/\im(d)\cong R(G)\otimes
  \Omega^1(S^1)/\im(d)\cong R(G)\otimes \C$$ and 
$$(\Omega^1(L[S^1/G])/\im(d))/\ch(K^1([S^1/G]))\cong R(G)\otimes \T\ .$$
The class
$[\cV,0]\in \hat K^0(S^1)$ satisfies $I([\cV,0])=0$ and hence corresponds to an element
of $R(G)\otimes \T$. This element is calculated in the following lemma.

For $g\in G$ we decompose
$V^\pm=\bigoplus_{\theta\in U(1)} V^\pm(\theta)$ according to eigenvalues of
\textcolor{black}{the action of} 
$g$. We set \textcolor{black}{$n^\pm_{\theta}:=\dim(V^\pm(\theta))$} and
let $\phi^\pm(\theta)\in U(\textcolor{black}{n^\pm_\theta})/conj$ denote the holonomies
of $V^\pm(\theta)$ (well defined modulo conjugation in the group $U(\textcolor{black}{n^\pm_\theta})$).
 \begin{lem}\label{dersa1}
We have
$[\cV,0]=a(\Phi)$, where $\Phi\in \Omega^1(L[S^1/G])/\im(d)\cong \C[G]^G$ is given by 
$$\Phi(g)=\frac{1}{2\pi i}\sum_{\theta\in U(1)}\theta \:\log\frac{\det(\phi^+(\theta))}{\det(\phi^-(\theta))}\ .$$
\end{lem}
\begin{proof}
We consider the map $q\colon [S^1/G]\to [*/G]$ with the canonical $K$-orientation given by the bounding $Spin$-structure of $S^1$.
By Proposition \ref{mainprop} we have a commutative diagram
$$
\begin{CD}
  R(G)\otimes \C @>{\sim}>> \Omega^1(L[S^1/G])/(\im(d)+\im(\ch)) @>a>> \hat K^0([S^1/G])\\
  @VV{=}V @VV{q^o_!}V @VV{\hat q_!}V\\
   R(G)\otimes \C @>{\sim}>> \Omega^0(L[*/G])/\im(\ch) @>a>> \hat K^1([*/G])
\end{CD}
\ .$$
In order to determine $[\cV,0]$ it therefore suffices to calculate
$\hat q_!([\cV,0])$.
Now observe that $q\colon S^1\to *$ is the boundary of $p\colon D^2\to *$.
Since the underlying topological $K$-orientation of $q$ is given by the bounding $Spin$-structure
we can choose a differential $K$-orientation of $p$ with product structure which restricts to the differential $K$-orientation of $q$. 
The bundle
$\bV$ is topologically trivial. Therefore we can find a geometric $G$-bundle
$\bW=(W,h^W,\nabla^W,z)$, again with product structure, on $D^2$ which restricts to $\bV$ on the boundary.
Let $\cW$ denote the corresponding geometric family over
$D^2$. Later we prove the bordism formula Proposition \ref{bordin}. It gives
$$\hat q_!([\cV,0])=[\emptyset,p_!R([\cW,0])]=-a\left(\int_{L[D^2/G]/L[*/G]} \Omega^2(\cW)\right)\ .$$
For $g\in G$ we have
\begin{eqnarray*}
\Omega^2(\cW)(g)&=&\frac{1}{2\pi i}\ch_2(\nabla^W)(g)\\&=&\frac{1}{2\pi i}\left(\ch_2(\nabla^{\det(W^+)})(g)-\ch^G_2(\nabla^{\det(W^-)})(g)\right)\\&=&\frac{-1}{2\pi i}\left[ \Tr g R^{\nabla^{W^+}}- \Tr g
 R^{\nabla^{W^-}}\right]\\&=&
\frac{-1}{2\pi i}\sum_{\theta} \theta [R^{\nabla^{\det W^{+}(\theta)}}-R^{\nabla^{\det W^-(\theta)}}]
\ .
\end{eqnarray*}

The holonomy $\det(\phi^\pm(\theta))\in U(1)$ of $\det(\bV^\pm(\theta))$ is equal to the integral of the curvature of $\det \bW^\pm(\theta)$:
$$\log\det(\phi^\pm)=\int_{D^2} R^{\nabla^{\det(W^\pm)}}\ .$$
It follows that
$\hat q_!([\cV,0])=a(\Phi)$
with $$\Phi(g)=\frac{1}{2\pi i}\sum_{\theta\in U(1)}\theta\log\frac{\det
    (\phi^+(\theta))}{\det(\phi^-(\theta))} \ .
$$ \hB
\end{proof}

\subsubsection{}

Consider a finite group $G$ and let
 $\cE$ be a \textcolor{black}{$G$-}equivariant  geometric family over a point. We consider an additional automorphism $\phi$
of $\cE$ which commutes with the action of $G$. Then we can form the mapping torus
$T(\cE,\phi):=(\R\times\cE)/\Z$, where $n\in \Z$ acts on
$\R$ 
by $x\mapsto x+n$, and by $\phi^n$ on $\cE$. 
The product $\R\times \cE$ is a $G\times \Z$-equivariant geometric family
over $\R$ (the pull-back of $\cE$ by the projection $\R\to *$).
The geometric structures descend to the quotient by $\Z$ and 
turn the  mapping torus $T(\cE,\phi)$ into a geometric family
over $[S^1/G]= [(\R/\Z)/G]$, where $G$ acts trivially on $S^1$.  In the present subsection we study the class
$$[T(\cE,\phi),0]\in \hat K([S^1/G])\ .$$
In the following we will assume that
the parity of $\cE$ is even, and that $\ind(\cE)=0$.

Let $\dim\colon K^0([S^1/G])\to R(G)$ be the dimension homomorphism, which in this case is an isomorphism.
Since $\dim I([T(\cE,\phi),0])=\dim(\ind(\cE))=0$ we have in fact
$$[T(\cE,\phi),0]\in \im(a)\cong (\Omega^1(L[S^1/G])/\im(d))/\ch(K^1([S^1/G]))\cong R(G)\otimes \T\ ,$$
 as in \ref{dwuduidwhdiudhuihwdu}.

 Set $V:=\ker(D(\cE))$. This graded $G$-vector space is preserved by the action of $\phi$. We use the same symbol $\phi$ in order to denote the induced action on $V$.

 We form the zero-dimensional family
$\cV:=(\R\times V)/\Z$ over $[S^1/G]$. This bundle is isomorphic to the
kernel bundle of $T(\cE,\phi)$. The bundle of Hilbert spaces of the family
$T(\cE,\phi)\cup_{[S^1/G]}\cV^{op}$ has a canonical subbundle of the form $\cV\oplus \cV^{op}$.
We choose the taming $(T(\cE,\phi)\cup_{[S^1/G]}\cV^{op})_t$
which is induced by the isomorphism
$$\left(\begin{array}{cc}0&1\\1&0\end{array}\right)$$ on this subbundle.
Note that
$[T(\cE,\phi),0]=[\cV,\eta^1((T(\cE,\phi)\cup_{[S^1/G]}\cV^{op})_t)].$
Since $$(T(\cE,\phi)\cup_{[S^1/G]}\cV^{op})_t$$ lifts to a product  under  the pull-back
$\R\to \R/\Z$ we see that $\eta^1((T(\cE,\phi)\cup_{[S^1/G]}\cV^{op})_t)=0$.

It follows that 
$[T(\cE,\phi),0]=[\cV,0]\in R(G)\otimes \T$.
This class has been calculated in terms of the action of $\phi$ on $V$ in Lemma \ref{dersa1}.

\subsection{Bordism}



\subsubsection{}

A zero bordism of a geometric family $\cE$ over an orbifold $B$ is a
geometric family $\cW$ over $B$ with boundary such that $\cE=\partial
\cW$. The notion of a geometric family with boundary was discussed in detail in
\cite[\textcolor{black}{Section 2}]{math.DG/0201112}. Note that the boundary
here is fibrewise so that the stackness of $B$ does not introduce new
problems.

\begin{prop}\label{kop1}
If $\cE$ admits a zero bordism $\cW$, then
in $\hat K^*(B)$ we have the identity
\begin{equation}\label{E_to_Omega}[\cE,0]=[\emptyset,\Omega(\cW)] .\end{equation}
\end{prop} 
\proof
Since $\cE$ admits a zero bordism we have $\ind(\cE)=0$.
In order to see this choose a presentation $B\cong [M/G]$. Then
$M\times_B\cE$ is a $G$-equivariant geometric family which admits a $G$-equivariant zero bordism $M\times_B\cW$. 
By the equivariant bordism invariance of the index it follows that
$\ind(M\times_B\cE)\in K_G(M)$ vanishes. This implies that
$\ind(\cE)=0$ in $K(B)$.

It follows from Lemma \ref{uiufwefewfwfewf} that after replacing 
$\cE$ by $\cE\sqcup_B \tilde \cE\sqcup_B \tilde \cE^{op}$ and $\cW$ by $\cW\sqcup_B  (\cE\times I)$ for a suitable geometric  family $\tilde \cE$ there
exists a taming $\cE_t$. This taming induces a boundary taming 
$\cW_{bt}$. The obstruction  to an extension of the boundary taming to a taming
of $\cW$ is $\ind(\cW_{bt})\in K(B)$. 
Using the method described in \ref{dhuidwqdwqd}
we can adjust the taming
$\cE_t$ such that $\ind(\cW_{bt})=0$. Here it might be necessary to add another family to $\tilde \cE$. Then we
extend the boundary taming $\cW_{bt}$ to
a taming $\cW_t$, possibly after a further stabilization, i.e. after adding
 a family $\cG\sqcup_B \cG^{op}$ with closed fibres.

 We now apply
\begin{theorem}\label{formula_bord}
$$\Omega(\cW)=d\eta(\cW_t)-\eta(\cE_t)\ .$$
\end{theorem}
 To prove Theorem \ref{formula_bord}, we adapt the proof of theorem \cite[Theorem 4.13]{math.DG/0201112}
using the remarks made in the proof of Theorem \ref{diqwdwqdddwqd}. 
We see that $(\cE,0)$ is paired with $(\emptyset,\Omega(\cW))$. This  implies \eqref{E_to_Omega}.
\hB

\subsubsection{}

Let $p\colon W\to B$ \textcolor{black}{be a} representable morphism which is a locally trivial fibre bundle of compact manifolds with boundaries. We let
 $q:=(p_{|\partial W})\colon (V:=\partial W)\to B$ denote the locally trivial
 bundle of closed manifolds obtained by restriction of $p$ to the fibrewise
 boundaries. We assume that $p$ has a topological $K$-orientation and a
differential $K$-orientation represented by $o_p$ which refines the topological $K$-orientation. 
We assume that the geometric data of $o_p$ have a product structure near $V$.
In this case we have a restriction $o_q:=o_p{|V}$ which represents a differential $K$-orientation
of $q$. It is easy to see that this restriction of representatives (with product structure)
preserves equivalence and gives a well-defined restriction of differential $K$-orientations.
We have the following version of bordism invariance of the push-forward in differential $K$-theory.
\begin{prop}\label{bordin} For $y\in\hat K(W)$ we set $x:=y_{|V}\in \hat K(V)$. Then 
we have
$$\hat q_!(x)=[\emptyset\:,\:  p^o_!R(y)]\ . $$
\end{prop}
\proof
The proof can be \textcolor{black}{literally} copied from
\cite[5.18]{bunke-2007}. \hB

\subsection{The intersection pairing for $[\C\P^1/(\Z/k\Z)]$}

\subsubsection{}\label{uiwreuifweuifewfewf}

We fix a number  $k\in \nat $ and consider the finite group $\Gamma:=\Z/k\Z$.  We furthermore fix a primitive \textcolor{black}{$k^{th}$} root of unity $\xi$ and let
$\Gamma$ act on $\C^2$ by $[n](z_0,z_1)=(\xi^nz_0,z_1)$. This induces an action of $\Gamma$ on $\C\P^1$. Let $X:=[\C\P^1/\Gamma]$ be the corresponding orbifold.

We cover $\C\P^1$ by the standard charts $U:=\{[u:1]\:|\:u\in \C\}$ and $V:=\{[1:v]|v\in \C\}$. The transition is given by $v=\frac{1}{u}$.
Therefore $\Gamma$ acts on $U$ by
$[n] u:=\xi^n u$, and on $V$ by $[ n]v=\xi^{-n}v$.
\subsubsection{}
We calculate
$K(X)\cong K_\Gamma(\C\P^1)$ using the Mayer-Vietoris sequence
associated to the covering $U\cup V$. These spaces are equivariantly
 homotopy equivalent to points.
Therefore  we have isomorphisms of rings $K_\Gamma(U)\cong K_\Gamma(V)\cong R(\Gamma)\cong \Z[\Z/k\Z]$.
The latter is the free $\Z$-module generated by the classes $[l]$, $l\in 0,\dots,k-1$, where $[l]$ is the representation of $\Z/k  \Z$ on $\C$ which sends 
$[1]$ to $\xi^l$.
Furthermore,  we have an equivariant homotopy equivalence
$U\cap V
\cong \C^*$ with a free $\Gamma$-action. Note that
$\C^*/\Gamma\cong \C^*$.
 We therefore have
$$K^i_{\Gamma}(\C^*)\cong \Z,\quad i=0,1\ .$$
The Mayer-Vietoris sequence reads
$$\xymatrix{K^0(X)\ar[r]^\beta&R(\Gamma)\oplus R(\Gamma)\ar[r]^(.6)\alpha\ar@{.>}@/_0.5cm/[l]_\sigma&\Z\ar[d]\\\Z\ar[u]^\delta&0\ar[l]&K^1(X)\ar[l]}.$$
The map $\alpha$ maps a pair of representations $(\chi,\mu)$ of $\gamma$ to the difference of their dimensions. In particular, it is surjective. Therefore
$K^1_\Gamma(X)\cong 0$.

 The map $\delta$ maps the integer
$1\in \Z$ to the class represented by the difference
$L-1$, where $1\cong \C\P^1\times \C$ with the trivial action of $\Gamma$ on the fibres, and
$L$ is the bundle obtained from $U\times \C$ and $V\times \C$, again
with trivial fibrewise action,  glued with $(u,z)\mapsto
(u^{-1},u^kz)$. {In order to see this, one can use the factorization
  through the boundary map of the Mayer-Vietoris sequence  for $
\C \P^1\setminus\{0,\infty\}$ with corresponding decomposition (and for K-theory
with compact supports). The main point is that the action is free here, so that we can pass to
the quotient with the projection map, where everything is known.}

 We now define a split $\sigma$ as follows.  Let $l,h\in\Z$ with
corresponding representations $([l],[h])\in R(\Gamma)\oplus R(\Gamma)$.
Then $\alpha([l],[h])=0$.
We define equivariant trivial bundles
$L_U:=U\times \C$ and $L_V:=V\times \C$, where the actions on the fibres are given by
$[l]$ and $[-h]$, respectively. 
Then  we can glue the trivial bundles equivariantly
using the transition function $\C^*\times \C \ni (u,z)\mapsto (u^{-1},u^{-h-l}z)$. The result is
$L_{l,h}:=\sigma([l],[h])$.

 Note that, by construction, as equivariant bundles
\begin{equation}\label{tensorformula} L_{l,h}\otimes L_{l',h'}\cong
L_{l+l',h+h'},\qquad L_{l,h}^*\cong L_{-l,-h}\end{equation}
Moreover, the bundle $L$ from above is precisely $L\cong L_{0,-k}$.

 Using a basis of $\ker(\alpha)$ consisting of
elements of the form $[l],[h]$ and the resulting linear split of $\beta$ and $\delta$ we get a decomposition
$$K^0(X)\cong \Z\oplus \ker(\alpha)\ .$$
\subsubsection{}
The manifold $\C\P^1$ has an equivariant complex structure. It gives
an equivariant $Spin^c$-structure and therefore an equivariant $K$-orientation.
In the following we calculate
$$\int_{[\C\P^1/\Gamma]}\colon K(X)\to R(\Gamma)\ .$$ The calculation is based 
 on  the explicit knowledge of the kernel and  cokernel  of the
$Spin^c$-Dirac operator twisted by suitable representatives of elements of $K(X)$.
In fact, the $Spin^c$-Dirac operator is the Dolbeault operator $D$.
 Therefore for a holomorphic bundle $E\to \C\P^1$ 
$$\ker(D^+\otimes E)\cong H^0(\C\P^1,E)\ ,\quad \coker(D^+)\cong H^1(\C\P^1,E)\cong H^0(\C\P^1,K\otimes E^*)^*\ ,$$
where $K$ denotes the canonical bundle. Observe that $K\cong
L_{-1,-1}$, using that the constant $-1$ showing up in the usual
transition functions is homotopic to $1$ in $\C^*$.

 


We now consider the case $E=L_{l,h} $ with $h,l\in \Z$.
The holomorphic sections of $L_{l,h}$ over $U$ (viewed as functions in the
trivialization fixed above) have a
basis of the form 
$u\mapsto u^s$ with $s\ge 0$. They are transformed to
$v\mapsto v^{-s+l+h}$ on $V$. These sections are holomorphic if
$0\le s\le l+h$.

The section $u^s$ is mapped by the generator of $\Gamma$ to
$\xi^{l-s} u^s$, i.e.~$\Gamma$ acts
by multiplication with $\xi^{l-s}$. 
Consequently, as $\Gamma$-representation we get
$$H^0(\C\P^1,L_{l,h})\cong\bigoplus_{s=0}^{l+h}  [l-s]\ .$$

 The holomorphic sections on $U$ of $K\otimes {L_{l,k}^*}$ are given by
$u^s du$ with $s\ge 0$. They are transformed to
$-v^{-s-2-l-h}dv$ on $V$. For holomorphy we hence need
$0\le s \le -l-h-2$.

 We see that there is no cancelation between kernels and cokernels.
As representations of $\Gamma$ we have, using that $K\otimes
L_{l,h}^*\cong L_{-l-1,-h-1}$ and that we have to look at the dual of
the space of holomorphic sections,
$$H^1(\C\P^1,L_{l,h})\cong \bigoplus_{s=0}^{-l-h-2}{[l+s+1]}\ .$$

\subsubsection{}
 For an explicit example, let us take $k=2$.
A basis of the $\Z$-module $K^0(X)\cong \Z^4$ is given by
 
$$(e_i)_{i=1}^4:=(1=L_{0,0},L_{0,-2},L_{-1,0},L_{0,-1})\ .$$
The matrix  of the intersection pairing
$$A_{i,j}:=( e_i,e_j)$$ is given by
{$$\begin{pmatrix}
1&0&0&0\\0&-1&-1&-1\\0&-1&0&-1\\0&-1&-1&0
  \end{pmatrix}
$$ which has determinant $-1$.
This illustrates} that the pairings
$$\xymatrix{ K^0(X)_\C\otimes K^0(X)\ar[r]^(.6){(.,.)}\ar[d]&\C\ar[d]\\
 K^0(X)_\C/K^0(X)\otimes K^0(X)\ar[r]&\T}$$
  are non-degenerate. 
We  have isomorphisms
$$ \Omega^0(LX)/\im(\ch)\cong \hat K^1(X)\ ,  \quad U^1(X)\cong H^0(LX)/\im(\ch)\cong K^0(X)_\C/K^0(X) ,\quad U^0(X)\cong 0$$
and an exact sequence
$$0\to \Omega^1(LX)/\im(d)\to \hat K^0(X)\to K^0(X)\to 0\ .$$
The pairing $ \hat K(X)\otimes  \hat K(X)\to \C/\Z$ is non-degenerate, as   we already know by Theorem \ref{trezreer}. In order to see this explicitly, assume that
$\hat x\in  \hat K^1(X)$. If it pairs trivially with the subgroup
$\Omega^1(LX)/\im(d)$, then we conclude that $\hat x\in U^1(X)\cong K^0(X)_\C/K^0(X)$.
The pairing of $\hat x$ with $\hat K^0(X)$ now factors over $K^0(X)$. We can  conclude from the topological result that $\hat x=0$.\newline
Similarly, if $\hat x\in \hat K^0(X)$ pairs trivially with $\hat K^1(X)$, then we
 conclude that $\hat x$ is given by a closed form of odd degree which is necessarily exact.
This again implies that $\hat x=0$.

\textcolor{black}{\section{Open questions}}

We list a number of questions left open which would be
  interesting to clarify. Moreover, at some points we left out more than just
  a few details where one might wish for a complete treatment.
  \begin{enumerate}
  \item Because we wanted to use our calculus of push-forward of orbifolds, we
    only defined the non-degenerate intersection pairing for global quotients
    by a finite group action. Indeed, one would not know what should replace
    $K([*/G])$ if, instead of $[M/G]$ one considers a general
    orbifold. However, we expect that the composition of the push-forward with
    $\Tr_G$, which is the object which is independent of the presentation, can
    be defined in general, at least for presentable orbifolds. From this, one
    should then get the non-degenerate pairing in general.
  \item Non-equivariant differential K-theory satisfies a strong uniqueness
    property \cite{bunke-2009} which can be used to automatically identify its
    many different models. Because the underlying homotopy theory for
    equivariant K-theory shares the basic relevant features, we expect that a
    similar uniqueness theorem can be established for orbifold differential
    K-theory, and probably for other interesting differential extensions of
    orbifold cohomology theories, as well. In particular, this would
    automatically give an identification of our theory with the one of Ortiz
    \cite{ortiz}. Alternatively, it would also be intersting to compare the
    two constructions directly.
  \item In this paper, we concentrate entirely on compact orbifolds. However,
    for many purposes, a compactly supported theory for non-compact orbifolds
    is convenient or neccessary. Secondly, a version for pairs is
    desireable. It should not be too hard to work out the details and
    relations of such a theory, but will certainly require care of some new
    technical details.
  \item We have used a couple of generalizations of local index theory which
    are not trivial and it would be desireable to work out a presentation of the
    details. In particular this applies to the details of the proof of
    Theorem \ref{diqwdwqdddwqd}  and of the adiabatic limit formula for
    eta-forms of Theorem \ref{adia1}.
  \end{enumerate}

\end{document}